\documentclass[11pt]{article}
 \usepackage[utf8]{inputenc}
 \usepackage[T1]{fontenc}  
\usepackage[english]{babel}
\usepackage{amsfonts}

\usepackage{amsmath,amsxtra,amsthm,latexsym,mathabx}
\usepackage{amssymb} % \Subset !
\usepackage{bbm} % \one
\usepackage{accents}
\usepackage{stmaryrd,trimclip} % short arrow

\usepackage[toc,page]{appendix}

\usepackage{hyperref}

\newtheorem{thm}{Theorem}[section]
 \newtheorem{cor}[thm]{Corollary}
 \newtheorem{lem}[thm]{Lemma}
 \newtheorem{prop}[thm]{Proposition}

 \theoremstyle{definition}
 \newtheorem{defn}{Definition}[section]
 
 \theoremstyle{remark}
 \newtheorem{rem}[thm]{Remark} 

 \newtheorem{ex}[thm]{Example}
 
 \numberwithin{equation}{section}

\renewcommand{\>}{\rangle}

\newcommand{\al}{\alpha}

\newcommand{\vphi}{\varphi}
\newcommand{\la}{\lambda}  
\newcommand{\de}{\delta}
\newcommand{\vep}{\varepsilon}
\newcommand{\ep}{\epsilon}
\newcommand{\ga}{\gamma}
\newcommand{\Ga}{\Gamma}

\newcommand{\Om}{\Omega}
\newcommand{\Si}{\Sigma}
\newcommand{\si}{\sigma}
\newcommand{\De}{\Delta}
\newcommand{\vka}{\varkappa}
\newcommand{\vth}{\vartheta}
\newcommand{\om}{\omega}

 \newcommand{\CC}{\mathbb{C}} 
 \newcommand{\EE}{\mathbb{E}} 
 \newcommand{\GG}{\mathbb{G}} 
\newcommand{\HH}{\mathbb{H}}  
 
\newcommand{\LL}{\mathbb{L}} \newcommand{\MM}{\mathbb{M}} 
\newcommand{\NN}{\mathbb{N}} 
\newcommand{\PP}{\mathbb{P}}

\newcommand{\RR}{\mathbb{R}}
\usepackage[bbgreekl]{mathbbol}
\DeclareSymbolFontAlphabet{\mathbbm}{bbold}
\DeclareSymbolFontAlphabet{\mathbb}{AMSb}%
\DeclareMathAlphabet{\mathmybb}{U}{bbold}{m}{n}
\newcommand{\ab}{\alpha}

\newcommand{\one}{\mathbbm{1}}

\newcommand{\Ac}{\mathcal{A}}

\newcommand{\Cc}{\mathcal{C}}

\newcommand{\Fc}{\mathcal{F}}

\newcommand{\Ic}{\mathcal{I}}

\newcommand{\Lc}{\mathcal{L}}
 \newcommand{\Nc}{\mathcal{N}} \newcommand{\Oc}{\mathcal{O}}

\newcommand{\Tc}{\mathcal{T}}

\newcommand{\Yc}{\mathcal{Y}}
\newcommand{\Zc}{\mathcal{Z}}

\newcommand{\Af}{\mathfrak{A}}  
  
 \newcommand{\Hf}{\mathfrak{H}} %\newcommand{\If}{\mathfrak{I}}

 \newcommand{\Sf}{\mathfrak{S}}  
\newcommand{\Xf}{\mathfrak{X}} \newcommand{\Yf}{\mathfrak{Y}}
\newcommand{\af}{\mathfrak{a}} 

\newcommand{\hf}{\mathfrak{h}}

\newcommand{\Dr}{\mathrm{D}}
\newcommand{\Fr}{\mathrm{F}}

\newcommand{\Kr}{\mathrm{K}}
\newcommand{\Nr}{\mathrm{N}}
\newcommand{\Sr}{\mathrm{S}}

\newcommand{\rr}{\mathrm{r}}
\newcommand{\tr}{\mathrm{t}}

\newcommand{\ubf}{\mathbf{u}}
\newcommand{\vbf}{\mathbf{v}}

\newcommand{\wt}{\widetilde}
\newcommand{\wh}{\widehat}

\newcommand{\<}{\langle}
\newcommand{\uph}{\upharpoonright}
\newcommand{\imb}{\hookrightarrow}

\newcommand{\ii}{\mathrm{i}}
\newcommand{\ee}{\mathrm{e}}
\newcommand{\dd}{\mathrm{d}}
\newcommand{\pa}{\partial}

\newcommand{\disc}{\mathrm{disc}}
\newcommand{\ess}{\mathrm{ess}}
\newcommand{\sym}{\mathrm{sym}}

\newcommand{\Max}{{\max}}
\newcommand{\Min}{{\min}}

\DeclareMathOperator{\im}{Im}
\DeclareMathOperator{\re}{Re}
\DeclareMathOperator{\dom}{dom}

\DeclareMathOperator{\Dist}{dist}

\DeclareMathOperator{\Gr}{Gr}
\DeclareMathOperator{\Div}{div}
\DeclareMathOperator{\gradm}{\mathbf{grad}}

\DeclareMathOperator{\gradn}{\mathbf{grad}_0}
\DeclareMathOperator{\Divn}{\Div_0}

\DeclareMathOperator{\Var}{\mathrm{Var}}

\DeclareMathOperator{\Hom}{{\mathcal{LH}}}
\DeclareMathOperator{\Na}{\nabla}

\newcommand{\Ao}{A}

\newcommand{\A}{\mathcal{A}}

\newcommand{\gan}{\ga_{\mathrm{n}}}

\newcommand{\x}{\mathbf{x}}
\newcommand{\n}{\mathbf{n}}
\newcommand{\Hs}{\Hf}

\newcommand{\D}{G}
\newcommand{\paD}{{\pa \D}}
\newcommand{\DepaD}{{\De_\paD}}

\newcommand{\imp}{z}

\newcommand{\cross}{{\natural}}
\newcommand{\mul}{\mathcal{M}}
\newcommand{\mulp}{M}

\newcommand{\oCC}{{\overline{\CC}}}
\newcommand{\ove}{\overline}
\newcommand{\Lra}{\Leftrightarrow}

\newcommand{\sto}{{\shortto}}

\makeatletter
\DeclareRobustCommand{\shortto}{%
  \mathrel{\mathpalette\short@to\relax}%
}

\newcommand{\short@to}[2]{%
  \mkern2mu
  \clipbox{{.3\width} 0 0 0}{$\m@th#1\vphantom{+}{\shortrightarrow}$}%
  }
\makeatother

\makeatletter
\def\blfootnote{\xdef\@thefnmark{}\@footnotetext}
\makeatother
\begin{document}

%%%%% redefines Proof
%\expandafter\let\expandafter\oldproof\csname\string\proof\endcsname
%\let\oldendproof\endproof
%\renewenvironment{proof}[1][\proofname]{%
% \oldproof[#1]%
%}{\oldendproof}
%%%%%%%

\title{Sobolev multipliers and fractional Gaussian fields on Lipschitz boundaries with applications to deterministic and random acoustic systems}

\author{Illya M. Karabash}
\date{}
\maketitle

%{ \center{\large Illya M. Karabash%$^\text{ a}$
%\\[2ex] 
%}   }  

\blfootnote{Institute for Applied Mathematics, the University of Bonn, Bonn, Germany}
\blfootnote{E-mails: karabash@iam.uni-bonn.de, i.m.karabash@gmail.com}

\bigskip

\bigskip

\begin{center}\noindent\it{}
To Volodymyr Derkach on the occasion of his 75th birthday

and

to Mark Malamud on the occasion of his 76th birthday
\end{center}

\vspace{4ex}

\begin{abstract}
Motivated by Applied Physics and Photonics  studies of random resonators, 
we study   in the stochastic part of this paper  random acoustic operators in non-smooth bounded domains $\D \subset \RR^d$ and  introduce m-dissipative  impedance boundary conditions  containing (eigenfunction) fractional Gaussian fields.
The deterministic part of the paper constructs and studies the spaces of pointwise multipliers on Lipschitz boundaries $\paD$, as well as the spaces of Sobolev (distribution-type) multipliers on boundaries $\paD$ of better regularity.
These multipliers are used  as generalized impedance coefficients $\zeta$  in impedance boundary conditions 	$\zeta p    =  \n  \cdot \vbf  $  accompanying the 1st order acoustic system. We study the m-dissipativity of associated acoustic operators and the discreteness of their spectra aiming the main efforts on the weakest possible assumptions on the regularity of $\zeta$.
In order to pass from deterministic results to the randomization, we introduce fractional Gaussian fields (FGFs) on Lipschitz  boundaries $\paD$ and study their regularity. To this end, we prove that a rough Weyl-type asymptotics takes place for the Laplace-Beltrami  eigenvalues on arbitrary compact Lipschitz boundary $\paD$.
\end{abstract}

\vspace{1ex}
{\footnotesize
\noindent
MSC2020-classes: 
%35Q61 %   	Maxwell equations
35F45   %	Boundary value problems for systems of linear first-order PDEs
47B44 % Accretive operators, dissipative operators, etc.
47B80   %	Random linear operators
58J90 %   	Applications of PDEs on manifolds
%35P05   %	General topics in linear spectral theory for PDEs
35J56   %	Boundary value problems for first-order elliptic systems
60H25 %  	Random operators and equations (aspects of stochastic analysis)
%78A25  %   	Electromagnetic theory, general
%47F10  % 	Elliptic operators and their generalizations 
%34G10 %   	Linear differential equations in abstract spaces
%47D03, %   	Groups and semigroups of linear operators
\\[0.5ex]
Keywords: maximal dissipative operator, discrete spectrum, essential spectrum, random boundary condition, pointwise multiplier, Browninan bridge, generalized random field, random operator, Weyl's law, compact resolvent,   extension theory, m-accretive operator 
}

\medskip

{\noindent\footnotesize \textsc{Acknowledgment.} 
The author is supported by the Heisenberg Programme (project 509859995) of the German Science Foundation (DFG, Deutsche Forschungsgemeinschaft), by the Hausdorff Center for Mathematics funded by the German Science Foundation under Germany's Excellence Strategy -- EXC-2047/1 -- 390685813, 
and by the CRC 1060 `The mathematics of emergent effects' funded by the German Science Foundation at the University of
Bonn. 
}

\section{\label{s:i}Introduction}

Presently, there is a substantial gap between the mathematical optimization  involving eigenvalues of dissipative wave equations  \cite{K13,OW13,KKV20,EK25} and the areas of Physics  where this type of nonselfadjoint spectral optimization problems has emerged. 
While the pioneering resonance optimization paper of  Harrell \& Svirsky \cite{HS86} seems to be motivated by the estimation of Schrödinger random resonances, the topic attracted a substantial attention   much later in connection with numerical  simulations  for photonic crystals and high-Q optical cavities   (see \cite{HBKW08,DMTSH14,KKV20,VS21} and references therein). 

The contemporary Applied Physics  studies of the Q-factor optimization often assume randomness of the structure surrounding the region where the resonance mode is predominantly localized. This helps to model the influence of small fabrication errors on the Q-factors $\frac{|\re \la|}{-2\im \la}$ of resonances $\la$ positioned near the real axis (high-Q resonances, see \cite{DMTSH14,VS21}).
 Rigorous mathematical models for such stochastic optimization problems are still not adequately developed and may require a substantial simplification in order to obtain qualitative mathematical  results.

The resonances of random structures are studied in Mathematical Physics mainly for discrete Schrödinger operators \cite{K16} or, more generally, on the level of nonselfadjoint random matrices \cite{FS15}. A few papers has considered asymptotics of continuation resonances for multidimensional random Schrödinger operators \cite{S14,AK21}.
In the context of dissipative Maxwell eigenproblems, the modeling of an uncertain surrounding by randomized boundary conditions on an artificial boundary was discussed in \cite{EK22}. Uncertain boundary conditions emerging in such models are linear and are supposed to describe the leakage of energy from the system. These boundary conditions are 
quite different from those of \cite{HO00}, where white noise enters as a boundary control (or source) term. 
Rigorous randomization of linear m-dissipate boundary conditions  was introduced for acoustic systems in the preceding paper of the author \cite{K26}, where the emphasis was put on non-local boundary conditions.
% written in terms of eigenfunctions of Laplace-Beltrami operator $\DepaD$ on the boundary $\paD$. 

The present paper puts the main focus on  
local \emph{impedance boundary conditions} 
\begin{gather*} %\label{e:ImpBC0}
	\imp (x)  p (x)   =  \n (x) \cdot \vbf  (x) , \qquad  x \in \pa \D,
\end{gather*}
as well as their \emph{distributional analogues} that occurs to be convenient for randomization, e.g., by means of appropriate Gaussian fields.
Following Leis \cite{L13}, we write the acoustic system   in the Schrödinger form 
$	\ii \pa_t \Psi = \af_{\al,\beta} \Psi $,  where  
\begin{gather} \label{e:a0}
(\af_{\al,\beta} \Psi)(x) = \af_{\al,\beta}  \begin{pmatrix} \vbf (x) \\ p (x)\end{pmatrix} =  \frac1\ii\begin{pmatrix} \ab^{-1} \nabla p (x) \\
		\beta^{-1} \nabla \cdot \vbf (x) \end{pmatrix} , \quad x \in \D \subset \RR^d, % \quad d \ge 2,
\end{gather}
is the corresponding differential operation in $x$ 
with  spatially varying and uniformly positive material parameters \quad
%\[\text{
 $\ab \in L^\infty (\D,\RR_\sym^{d \times d})$ \quad and \quad $\beta \in L^\infty (\D,\RR)$,
%  } \]
which play the role of the weights in the weighted phase space 	$\LL^2_{\ab,\beta} (\D)$ of the system (see Section \ref{s:Ac1}).
Engineeringly oriented studies often consider wave equations in domains with non-smooth (typically Lipschitz) boundaries $\paD$. 
The boundary $\paD$ in this paper is assumed to have either Lipschitz regularity, or $C^{k-1,1}$-regularity with an appropriate $k \in \NN$ depending on $d \ge 2$.
Any of these assumptions implies that the   outward normal unit vector $\n: \paD \to \RR^d$ is well-defined  as a vector-field $\n \in L^\infty  (\pa \D, \RR^d)$. The $L^p$-spaces on $\paD$ correspond to the standard surface measure $\dd \Si = \dd \Si_{\paD}$ of the $C^{0,1}$-boundary  $\paD$ (i.e., Lipschitz boundary), see \cite{GMMM11,G11}.

%Following \cite{L13}, the \emph{`energy space'} $\LL^2_{\ab,\beta} (\D) $ can be introduced as the linear space    
%$ L^2 (\D,\CC^d) \times  L^2 (\D)$ equipped with the weighted norm  $ \| \cdot \|_{\LL^2_{\ab,\beta} }$ corresponding to the energy inside $\D$
%\[
%\| \{\vbf,p\}  \|_{\LL^2_{\ab,\beta} }^2  
%= ( \ab \vbf | \vbf)_{L^2 (\D,\CC^d) }  + ( \beta p  | p )_{L^2 (\D) } = \int_\D ( \ab  %\vbf \vert \vbf )_{\CC^d}  + \int_\D \beta |p|^2 ,
%\]
%which makes $\LL^2_{\ab,\beta} (\D) $ a Hilbert space. 

In the cases where the acoustic system emerges as a dimensionally reduced Maxwell system \cite{FK96}, 
it makes sense to consider following Shchukin and Leontovich (see \cite{YI18}) complex valued    $\imp:\paD \to \CC$.
Under certain additional conditions, an impedance coefficient $\imp:\paD \to \overline{\CC}_\rr$ with values in the complex right half-plane $ \overline{\CC}_\rr := \{z \in \CC : \re z \ge 0\}$ ensures dissipativity  or m-dissipativity of the associated acoustic operator 
(see Section \ref{s:Ac1} for definitions).
In Mechanics problems, the function $\imp (\cdot)$ is usually nonnegative and is called the (boundary) damping coefficient.
 Impedance boundary conditions 
with uniformly positive $L^\infty$-impedance coefficients are often  used as the main  class of absorption boundary conditions for acoustic and Maxwell systems \cite{JS21,EK22,PS22,V22,EK25}. Their m-dissipativity and the m-dissipativity of more general classes of absorbing boundary conditions were studied for acoustic systems in \cite{KZ15,JS21,S21,V22,PT24,K25,K26}. The m-dissipativity property is crucial for modeling of lossy resonators due to the following characterization \cite{P59,Kato,E12}:
\begin{multline*}% \label{e:m-disGen}
	\text{an operator $A:\dom A \subseteq \Xf \to \Xf$ is m-dissipative if and only if}
	\\ \text{$\ee^{-\ii A t}$ is a (strongly continuous) contraction semigroup.}
\end{multline*}

For randomization, it is convenient to replace the impedance coefficient $\imp (\cdot)$ by generalized
functions (distributions) $\zeta$ on $\paD$. 
To this end, we employ the theory of Sobolev (pointwise) multipliers. 
The classes of Sobolev multipliers $\MM [W^{s_1,p_1} (\RR^n) \to W^{s_2,p_2} (\RR^n)]$ between different  Sobolev spaces on $\RR^n$ (as well as  between Besov or  Triebel–Lizorkin spaces) have been a topic of active research, see, e.g., \cite{S67,MV95,KS02,MS04,MV04,NS06,MS09}.  

It is certainly not a goal of the present paper to take the whole extensive (and impressive) theory of multipliers on $\RR^n$ and carry it over to the case of the Lipschitz boundary  $\paD$ of $\D$. The goal is rather to find the part of this theory that is important for the study of m-dissipativity and randomization and to emphasize some differences in comparison with the $\RR^n$-case.

A transfer of the $\RR^n$-theory of Sobolev multipliers to Lipschitz boundaries $\paD$ meets serious difficulties   from the first steps,
and so, we consider on such boundaries  multipliers in a narrow pointwise sense (they are always $L^1(\paD,\CC)$-functions, see Section \ref{s:PM}). 

In order to come closer to spaces of distributions  and to analogies with the theory of  Sobolev  multipliers on $\RR^n$, we use in Section \ref{s:Gen} the stronger regularity assumption that $\paD$ is a $C^{k-1,1}$-boundary for $k > (d-1)/2$. 
Note that, in the 2-dimensional case $d=2$, this assumption is satisfied for Lipschitz boundaries automatically, and so,
the results of our two approaches overlap, see Sections \ref{s:2d} and \ref{s:RandAcOp}. 

In order to get reasonable models for dissipation resonances, and then, to randomize these models, one has to find classes of distribution-type  impedance coefficients that produce acoustic operators $\A$ with reasonable spectral  properties.
The most important of these properties are the m-dissipativity of $\A$, the resolvent compactness, and the discreteness of its spectrum (see the discussion in \cite{K26}).
From the point of view of the m-dissipativity, we show that multipliers $\zeta \in \MM [H^{1/2} (\paD) \sto H^{-s_2} (\paD)]$ between fractional Sobolev spaces of a positive regularity $1/2$ and a negative regularity $(-s_2)$ with large $s_2>0$ are of the main interest
 (see Sections \ref{s:MulpMd} and \ref{s:GenFImp}).
 
 On the other hand, the impedance coefficients from the space of 
compact Sobolev multipliers $ \MM^{\Sf_\infty} [H^{1/2} (\paD) \sto H^{-1/2} (\paD)]$
are important from the point of view of the study of 
the discrete spectrum $\si_\disc (\A)$  of $\A$ (see Section \ref{s:CompRes} for basic definitions).
In the case $d=2$, we connect the class $ \MM^{\Sf_\infty} [H^{1/2} (\paD) \sto H^{-1/2} (\paD)]$ with (fractional) Sobolev spaces $H^{-s} (\paD)$, $s \in (0,1/2)$, 
see \eqref{e:(d-1)/2} and Section \ref{s:2d}.

It is possible to meet sometimes the popular belief that a good partial differential  operator $T$, e.g., a Laplacian, in a bounded domain $\D$ with a regular enough boundary has a purely discrete spectrum for any reasonable boundary condition, reasonable, e.g., in the sense that the resolvent set $\rho (T)$ is non-empty. 
However, it is well-known (at least for the operator theory community) that this is actually not so, and that there exist selfadjoint boundary conditions for the Laplacian in a ball that produce a non-empty essential spectrum $\si_\ess (T)$.
For acoustic operators $\A$, it makes sense to study the resolvent compactness and the discreteness of the spectrum for the part $\A |_{\GG_{\ab,\beta}}$ of the operator in the reducing subspace $\GG_{\ab,\beta} = \LL^2_{\ab,\beta} (\D) \ominus (\HH_0 (\Div 0,\D) \oplus\{0\})$, see Section \ref{s:CompRes} for  explanations. Example \ref{e:Ess} considers some of the cases where   $\si_\ess (\A |_{\GG_{\ab,\beta}}) \neq \emptyset$.

In the case of random boundary conditions, the almost sure (a.s.) discreteness of of $\si (\A |_{\GG_{\ab,\beta}}) $ is important for the representation of 
the random spectrum $\si (\A |_{\GG_{\ab,\beta}}) $ via a proper stochastic point process in $\CC$. Such a representation ensures that a distance $\Dist (\si (\A), \la_*) $ to a fixed real frequency $\la_* \in \RR \setminus \{0\}$ is a random variable, and that the expectation
$\EE \Dist (\si (\A), \la_*)$ can be used as a cost function for optimization of resonances over various feasible families of deterministic material parameters. This is the simplest stochastic analogue of the deterministic cost function of  \cite{EK25}.

The distribution-valued random impedance coefficients are crucial for the randomization by means of (eigenfunction) fractional Gaussian fields. We define analogues of  fractional Gaussian fields (FGFs) on Lipschitz boundaries  using the $L^2 (\pa \D)$-orthonormal basis $\{Y_j\}_{j \in \NN}$ of 
real-valued eigenfunctions of the nonnegative Laplace-Beltrami operator $\De_\paD$ on $\paD$. The eigenfunctions are numbered in such a way that $\De_\paD Y_j (x)=  \mu_j Y_j (x)$ and $\{\mu_j\}_{j \in \NN}$ is a non-decreasing sequence (for the operator $\De_\paD$ in the case of $C^{0,1}$-boundary, see \cite{GMMM11} and also Sections \ref{s:m-disGIBC} and  \ref{a:W}). We denote by $b_0 = b_0 (\paD)$ the dimensionality $\dim \ker \DepaD$ of the kernel 
$\ker \DepaD := \{Y \in L^2 (\paD): \DepaD Y = 0 \}$ of 
$\DepaD$. Note that $b_0 \in \NN$ and that, for $1 \le j \le b_0$, the eigenfunctions $Y_j$ are locally constant and can be chosen in such a way that they are nonnegative (in the sequel, we always assume this  choice).

We define on $\paD$ an FGF with index $s$   as a formal  random series 
\begin{gather*} %\label{e:FGF0}
	\Xi_s (x,\om)  = \sum_{n=b_0+1}^{+\infty} \xi_n (\om) \mu_n^{-s/2} Y_n (x), \qquad x \in \paD, \quad \om \in \Om, 
\end{gather*}
where $\xi_n$ are independent identically distributed (i.i.d.) $\RR$-valued Gaussian random variables of unit variance on a certain complete probability space $(\Om,\Fc,\PP)$.
If $t < s-(d-1)/2$,
% =: \hf_\paD $, 
this  series rigorously define a random vector in a generalized Sobolev-type space $H^t_\DepaD$, see Sections \ref{s:m-disGIBC} and \ref{s:2d}
(for FGFs on compact Riemann $C^\infty$-manifolds, we refer to the recent overview 
\cite{CS25}).
In order to obtain this  convergence result for $\sum \xi_n  \mu_n^{-s/2} Y_n$ on an arbitrary compact Lipschitz boundary $\pa \D$, we prove for  $\paD$ the rough analogue \[
\mu_n  \asymp  n^{2/(d-1)} \quad \text{ as $n \to +\infty$}
\]	
of Weyl's asymptotics of the Laplace-Beltrami eigenvalues $\mu_n$, see Theorem \ref{t:W}.
% and the proof in Appendix \ref{a:W}.

As an application of our approach, we construct in Section \ref{s:RandAcOp} various families of random boundary conditions and associated acoustic operators. Let us give a particular  example.  Let $\{c_n\}_{n=1}^{b_0} \subset \ove \RR_+ := [0,+\infty)$ be a sequence of nonnegative constants, and let $\{\eta_n\}_{n=1}^{b_0}$ be a sequence  
of i.i.d. positive exponentially distributed random variables (the random variables  $\eta_n$
are assumed to be  non-degenerate in the sense that their expectations $\EE \eta_n $ are positive).

\begin{thm} \label{t:ExFGF}
Let $d=2$. Let us fix arbitrary constants $s>0$ and $c\in \RR$.
We define the (generalized) random impedance coefficient $\zeta$ by 
\[\textstyle
\zeta = c \Xi_s + \sum_{n=1}^{b_0} c_n  \eta_n Y_n , \quad \text{where  $\Xi_s  = \sum_{n=b_0+1}^{+\infty} \xi_n \mu_n^{-s/2} Y_n$ is an FGF as above.}
\]
Then $\zeta$ is a random vector in the space $\MM^{\Sf_\infty} [H^{1/2} (\paD) \sto H^{-1/2} (\paD)]$, and the associated with \eqref{e:a0} and  $\zeta   \ga_0 (p) =  \gan (\vbf)$ randomized  acoustic operator $ \A:\om \mapsto  \A_\om:\dom\Ac_\om \subset \LL^2_{\ab,\beta} (\D) \to \LL^2_{\ab,\beta} (\D)$ 
 has the following properties:
\item[(i)] The operator $ \A$ is a.s. m-dissipative and is weakly random in the resolvent sense.
\item[(ii)]  With probability 1, $\A |_{\GG_{\ab,\beta}}$ has a compact resolvent and $\si (\A |_{\GG_{\ab,\beta}}) = \si_\disc (\A |_{\GG_{\ab,\beta}})$.
\item[(iii)] The probability that $\A$ is selfadjoint is  either 1, or $0$. 
Namely, $\PP \{\A = \A^*\} =1$ if and only if $0= c_1 = \dots = c_{b_0}$.
\end{thm}

This theorem follows from the results of Section \ref{s:RandAcOp}, see Remark \ref{ex:ExFGF}. 

Since $\Xi_s$ becomes a white noise model for $s=0$, this result can be interpreted in the following way. We are not able to use in a meaningful way a white noise $\Xi_0$ as the most singular component of the impedance coefficient $\zeta$ saving simultaneously the m-dissipativity of $\A$, but, in the 2-dimensional case, we "almost do this" in the sense that arbitrary small values of $s>0$ are allowed in Theorem \ref{t:ExFGF}.
 
%On the base of Theorems \ref{} and \ref{}, we expect that, in the case $d=2$,
% it is possible to define an acoustic operator $\A^{\mathrm{wn}}$ associated with
% the white noise impedance coefficient $\Theta_0$, but such an operator 
%$\A^{\mathrm{wn}}$ would have a nontrivial essential spectrum.

\textbf{Notation.} 
In what follows, $\Xf$, $\Xf_1$, and $\Xf_2$ are (complex) Hilbert spaces. 
The notation $T:\dom T \subseteq \Xf_1 \to \Xf_2$ means that a (linear) operator $T$ is considered as an operator from $\Xf_1$ to $\Xf_2$, whereas it is
defined on a domain (of definition) $\dom T  $ of $T$ that is a certain linear (possibly non-closed) subspace of $\Xf_1$.
If $\dom T = \Xf_1$, we sometimes write $T: \Xf_1 \to \Xf_2$. By $[\Xf_1,\Xf]_\vth$ we denote the spaces obtained by the complex interpolation, see \cite{T78} and \cite[Vol.2]{RS}.

An orthogonal sum of  Hilbert spaces $\Xf_1$ and $\Xf_2$ is denoted by $\Xf_1 \oplus \Xf_2$. By $\Gr T := \{\{f,Tf\} \ : \ f \in \dom T\} \subset \Xf_1 \oplus \Xf_2 $ we denote the graph of the operator $T$, and consider it as a normed space with the graph norm. We  use the natural identification of $\Gr T$ and   $\dom T$, so that, $\dom T$ also becomes a normed space with the graph norm.

By $\Lc (\Xf_1,\Xf_2)$ we denote the  Banach space of bounded operators $T : \Xf_1 \to \Xf_2$, while 
$\Hom (\Xf_1,\Xf_2)$ is the set of \emph{linear homeomorphisms} from $\Xf_1$ to $\Xf_2$.
Besides, $\Lc (\Xf) := \Lc (\Xf,\Xf)$ and $\Hom (\Xf) := \Hom (\Xf,\Xf)$.

By $I_\Xf$ (by $0_{\Lc(\Xf)}$) the identity operator (resp., the zero operator) in the space $\Xf$ is denoted, although  the subscript is dropped  if the space is clear from the context. 
%In the notation for the resolvent 
%$(T-\la)^{-1}=(T-\la I)^{-1}$,   the identity operator $I$ is often skipped. 

The spectrum an operator $T:\dom T \subseteq \Xf \to \Xf$ is denoted by $\si (T)$.
The notation $\rho (T) := \CC \setminus \si (T)$ stands for the resolvent set.
A symmetric operator $T$ is called nonnegative (and we write $T \ge 0$) if $(Tf|f)_{\Xf} \ge 0$ for all $f \in \dom T$. The notation $T_1 \le T_2$ means $T_2 -T_1 \ge 0$.  A symmetric operator is uniformly positive if $T \ge \de I$ for a certain constant $\de>0$. % For $T \in \Lc (\Xf)$, $\re T = \frac12(T+T^*)$ and $\im T = \frac{1}{2\ii} (T - T^*)$ are the real and imaginary parts of $T$ (the definition of $\re T$ and $\im T$ are extended in Section \ref{} to a more general situation of rigged Hilbert spaces).

%By $\one$ we denote the constant function equal to $1$. For $1 \le p \le \infty$, $\ell^p (\NN)$ are the standard complex Banach $\ell^p$-spaces of sequences $\{a_j\}_{j \in \NN}$, and $\Sf_p = \Sf_p (\Xf)$  are the Schatten-von-Neumann ideals of compact operators in $\Xf$.
%So $\dom T$ is also considered as a normed space with the graph norm.

\section {Acoustic operators in bounded Lipschitz domains}
\label{s:Ac}

\subsection{Acoustic systems, trace maps, and functional spaces}
\label{s:Ac1}

Let $\D$ be a Lipschitz domain in $\RR^d$, where $d \ge 2$. This means that  $\D$ is a non-empty bounded open connected subset of $\RR^d$ and that its  boundary $\pa \D$ is of the
$C^{0,1}$-regularity \cite{GMMM11,G11}, i.e., $\paD$ is a Lipschitz (continuous) boundary. This assumption implies, in particular, that the   outward normal unit vector-field 
$
\n \in L^\infty  (\pa \D, \RR^d)
$ is well-defined almost everywhere (a.e.) on $\partial \D$   with respect to (w.r.t.) the surface measure $\dd \Si$ of $\paD$.

By $L^2 (\D,\CC^d)$ we denote the Hilbert space of $\CC^d$-valued vector fields in $\D$ equipped with the inner product 
\[
(  \ubf | \vbf )_{L^2 (\D,\CC^d)} =  \int_{\D} \ubf \cdot \overline{\vbf}   =  \int_{\D} (  \ubf  | \vbf  )_{\CC^d}  .
\] 
For $s > 0$, we use  the  complex Hilbertian Sobolev spaces 
$H^s ( \D) = W^{s,2} ( \D, \CC)$ and $H^s_0 ( \D) = W^{s,2}_0 ( \D, \CC)$, as well as the corresponding spaces of distributions $H^{-s} ( \D) $ that are 
dual to  $H^s_0 ( \D) $ w.r.t. the pivot space $L^2 (\D) =L^2 (\D,\CC) $.

On $\pa \D$, we use the spaces
$L^q (\pa \D) = L^q  (\pa \D, \CC)$ of 
$\CC$-valued $L^q$-functions with $1 \le q \le \infty$.
For $s \in (0,1]$,  one defines  the complex Hilbertian Sobolev spaces of $\CC$-valued  functions $H^s (\pa \D) $ via localization  and pullback \cite{G11,GMMM11}. This approach leads to an infinite family of inner products and associated equivalent norms on $H^s (\pa \D) $. We assume that, for each $s $, one of such inner  products that makes $H^s (\paD) $ a Hilbert space is fixed.
The Hilbert spaces $H^{-s} (\pa \D) $ are defined as dual spaces to  $H^{s} (\pa \D) $  w.r.t.  $H^0 (\pa \D) =(L^2 (\pa \D), \ (\cdot|\cdot)_{L^2 (\paD)})$,
where $(f,g)_{L^2 (\paD)} = \int_\paD f\ove g $.

By  $\RR_{\sym}^{d\times d}$ we denote the Hilbert space of real symmetric $d\times d$-matrices equipped with the  Frobenius norm. 
The material parameters $\ab  = (\ab_{j,k} )_{j,k=1}^d \in L^\infty (\D, \RR_{\sym}^{d\times d})$ and $\beta \in L^\infty (\D,\RR)$ are assumed to be uniformly positive functions.
% with the values in $\RR_{\sym}^{d\times d}$ and $\RR$ respectively. 
In the case of $\ab (\cdot)$, this means that, for a certain real constant $\al_0>0$,  $\ab (\x) \ge \al_0 I $ for almost all (a.a.) $\x \in \D$. 
%In the inequality $\ab (\x) \ge \al_0 \II $,   
%$\ab (\x)$ and  the identity $d\times d$-matrix $\II$ are identified with selfadjoint operators in $\CC^d$.
By $\ab^{-1} (\cdot)$, we denote the pointwise inversion  $ (\ab (\x))^{-1}$, $\x \in \D$, for the matrix-function $\ab (\cdot)$. 
%The uniform positivity of  $\ab (\cdot)$ implies  $\ab^{-1} \in L^\infty (\D, \RR_{\sym}^{d\times d}) $.
For Banach spaces $\Yf_1$, $\Yf_2$, and a (linear) operator $T:\dom T \subseteq \Yf_1 \to \Yf_2$, we consider the  domain  $\dom T  \subseteq \Yf_1$ of $T$ as a normed space with  the graph norm.
By $\Lc (\Yf_1,\Yf_2)$ we denote the  Banach space of bounded operators $T : \Yf_1 \to \Yf_2$ with $\dom T =\Yf_1$. %Besides, $\Lc (\Yf_1) = \Lc (\Yf_1,\Yf_1)$.

For the following statements about divergence, gradient, and trace operators,
we refer to \cite{G11,L13,KZ15,S21,K26}.
 The scalar trace 
 \[\text{$\ga_0 : p \mapsto p\!\uph_{ \pa \D}$ is understood as a continuous operator $\ga_0 \in \Lc (H^1 (\D), H^{1/2} (\pa \D))$.}
 \]
The map $ f \mapsto \nabla f $
defines two closed gradient operators with two different domains 
\[
\gradm : H^1 (\D) \subset L^2 (\D) \to L^2 (\D,\CC^d), \quad \gradn : H_0^1 (\D) \subset L^2 (\D) \to L^2 (\D,\CC^d),
\]
where $H_0^1 (\D) = \{ f \in H^1 (\D) : \ \ga_0 f = 0 \}$.
The adjoint of $(-1)\gradn$ is the divergence operator 
\[
\Div : \ubf \mapsto \nabla \cdot \ubf, \quad
\Div : \HH (\Div, \D) \subset L^2 (\D,\CC^d) \to L^2 (\D)
\]
with the domain 
\[
 \HH (\Div, \D) := \{ \ubf \in L^2 (\D,\CC^d)  : \nabla \cdot \ubf \in L^2 (\D)\} ,
\] 
which equipped with the associated graph norm becomes a Hilbert space.
The normal trace is understood as a continuous operator from $\HH (\Div, \D)$ to $H^{-1/2} (\pa \D)$:
\[
\text{$
\gan : \ubf  \mapsto \n \cdot \ubf \uph_{\pa \D}$, \qquad  $\gan \in \Lc (\HH (\Div, \D), H^{-1/2} (\pa \D))$.}
\]
The operator $\Div_0 : \HH_0 (\Div,\D) \subset L^2 (\D,\CC^d)   \to L^2 (\D)$ 
defined by $\Div_0 := \gradm^*$
is a restriction of the operator $\Div$ to 
\[
\text{ $\HH_0 (\Div,\D) = \{\ubf \in \HH (\Div, \D) \ : \ \gan (\ubf) = 0\}$.}
\]

The  Hilbert space  $ \LL^2_\ab = \LL^2_\ab (\D)$ equals   
$ (L^2 (\D,\CC^d),  \| \cdot \|_{\LL^2_\ab} ) $, where the  weighted norm $ \| \cdot \|_{\LL^2_\ab}$ is defined by
$  \| \vbf  \|_{\LL^2_\ab}^2 =  ( \ab \vbf | \vbf)_{L^2 (\D,\CC^d) }$.
We use also the weighted Hilbert space of $\CC$-valued functions 
$L^2_\beta = L^2_\beta (\D)$ with the norm
$\| p  \|_{L^2_\beta} :=  ( \beta p | p)_{L^2 (\D) }^{1/2} $.

Following \cite{L13}, the \emph{`energy space'} $\LL^2_{\ab,\beta} (\D) $ is introduced as  
the orthogonal sum 
\[
\LL^2_{\ab,\beta} = \LL^2_{\ab,\beta} (\D)  := \LL^2_{\ab} (\D) \oplus L^2_\beta (\D) .
 \]
The energy space   $\LL^2_{\ab,\beta} (\D)$ is the phase space where we consider 
the acoustic evolution equation $\ii \pa_t \Psi = \af_{\al,\beta} \Psi$ with 
the formal spatial differential operation 
\begin{gather*}
 \af_{\al,\beta} \Psi (x) = \af_{\al,\beta}  \begin{pmatrix} \vbf (x) \\ p (x) \end{pmatrix} =  \frac1\ii\begin{pmatrix} (\ab(x))^{-1} \nabla p (x)  \\
 (\beta(x))^{-1} \nabla \cdot \vbf (x) \end{pmatrix} , \qquad x \in \D. 
 %\label{e:aEq}
\end{gather*}
The  operation $\af_{\al,\beta}$ is symmetric w.r.t. the inner product $(\cdot,\cdot)_{\LL^2_{\ab,\beta}}$ of  $\LL^2_{\ab,\beta} (\D)$. Here and below partial differential operators  in $\D$ are considered as operators in $\LL^2_{\ab,\beta} (\D)$, i.e., their symmetry, m-dissipativity, etc. are understood in the sense of $(\cdot,\cdot)_{\LL^2_{\ab,\beta}}$.

Let us  introduce the closed symmetric acoustic operator 
\begin{gather*}%\label{e:A}
 \A_\Min \begin{pmatrix} \vbf \\ p \end{pmatrix} = \frac1\ii \begin{pmatrix} 0 &  \ab^{-1} \gradn \\
 \beta^{-1} \Divn & 0 \end{pmatrix} \begin{pmatrix} \vbf \\ p \end{pmatrix} 
\end{gather*}
with (natural) minimal  domain $\dom \A_\Min =  \HH_0 (\Div,\D) \times  H_0^1 (\D) $. 
The adjoint operator $\A_\Max = \A_\Min^*$ is also associated with  $\af_{\ab,\beta}$, but has the maximal  domain  $\dom \A_\Max = \HH (\Div,\D) \times H^1 (\D)$.

\subsection{M-dissipativity and generalized impedance  conditions}
\label{s:GIBC}

In the conventions of \cite{E12,EK22,K26},  an  operator $T:\dom T \subseteq \Xf \to \Xf$ in a Hilbert space $\Xf$ is called \emph{dissipative 
(accretive)} if $\im (Tf|f)_{\Xf} \le 0$ (resp., $\re (Tf|f)_{\Xf} \ge 0$) 
for all $f \in \dom T$.
Clearly, $T$ is accretive if and only if $(-\ii) T$ is dissipative. 

An  operator $T:\dom T \subseteq \Xf \to \Xf$ is called  \emph{m-dissipative} if $\CC_+ := \{z \in \CC : \im z >0\}$ 
is a subset of its resolvent set $\rho (T)$
and 
$\| (T- z)^{-1}\| \le (\im z)^{-1}$ for all $z \in \CC_+ $  (m-dissipative operators are always dissipative, see, e.g., \cite{Kato}). 
An operator $T$ is called \emph{m-accretive} if $(-\ii) T$ is m-dissipative. A closable operator $T$ is called \emph{essentially m-dissipative (essentially m-accretive)} if its closure $\ove T$
is m-dissipative (resp., m-accretive). An  operator $K \in \Lc (\Xf)$ is called a \emph{contraction} if $\| K \| \le 1$.

The role of m-dissipative operators in the modeling of lossy systems is shown by the following result \cite{P59,Kato,E12}:
%\begin{multline*} \label{e:m-disGen1}
%	\text{
	an operator $T$ is m-dissipative if and only if
	$(-\ii) T$ is a generator of a contraction semigroup.
%	} \end{multline*}

%\begin{defn}[GIBCs of DtN-type] \label{d:GIBC}
A boundary condition 
$\Zc  \ga_0 (p) =  \gan (\vbf)  $, where  
$\Zc:\dom \Zc \subseteq H^{1/2} (\pa \D)\to H^{-1/2} (\pa \D)$,
is called 
a \emph{generalized impedance boundary condition} (GIBC) \cite{HJN05,PT24,K26,K25}.
We use in this paper only GIBCs of Dirichlet-to-Neumann-type (DtN-type), cf. \cite{HJN05}.
The operator $\Zc$  is called the \emph{impedance operator} (of DtN-type). 
Other classes of boundary conditions for acoustic systems, as well as related  questions of well-posedness and m-dissipativity, were considered in \cite{HJN05,KZ15,S21,V22,PT24,HKW25,K26} (for  Maxwell systems, see also \cite{LL04,YI18,S21,EK22,EK25} and the references therein).

\begin{defn}[\cite{HJN05,KZ15,PT24,K26,K25}] \label{d:BC}
Consider an operator $\Zc:\dom \Zc \subseteq H^{1/2} (\pa \D)\to H^{-1/2} (\pa \D)$.
The acoustic operator $\Ac_\Zc$ associated with the GIBC of DtN-type $\Zc  \ga_0 (p) =  \gan (\vbf)  $
is  the restriction  $ \A_\Zc = \A_\Max \uph_{\dom \A_\Zc}$
of $\A_\Max$ to 
\[
\dom  \A_\Zc = \{ \{\vbf,p\} \in \dom \A_\Max \ : \ \Zc  \ga_0 (p) =  \gan (\vbf)  \text{ in $H^{-1/2} (\paD)$} \} .
\]
The  boundary condition $\Zc  \ga_0 (p) =  \gan (\vbf)  $ is called dissipative or m-dissipative 
 if the associated acoustic operator $\A$ is dissipative or m-dissipative, respectively. 
\end{defn}

The classical \emph{impedance boundary condition} \cite{HJN05,JS21,K25}
\begin{gather*} %\label{e:IBC}
\imp  \ga_0 (p) =  \gan (\vbf)  , 
\end{gather*}
corresponds to the particular case where $\Zc$ is an operator of  multiplication on a function $z(x)$, $x\in \paD$.
It is an analogue of the Leontovich boundary condition for Maxwell system \cite{LL04,YI18,EK22,PS22,EK25}.
The \emph{impedance coefficient} $\imp$  is usually assumed to be a uniformly positive  $L^\infty$-function.
This assumption ensures the m-dissipativity of the impedance boundary condition and the discreteness of the nonzero spectrum $\si (\cdot) \setminus \{0\}$ for acoustic and Maxwell operators \cite{LL04,JS21,EK22,PS22,EK25}.
For acoustic systems,  the uniform positivity can be replaced  by the assumption that $\imp $ belongs to  $L^{q} (\paD, \overline{\CC}_\rr)$ with $q>d-1$ \cite{K26}, where $L^{q} (\paD, \overline{\CC}_\rr)$ is the space  of $L^{q}$-functions with values in $ \overline{\CC}_\rr := \{\la \in \CC : \re \la \ge 0\}$  (the case of a Maxwell system is substantially more difficult \cite{EK22,PS22,EK25}).

One of the goals of the present paper is to relax the assumptions on $\imp (\cdot)$ and, in particular, to consider distributions as impedance coefficients 
 in such a way that the associated acoustic operator $\A$ remains  m-dissipative and, possibly, has other good spectral properties.
In the case where $\imp (\cdot)$ is not an $L^\infty (\paD)$-function, the difficulty is
that  the operator of the multiplication on $\imp (\cdot)$ may have many non-equivalent interpretations, see Sections \ref{s:MulpMd}, \ref{s:GenFImp}, and \ref{s:dis}. 
%The other goals are to find classes of distribution that can be employed  as impedance coefficients and use these classes  for the randomization of $\imp$.
It is not a priori clear which of these interpretations produce m-dissipative acoustic operators.
%The question becomes even more involved for randomized $\imp$.

\subsection{Laplace-Beltrami operators and m-dissipative GIBCs.}
\label{s:m-disGIBC}

By  
$\<\cdot,\cdot\>_\paD$, we denote the (sesquilinear) $L^2 (\paD)$-pairing between spaces 
$H^s (\paD)$ and $ H^{-s} (\paD)$, $s \in [-1,1]$, that is generated by the  continuous extension of the inner product $(\cdot|\cdot)_{L^2 (\paD)}$.
For an operator  $T: \dom T \subseteq H^s (\paD) \to H^t (\paD)$ with arbitrary $s,t \in [-1,1]$, we use  (see \cite{GMMM11,EK22} and  Appendix \ref{s:GenAcc}) the adjoint operator
\[ \text{ $T^\cross: \dom T^\cross \subseteq H^{-t} (\paD) \to H^{-s} (\paD)$ \ w.r.t. the pairing $\<\cdot,\cdot\>_\paD$}
\]
and call $T^\cross$ the $\cross$-adjoint (or natural adjoint) operator to $T$.
If $s=-t$ and $T=T^\cross$, we say that $T$ is $\cross$-selfadjoint operator.

The 
 nonnegative selfadjoint   Laplace-Beltrami  operator
 \[
 \text{
 $ \De_\paD: \dom (\De_\paD) \subset L^2 (\pa \D) \to L^2  (\pa \D) $}
 \]  
 on $\paD$ has  a purely discrete spectrum
 $\si (\De_\paD) = \si_\disc (\DepaD) \subset \ove \RR_+$ \cite{GMMM11} (for  the definition of $\si_\disc (\cdot)$, see Section \ref{s:CompRes}).
 
 For all $s>0$,   the Hilbert space
 $H^s_\DepaD := \dom (\De_\paD^{s/2})$ 
is assumed to be equipped with the graph norm 
of $\De_\paD^{s/2}$. We define 
$H^{-s}_\DepaD$ as the dual space to $H^s_\DepaD$ w.r.t. the pivot space $H^0_\DepaD := L^2 (\paD)$.
Then $H^s_\DepaD$ is a Hilbert space for every $s \in \RR$ and 
\begin{gather}\label{e:HsDeImbImb}
\text{the compact dense embeddings $ H^s_\DepaD \imb \imb H^t_\DepaD$ hold for all $t <s$ \cite{GMMM11}.}
\end{gather}
It follows from the results of \cite{GMMM11} that 
\begin{gather} \label{e:HsDe=Hs}
\text{for $s \in [-1,1]$, \qquad  $ H^s_\DepaD = H^s (\paD) $ \qquad (up to equivalence of norms),}\\
\text{and } \qquad H^{t_0(1-\vth)+t_1\vth}_\DepaD = [H^{t_0}_\DepaD , H^{t_1}_\DepaD ]_\vth \qquad \text{ for all $t_0,t_1 \in \RR$ and  $\vth \in [0,1]$,}
  %= (H^s (\paD),H^{-s} (\paD))_{\vth,2} . % \qquad \vth \in [0,1].
 \label{e:IntHs-sDe}
\end{gather}
where $[\Yf_0,\Yf_1]_\vth$ denotes the interpolation space produced by 
the complex interpolation (for basic facts concerning interpolation, see, e.g., \cite{T78}).

The definition of the sesquilinear pairing $\<f,g\>_\paD$ introduced earlier  for $f \in H^s (\paD)$ and $g \in H^{-s} (\paD)$ with $s \in [-1,1]$, can be extended to 
the sesquilinear pairings  
$\<f,g\>_\paD$ of $f \in H^s_\DepaD $ and $g \in H^{-s}_\DepaD$ with arbitrary $s \in \RR$.
We keep for this extension the same notation $\<\cdot,\cdot\>_\paD$ and, as above, consider the associated $\cross$-adjoint operators $T^\cross$ to $T:\dom T \subseteq H^s_\DepaD \to H^t_\DepaD $ with $s,t \in \RR$, see Appendix \ref{s:GenAcc}.

By $\Hom (\Xf_1,\Xf_2) $ we denote the set of \emph{linear homeomorphisms} between   Hilbert spaces $\Xf_1$ and $\Xf_2$.
For all $s,t \in \RR$, $c>0$, and for  the positive selfadjoint operator $\De_\paD+cI$,
the eigenfunction expansion  of $\De_\paD$ implies that  the fractional power 
% (defined via the functional calculus) 
 \[
 (\De_\paD+cI)^{s/2}: \dom (\De_\paD+cI)^{s/2} \subseteq L^2 (\pa \D) \to L^2  (\pa \D) 
 \] 
 can be modified via an appropriate extension by continuity and  a restriction to $H^t_\DepaD$
 in order  to produce the linear homeomorphism
$
(\De_\paD+c)^{s/2}_{H^t \shortto H^{t-s}} \in \Hom (H^t_\DepaD, H^{t-s}_\DepaD) .
$
The corresponding $\cross$-adjoint operator is in the same class:
\begin{gather*} \label{e:Des2}
\left((\De_\paD+c)^{s/2}_{H^t \shortto H^{t-s}}\right)^\cross = (\De_\paD+c)^{s/2}_{H^{-t+s} \shortto H^{-t}} .
\end{gather*}

Recall that an extension $ \wh T : \dom T \subseteq \Xf_1 \to \Xf_2 $ of $T : \dom T \subseteq \Xf_1 \to \Xf_2 $ is called a \emph{proper extension} if $\dom \wh T \supsetneqq \dom T$. 
The following definition is a natural generalization of notions of accretive and nonnegative operators, see \cite{GMMM11,EK22}. 

\begin{defn} \label{d:Gac}
For $s \in \RR$, consider an operator $T:\dom T \subseteq H^s_\DepaD \to H^{-s}_\DepaD $.
\item[(i)]
The operator $T$ is called  \emph{accretive}   if
  $\re \<T f , f \>_\paD \ge 0 $ for all $f \in \dom T$.
\item[(ii)] The operator $T$ is called  \emph{nonnegative}  if 
 $\<T f , f \>_\paD \ge 0 $ for all $f \in \dom T$. 
\item[(iii)] An accretive (nonnegative) operator $T$ is called \emph{maximal accretive} (resp., maximal nonnegative) if it has no proper extensions $\wh T$
%:\dom \wh T \subseteq H^s_\DepaD \to H^{-s}_\DepaD $ 
in the same class.
\end{defn}

%\begin{rem}
This definition is not sensitive to a change of the norms in
$H^{\pm s}_\DepaD$ and  can be applied also to the spaces $H^{\pm s} (\paD)$ with $s \in [-1,1]$,  see \eqref{e:HsDe=Hs}.
An important technical tool for the whole paper is the following theorem, which is an combination and  adaptation  of several results of  \cite{P59,EK22,K26,K25}.
 
\begin{thm}[cf. \cite{EK22,K25}] \label{t:DtN-GIBC}
For an operator $\Zc:\dom \Zc \subseteq H^{1/2} (\pa \D)\to H^{-1/2} (\pa \D)$, let $\A_\Zc$ be the acoustic operator associated with $\Zc  \ga_0 (p) =  \gan (\vbf)  $ (see  Definition \ref{d:BC}).
Then the following statements are equivalent:
\begin{itemize}
\item[(i)] the operator $\A_\Zc$ is m-dissipative in $\LL^2_{\ab,\beta} (\D) $;
\item[(ii)] $
(\De_\paD+1)^{-1/4}_{ H^{-1/2} \shortto H^0 } \ \ \Zc \ \ (\De_\paD+1)^{-1/4}_{H^0 \shortto H^{1/2}}  $
is an  m-accretive operator in $L^2 (\paD)$;
%\item[(ii)] $\Zc$ is maximal accretive and closed;
%\item[(iii)] $\Zc$ is maximal accretive and closed (as an operator from $H^{1/2} (\paD)$ to $H^{-1/2} (\paD)$).
\item[(iii)] $\Zc$ is an accretive  operator and there exists a densely defined 
accretive operator 
$Q : \dom Q \subseteq H^{1/2} (\pa \D)\to H^{-1/2} (\pa \D) $
such that $\Zc = Q^\cross$;
\item[(iv)] $\Zc$ is maximal accretive and densely defined. 
%The operator $(- \ii) \A_\Zc$,  which is  associated with the acoustic differential expression $(-1)\sfr_{\al,\beta}$, is a generator of a contraction semigroup 
%$\{\ee^{- \ii t \A_\Zc}\}_{t\ge 0}$  in $\LL^2_{\ab,\beta} (\D) $.
\end{itemize}
Moreover, $\A_\Zc$ is selfadjoint if and only if $\Zc^\cross = - \Zc$.
\end{thm}

The proof of this theorem is given in Appendix \ref{s:ProofA}.

\section{Pointwise multipliers  on $\paD$ and m-dissipativity}
\label{s:PM+M-d}

Our 2nd main tool is the family of spaces of pointwise multipliers between Hilbert spaces   $H^{t_1}_{\DepaD} $ and $H^{t_2}_{\DepaD} $  (we use mainly the case $t_2< 0 < t_1$).
%We use also the family of pointwise multipliers from $H^{s_1} (\paD)$ to $H^{-s_2} (\paD)$ with $s_1,s_2 \in [0,1]$.

The spaces of or Sobolev multipliers
$
M[H^{t_1} (\RR^d) \sto H^{t_2} (\RR^d)]
$
 acting between fractional Sobolev spaces $H^{t_1} (\RR^n)$ and $H^{t_2} (\RR^n)$ have been studied intensively, see  \cite{S67,MV95,MS04,MV04,NS06,MS09} and references therein.
If $t_2<0$, Sobolev multipliers can be distributions.
In Section \ref{s:MultGen},  we adapt  
the spaces of distributions $\MM [H^{t_1} (\RR^n) \sto H^{t_2} (\RR^n)]$ to the case where $\RR^d$ is replaced by a $C^{m,1}$-boundary $\paD$ with a suitable smoothness parameter $m = m(d)$.

In this section, we work with a boundary $\paD$ of a general Lipschitz domain and are primarily interested in the multiplication operators $(p \! \uph_\paD) \to \imp \ ( p \! \uph_\paD)$  that are generated by  integrable impedance coefficients $\imp \in L^1 (\paD)$. 

%In the case of $\RR^n$, there exist characterizations \cite{MV04,LG06,BS19} for distributions
% in some of the spaces of multipliers  $\MM(H^{t_1} (\RR^d) \sto H^{t_2} (\RR^d))$, as well
% as sufficient and/or necessary conditions in other, more difficult, cases, see \cite{MV04,LG06,BS18}. 
%The discussion of the applicability of these characterizations on $\RR^n$ to the impedance boundary conditions on $\paD$ is postponed to Section \ref{}.

\subsection{Pointwise multipliers on Lipschitz boundaries}
\label{s:PM}

Let $H^\infty_\DepaD := \bigcap_{s \in \RR}  H^s_\DepaD $.
Let $\Xf_1$, and $\Xf_2$ be Hilbert spaces and let $\Sf_\infty (\Xf_1,\Xf_2)$ be the space of compact operators from $\Xf_1$ to $\Xf_2$. 
%We put $\Sf_\infty (\Xf_1) := \Sf_\infty (\Xf_1,\Xf_1)$.

The notation `$\lesssim$'  below in \eqref{e:intADe} (and throughout the paper) means that the corresponding inequality `$\le $' is valid after multiplication of the right  side  on a certain constant $c>0$ (in \eqref{e:intADe}
this constant is  independent of $g_1$ and $g_2$, but obviously depends on $f$).

\begin{defn} \label{d:MMs1s2De}
Let $s_1,s_2 \in \RR$. 
 Assume that  a measurable function $f: \paD \to \CC$ is such that the product $f g_1 \ove g_2$ belongs to $ L^1 (\pa \D)$ for all  $g_1, g_2  \in H^\infty_\DepaD $
 and satisfies 
\begin{gather} \label{e:intADe}
\left|\int_{\paD}  f g_1 \ove{g}_2  \right| \lesssim \| g_1 \|_{H^{s_1}_\DepaD}  \| g_2 \|_{H^{s_2}_\DepaD},  \qquad g_1, g_2   \in H^\infty_\DepaD .
\end{gather}
\item[(i)] Then we say that $f$ belongs to the space  $\MM (H^{s_1}_\DepaD , H^{-s_2}_\DepaD  )$  of  pointwise multipliers from $H^{s_1}_\DepaD$ to $H^{-s_2}_\DepaD$.
\item[(ii)] Let us denote by 
$
\mulp_f^{H^{s_1} \sto H^{-s_2}} 
$
the unique $\Lc (H^{s_1}_\DepaD ,H^{-s_2}_\DepaD)$-operator  satisfying 
\begin{equation*}  %\label{e:Mfs-t}
\< \mulp_f^{H^{s_1} \sto H^{-s_2}} g_1 , g_2\>_\paD = \int_{\paD}  f g_1 \ove{g}_2 , \qquad  g_1, g_2   \in H^\infty_\DepaD .
\end{equation*}
If
$
\mulp_f^{H^{s_1} \sto H^{-s_2}} \in  \Sf_\infty (H^{s_1}_\DepaD ,H^{-s_2}_\DepaD) 
$
we  say that 
$f$ belongs to 
%\[ \text{
the space \linebreak $ \MM^{\Sf_\infty} (H^{s_1}_\DepaD , H^{-s_2}_\DepaD )$ of compact pointwise multipliers
(from $H^{s_1}_\DepaD$ to $H^{-s_2}_\DepaD$).
%\]
\end{defn}

\begin{rem} \label{r:DensHinf}
The operator $\mulp_f^{H^{s_1} \sto H^{-s_2}} \in \Lc (H^{s_1}_\DepaD ,H^{-s_2}_\DepaD)$ in Definition \ref{d:MMs1s2De} is well-defined.
Indeed, the linear space $H^\infty_\DepaD$ contains all eigenfunctions of $\DepaD$, and so, is dense in every $H^s_\DepaD$ with $s \in \RR$
(this follows from $\DepaD = \De_\paD^*$ in $L^2 (\paD)$, see \cite{GMMM11}).
\end{rem}

Note that the function $\one $ that takes the value $1$ everywhere on $\paD$ is an eigenfunction of $\DepaD$. So, $\one \in H^\infty_\DepaD$. 
With $g_1 = \one$ or $g_2 = \one$ in Definition \ref{d:MMs1s2}, we get 
\begin{equation} \label{e:MimbL1}
\MM (H^{s_1}_\DepaD , H^{-s_2}_\DepaD ) \subseteq L^1 (\paD) \cap H^{-s_1}_\DepaD \cap  H^{-s_2}_\DepaD \quad \text{for all $s_1,s_2 \in \RR$.}
\end{equation}

In order to define $\MM (H^{s_1}_\DepaD , H^{-s_2}_\DepaD)$ as a normed space, we use  \eqref{e:MimbL1}  and consider $\MM (H^{s_1}_\DepaD , H^{-s_2}_\DepaD)$ as a
 linear subspace of $H^{-\min\{s_1,s_2\}}_\DepaD$
(so, we interpret functions $f$ in Definition \ref{d:MMs1s2De} as the corresponding equivalence classes). 

The norm in the resulting  linear space $ \MM (H^{s_1}_\DepaD, H^{-s_2}_\DepaD)$ is defined  as the operator norm
$
\|f \|_{\MM (H^{s_1}_\DepaD, H^{-s_2}_\DepaD)} := \| \mulp_f^{H^{s_1}  \shortto H^{-s_2} } \|_{\Lc (H^{s_1}_\DepaD, H^{-s_2}_\DepaD)} .
%, \qquad  f \in \MM (H^{s_1}_\DepaD, H^{-s_2}_\DepaD) .
$
In these settings, the positive definiteness of $\|\cdot \|_{\MM (H^{s_1}_\DepaD, H^{-s_2}_\DepaD)} $ follows from the fact that $H^\infty_\DepaD$ is
dense in the space $H^t_\DepaD$ with $t=\min\{s_1,s_2\}$ (see Remark \ref{r:DensHinf}).

In what follows, we consider $\MM (H^{s_1}_\DepaD , H^{-s_2}_\DepaD)$ as normed spaces.
Note that $\one \in \MM (H^{s_1}_\DepaD , H^{-s_2}_\DepaD) \neq \{0\}$ if
$s_1 \ge -s_2$. The set $ \MM^{\Sf_\infty} (H^{s_1}_\DepaD, H^{-s_2}_\DepaD )$ of compact multipliers is a closed subspace of the normed space $\MM (H^{s_1}_\DepaD , H^{-s_2}_\DepaD)$.

The embeddings \eqref{e:HsDeImbImb} yield that if $ s_1 \le t_1$, $s_2 \le t_2 $, and $s_1 +s_2 < t_1 + t_2$, then 
\begin{gather} \label{e:MinM12}
%\MM (H^{s_1}_\DepaD , H^{-s_2}_\DepaD) \imb \MM (H^{t_1}_\DepaD , H^{-t_2}_\DepaD) \qquad \text{ if } \quad s_1 \le t_1 \text{ and } s_2 \le t_2 ;\\
	\MM (H^{s_1}_\DepaD , H^{-s_2}_\DepaD) \imb \MM^{\Sf_\infty} (H^{t_1}_\DepaD , H^{-t_2}_\DepaD)  \imb \MM (H^{t_1}_\DepaD , H^{-t_2}_\DepaD) .
\end{gather}

Definition \ref{d:MMs1s2De} and Remark \ref{r:DensHinf} imply that
%\begin{gather*} %\label{e:M=M}
$\MM (H^{s_0}_\DepaD , H^{-s_1}_\DepaD) = \MM (H^{s_1}_\DepaD , H^{-s_0}_\DepaD )$
for all $s_0,s_1 \in \RR$,
%\end{gather*}
that $f$ and $\overline f$  belong to $\MM (H^{s_0}_\DepaD , H^{-s_1}_\DepaD) $ simultaneously with the same norm, and that 
%\begin{gather*}
$(\mulp_f^{H^{s_0}  \shortto H^{-s_1} } )^\cross = \mulp_{\overline{f}}^{H^{s_1}  \sto H^{-s_0}  }$ 
(for $\cross$-adjoint operators, see Sections \ref{s:m-disGIBC} and \ref{s:GenAcc}).
%\end{gather*}
If $f \in \MM (H^{s_0}_\DepaD , H^{-s_1}_\DepaD) $,
one can interpolate between the bounded operators  $ \mulp_f^{H^{s_0} \sto H^{-s_1}}$ and  $ \mulp_f^{H^{s_1} \sto H^{-s_0}}$ using the  \eqref{e:IntHs-sDe} and the reiteration formula for the interpolation (see \cite{T78}). This gives  for all $\vth \in [0,1]$ the formula
\begin{gather}\label{e:MpInt}
\MM (H^{s_0}_\DepaD , H^{-s_1}_\DepaD) = \MM (H^{s_1}_\DepaD , H^{-s_0}_\DepaD ) \imb 
\MM (H^{s_0(1-\vth) +s_1 \vth}_\DepaD , H^{-s_1 (1-\vth) - s_0 \vth}_\DepaD) .
\end{gather} 

\begin{rem} \label{d:MMs1s2HpaD}
The equality $ H^s_\DepaD = H^s (\paD) $ for $s \in [-1,1]$ (see \eqref{e:HsDe=Hs})
makes it possible to substitute the spaces $H^s (\paD)$ into Definition \ref{d:MMs1s2De}. 
This leads for $s_1,s_2 \in [-1,1]$ to the normed spaces $\MM (H^{s_1} (\pa \D), H^{-s_2} (\pa \D) )$ and $ \MM^{\Sf_\infty} (H^{s_1} (\pa \D), H^{-s_2} (\pa \D) )$.
The advantage of the passage to the spaces $H^s (\paD)$
is the possibility to use for $s \in [-1,1]$ the Sobolev-type embeddings
and, in this way, to connect the spaces of pointwise multipliers with $L^q(\paD)$-spaces,
see Section \ref{s:Lq}.
\end{rem}

\subsection{M-dissipativity and restricted multiplication operators}
\label{s:MulpMd}

Below, we use without special notice the fact that $ H^s_\DepaD = H^s (\DepaD) $ for $s \in [-1,1]$. 
Since $\gan (\vbf)$ in the boundary condition $\imp  \ga_0 (p) =  \gan (\vbf)$ belongs to $H^{-1/2} (\Om)$, we need to adapt  the multiplication operators $\mulp_\imp^{H^{t} \shortto H^{-s} }$ to the $H^{-1/2} (\paD)$-settings.

\begin{defn}[restricted multiplication operators] \label{d:wtM}
Let $s,t \in \RR$ and let \linebreak $f \in \MM (H^{t}_\DepaD , H^{-s}_\DepaD)$.
%\item[(i)] 
The restricted multiplication operator  
%\begin{gather*}
\[
 \mulp_f^{t, -s} : \dom \mulp_f^{t,-s} \subseteq H^{1/2} (\paD)  \to H^{-1/2} (\paD) %\\
\] 
%\quad \text{
with %} \quad 
$
\dom  \mulp_f^{t,-s} := \{ u \in H^{1/2} (\paD) \cap H^t_\DepaD  : \ 
\mulp_f^{H^t \shortto H^{-s}} u \in  H^{-1/2} (\paD)   \}
$
%\end{gather*}
is defined by 
$
\mulp_f^{t, -s} u := \mulp_f^{H^t \shortto H^{-s}} u  $ \quad  for all $ u \in \dom  \mul_f^{t,-s}$.
\end{defn}

%The logic behind this definition is that, for $\imp \in \MM (H^{1/2}_\DepaD , H^{-s}_\DepaD)$,  %the impedance boundary condition 
%\[
%\text{$\imp  \ga_0 (p) =  \gan (\vbf)$ can be naturally interpreted as 
% $\mulp_\imp^{1/2, -s}  \ga_0 (p) =  \gan (\vbf)$.}
%\]

 Recall that, for an operator $\Zc:\dom \Zc \subseteq H^{1/2} (\pa \D) \to H^{-1/2} (\pa \D)$, $\A_\Zc$ is the acoustic operator associated with the boundary condition $\Zc  \ga_0 (p) =  \gan (\vbf)  $ (see  Definition \ref{d:BC}), and note that in this paper we do not work with linear relations (cf. \cite{EK22,K26}).

\begin{thm} \label{t:Inter}
Assume that $s \ge 1/2$ and $\imp \in \MM (H^{1/2}_\DepaD , H^{-s}_\DepaD )$.
Then  the following statements are equivalent:
\begin{itemize}
\item[(i)] $\int_\paD \re (\imp) |g|^2 \ge 0 $ for all $g \in H^\infty_\DepaD$;
\item[(ii)] the operator $\mulp_{\imp}^{s, -1/2}$ is accretive;
\item[(iii)] there exists a restriction $\wt \Zc$ of $\mulp_{\imp}^{1/2, -s}$ such that 
		the acoustic operator $\A_{\wt \Zc} $ is m-dissipative;
\item[(iv)] there exists an extension  $\wh \Zc$ of $\mulp_{\imp}^{s, -1/2}$ such that 
		$\A_{\wh \Zc} $ is m-dissipative.
\end{itemize}
 \end{thm}

This theorem is proved in Section \ref{ss:PrT:1/2maxdis}. 
The role of the assumption \linebreak $\imp \in \MM (H^{1/2}_\DepaD , H^{-s}_\DepaD )$ 
is that it guaranties that $\mulp_{\imp}^{s,-1/2}$ and $\mulp_{\imp}^{1/2, -s}$ are densely defined. 
This allows us to ensure that $\wt \Zc$ and $\wh \Zc$ are operators (and are not multi-valued linear relations). 
Since $\wt \Zc$ is a restriction of the multiplication operator $\mulp_{\imp}^{1/2, -s}$ associated with the impedance coefficient $\imp (\cdot)$, the m-dissipative boundary condition $\wt \Zc  \ga_0 (p) =  \gan (\vbf)  $ is associated with $\imp (\cdot)$ in a certain (quite implicit) sense. This association of $\wt \Zc$ and $\imp$ is vague because 
$\dom \wt \Zc$ is not described  and because we have no control yet on the uniqueness of  m-dissipative boundary conditions of this type. 
 
The next theorem characterizes in terms of the closure $\ove{\mulp_\imp^{s,-1/2}}$ one of the cases where an m-dissipative interpretation of  $\imp  \ga_0 (p) =  \gan (\vbf)  $ is unique.

\begin{thm} \label{t:Macc}
Assume that $s \ge 1/2$, $\imp \in \MM (H^{1/2}_\DepaD , H^{-s}_\DepaD )$, 
and $\int_\paD \re (\imp) |g|^2 \ge 0 $ for all $g \in H^\infty_\DepaD$.
	Then the following statements are equivalent: 
		\begin{itemize}
		\item[(i)] the operator $\mulp_\imp^{1/2,-s}$ is accretive;
		\item[(ii)] $\mulp_\imp^{s,-1/2}$ is closable and $\ove{\mulp_\imp^{s,-1/2}} = \mulp_\imp^{1/2,-s}$;
		\item[(iii)] the acoustic operator $\A_\Yc$ with $\Yc =  \mulp_\imp^{s,-1/2}$ is essentially m-dissipative;
		\item[(iv)] the acoustic operator $\A_\Zc$ with $\Zc =  \mulp_\imp^{1/2,-s}$ is m-dissipative.
		\end{itemize}
\end{thm}

This theorem is proved in Section \ref{ss:PrT:1/2maxdis}.

\begin{rem} \label{r:Mp12-12}
Consider the case where $s=1/2$, $\imp \in \MM (H^{1/2}_\DepaD , H^{-1/2}_\DepaD )$, and $\re \imp \ge 0$ a.e. on $\paD$. Then $\Zc = \mulp_\imp^{1/2,-s}$ coincides with $\mulp_\imp^{s,-1/2}= \mulp_\imp^{H^{1/2} \sto H^{-1/2}} $ and is a bounded accretive operator. By Theorem \ref{t:Macc}, $\A_\Zc$  is m-dissipative. If additionally $\re \imp = 0$ a.e. on $\paD$,
then $\Zc^\cross = - \Zc$, and so, Theorem \ref{t:DtN-GIBC} implies $\A_\Zc = \A_\Zc^*$.
\end{rem}

\subsection{A bit of  operator theory for restricted multipliers}
\label{s:ResM}

Let $f  \in \MM (H^{1/2}_\DepaD , H^{-s}_\DepaD )$ with $s \in [1/2,+\infty)$.
Then also $\ove  f \in \MM (H^{1/2}_\DepaD , H^{-s}_\DepaD )$.

\begin{lem} \label{l:Mf12-s}
(i) The functions $f$ and $\ove f$ are in $ \MM (H^s_\DepaD , H^{-1/2}_\DepaD ) $.
\item[(ii)] The operator $ \mulp_f^{s, -1/2}$ is a densely defined restriction of 
$\mulp_f^{1/2, -s}$. 
\item[(iii)] $( \mulp_{\ove f}^{s, -1/2})^\cross =  \mulp_f^{1/2, -s} $, and so, $\mulp_f^{1/2, -s}$ is a closed operator.
\item[(iv)] The operator $\mulp_f^{s, -1/2}$ has the closure $\ove{\mulp_f^{s, -1/2}} = (\mulp_{\ove f}^{1/2, -s})^\cross$.
\end{lem}
\begin{proof}
%Let $s \in [1/2,+\infty)$ and let $f  \in \MM (H^{1/2}_\DepaD , H^{-s}_\DepaD )$.
%Recall that  $ H^t_\DepaD = H^t (\DepaD) $ for $t \in [-1,1]$, and so,
%$H^s_\DepaD$ is dense in $ H^{1/2}_\DepaD = H^{1/2} (\paD) $.

(i) Statement (i) follows from  Definition \ref{d:wtM}.

(ii) Since $1/2\le s$, 
the operator 
$
 \mulp_f^{H^{s} \shortto H^{-1/2}} \in \Lc (H^{s}_\DepaD, H^{-1/2}_\DepaD )
$
  defined in Definition \ref{d:MMs1s2De} coincides as a mapping  with the operator $  \mulp_f^{s, -1/2}$ and has the same domain 
%\begin{gather} \label{e:wtMs-1/2dom}
$\dom  \mulp_f^{s, -1/2} = \dom \mulp_f^{H^{s} \shortto H^{-1/2}}  = H^{s}_\DepaD $,
%\end{gather}
which is dense in $ H^{1/2}_\DepaD$. 
The comparison of the domains in Definition \ref{d:wtM} gives $\Gr \mulp_f^{s, -1/2} \subseteq \Gr \mulp_f^{1/2, -s}$.
This proves (ii).

(iii) Since $\ove  f \in \MM (H^{1/2}_\DepaD , H^{-s}_\DepaD )$,  statement (ii) yields that  
the operator $\mulp_{\ove f}^{s, -1/2}$ is densely defined and there exists the $\cross$-adjoint  operator 
$(\mulp_{\ove f}^{s, -1/2})^\cross$. Let us find it.
The $\cross$-adjoint of $ \mulp_{\ove f}^{H^{s} \shortto H^{-1/2}} \in \Lc (H^{s}_\DepaD , H^{-1/2}_\DepaD )$
is
$
 \mulp_f^{H^{1/2} \shortto H^{-s}} \in \Lc (H^{1/2}_\DepaD, H^{-s}_\DepaD ).
$
Hence, for all $g \in H^s_\DepaD = \dom \mulp_{\ove f}^{s, -1/2} $ and $u \in H^{1/2}_\DepaD $,
\begin{gather*} \label{e:ofgu=}
\<\mulp_{\ove f}^{s, -1/2} g , u\>_\paD= \<\mulp_{\ove f}^{H^{s} \shortto H^{-1/2}} g , u\>_\paD = \< g , \mulp_f^{H^{1/2} \shortto H^{-s}}  u\>_\paD .
\end{gather*}
The definition of $\cross$-adjoint operator implies that $u \in H^{1/2}_\DepaD$
belongs to $\dom (\mulp_{\ove f}^{s, -1/2})^\cross $ if and only if 
$\mulp_f^{H^{1/2} \shortto H^{-s}}  u \in H^{-1/2}_\DepaD$, i.e., 
if and only if $u \in \dom \mulp_f^{1/2, -s} $. 
Thus, $\mulp_f^{1/2, -s} = ( \mulp_{\ove f}^{s, -1/2})^\cross$.

(iv) By (i), (ii), and (iii), the operators $\mulp_f^{s, -1/2}$
and $( \mulp_f^{s, -1/2})^\cross =  \mulp_{\ove f}^{1/2, -s} $ are densely defined. 
Thus,
$
\ove{\mulp_f^{s, -1/2}} = (\mulp_f^{s, -1/2})^{\cross\cross} = (\mulp_{\ove f}^{1/2, -s})^\cross .
$
\end{proof}

\begin{lem} \label{l:Mf12-sAcc}
Assume additionally that the operator $\mulp_f^{1/2,-s}$ is accretive. Then: 
\item[(i)] The operators $\mulp_f^{s,-1/2}$,  $\mulp_{\ove f}^{s,-1/2}$, and $\mulp_{\ove f}^{1/2,-s}$ are accretive.
\item[(ii)] $ \mulp_f^{1/2, -s} $ is a maximal accretive, densely defined, and closed operator;
\item[(iii)] $ \ove{ \mulp_f^{s,-1/2}} = \mulp_f^{1/2, -s} $.
\end{lem}
\begin{proof}
(i) The operator $\mulp_f^{s,-1/2}$ is accretive since it is a restriction of the accretive operator 
$\mulp_f^{1/2,-s}$. The accretivity of $ \mulp_{\ove f}^{1/2, -s} $ is obtained by the comparison of the process of the construction for $ \mulp_{\ove f}^{1/2, -s} $ and $\mulp_f^{1/2,-s}$ in  Definitions  \ref{d:MMs1s2De}-\ref{d:wtM} and Remark \ref{r:DensHinf}. Then $ \mulp_{\ove f}^{s,-1/2} $ is also accretive as a restriction of $ \mulp_{\ove f}^{1/2, -s} $.

(ii) It follows from statement (i) and Lemma \ref{l:Mf12-s} that $  \mulp_{\ove f}^{s,-1/2} $ is accretive and densely defined. Taking into account Lemma \ref{l:Mf12-s} (iii), we see that the 
accretive densely defined closed operator $ \mulp_f^{1/2, -s} $ is the $\cross$-adjoint  of 
the accretive densely defined operator $ \mulp_{\ove f}^{s,-1/2} $. The part (iii) $\Leftrightarrow$ (iv) of Theorem \ref{t:DtN-GIBC} shows that $\mulp_f^{1/2,-s}$ is maximal accretive.

(iii) Statements (i)-(ii) and  Lemma \ref{l:Mf12-s} imply that $\mulp_f^{s,-1/2} $ is accretive, densely defined, and has the closure $\ove{\mulp_f^{s, -1/2}} = (\mulp_{\ove f}^{1/2, -s})^\cross$, which is also accretive (as a closure) and densely defined.
Using statement (i) and Theorem \ref{t:DtN-GIBC} (iii) $\Lra$ (iv),
we see that accretive densely defined operator $\ove {\mulp_f^{s,-1/2} }$ is maximal accretive as a $\cross$-adjoint to  $\mulp_{\ove f}^{1/2, -s}$. By statement (ii) and Lemma \ref{l:Mf12-s}, $ \mulp_f^{1/2, -s} $ is a  maximal accretive  extension of  $\ove {\mulp_f^{s,-1/2} }$. Thus, $ \ove{ \mulp_f^{s,-1/2}} = \mulp_f^{1/2, -s} $.
\end{proof}

\subsection{Proofs of Theorems \ref{t:Inter} and \ref{t:Macc} \label{ss:PrT:1/2maxdis} }

\begin{prop} \label{p:tPh59}
Let $\Zc:\dom \Zc \subseteq H^{1/2} (\paD)  \to H^{-1/2} (\paD)$ be 
 a densely defined accretive operator. Then there exist a maximal accretive extension $\wh \Zc$ of 
 $\Zc$ and a densely defined maximal accretive   restriction $\wt \Zc:\dom \wt \Zc \subseteq H^{1/2} (\paD)  \to H^{-1/2} (\paD)$ of $Z^\cross$  such that $(\wh \Zc)^\cross = \wt \Zc$ and $\wh \Zc = (\wt \Zc)^\cross$. 
\end{prop}  
\begin{proof}
Using the method of \cite[Section 6]{EK22} (see also Appendix \ref{s:ProofA}), this proposition can be reduced to the corollary after \cite[Theorem 1.1.2]{P59} (the corollary after Theorem 1.1.2 in \cite{P59}  should be applied to the operator $(\De_\paD+1)^{-1/4}_{ H^{-1/2} \shortto H^0 } \ \ \Zc \ \ (\De_\paD+1)^{-1/4}_{H^0 \shortto H^{1/2}}  $, see \cite[Section 6]{EK22} and also Appendix \ref{a:AbsBC} for the details of this reduction). Note that the conventions  for dissipative operators in the Phillips paper \cite{P59} are different from the conventions of \cite{E12,EK22} that are used in the present paper.
\end{proof}

\begin{prop}[cf. \cite{EK22,K26}] %[Sections 6.2 and 7]
	\label{p:EssMD}
	Let $\A_\Yc$ be an acoustic operator associated with \linebreak $\Yc \ga_0 (p) =  \gan (\vbf)$,
	where  the operator $\Yc:\dom \Yc \subseteq H^{1/2} (\paD) \to H^{-1/2} (\paD)$ is closable.
		Then  $\ove{\A_\Yc} = \A_{\ove \Yc}$. % and $(\A_\Yc)^* = \A_{-\Yc^\cross}$.
		 Besides, $\A_\Yc$ is essentially m-dissipative if and only  $\ove \Yc$ is maximal accretive and densely defined. 
\end{prop}

\begin{proof}
	This proposition follows from the combination of  \cite[Remarks 6.1-6.3]{EK22} with the part
	(i) $\Lra$ (iv) of Theorem \ref{t:DtN-GIBC}.
\end{proof}

\begin{proof}[Proof of Theorem \ref{t:Inter}]
Let $s \ge 1/2$ and $\imp \in \MM (H^{1/2}_\DepaD , H^{-s}_\DepaD ) = \MM (H^{s}_\DepaD , H^{-1/2}_\DepaD ) $. 

(i) $\Lra$ (ii). 
Since $ \MM (H^{s}_\DepaD , H^{-1/2}_\DepaD ) \imb  \MM (H^{s}_\DepaD , H^{-s}_\DepaD )$,
the operator $\mulp_\imp^{H^{s} \sto H^{-s}} \in \Lc (H^{s}_\DepaD , H^{-s}_\DepaD ) $
is well-defined and coincide with $\mulp_{\imp}^{s, -1/2}$ as a map.

Assume that, for all $g \in H^\infty_\DepaD$, we have $\int_\paD \re(\imp) |g|^2  \ge 0 $,
which can be written as $ \re \< \imp g,g\>_\paD \ge 0$. This inequality can be extended by continuity to 
\[ \text{$ \re \< \mulp_\imp^{H^{s} \sto H^{-s}} g,g\>_\paD \ge 0$ for all $g \in H^{s}_\DepaD$.}
\]
This implies the accretivity of $\mulp_{\imp}^{s, -1/2}$. 

Assume that $\mulp_{\imp}^{s, -1/2}$ is accretive. Then $ \re \< \imp g,g\>_\paD \ge 0$ for all 
$g \in H^{s}_\DepaD$. This implies statement (i).

(ii) $\Rightarrow$ (iii) and (ii) $\Rightarrow$ (iv). 
Let $\mulp_{\imp}^{s, -1/2}$ be accretive. Statement (i) implies that $\mulp_{\ove \imp}^{s, -1/2}$ is
also  accretive. By Lemma \ref{l:Mf12-s}, $\mulp_{\imp}^{s, -1/2}$ and $\mulp_{\ove \imp}^{s, -1/2}$ are densely defined. Besides, $(\mulp_{\ove \imp}^{s, -1/2})^\cross = \mulp_{\imp}^{1/2, -s}$. By Proposition \ref{p:tPh59}, there exists a maximal accretive densely defined extension 
$\wh \Zc$ of $\mulp_{\imp}^{s, -1/2}$ and a maximal accretive densely defined restriction $\wt \Zc$ of $(\mulp_{\ove \imp}^{s, -1/2})^\cross = \mulp_{\imp}^{1/2, -s}$. The part (i) $\Leftrightarrow$ (iv) of Theorem \ref{t:DtN-GIBC} applied to $\wt \Zc$ and $\wh \Zc$ completes the proof of statements (iii) and (iv).

(iii) $\Rightarrow$ (i).  Assume that there exists a restriction $\wt \Zc$ of $\mulp_{\imp}^{1/2, -s}$ such that $\A_{\wt \Zc} $ is m-dissipative.  The part (i) $\Leftrightarrow$ (iv) of Theorem \ref{t:DtN-GIBC}
implies that $\wt \Zc$ is maximal accretive and densely defined. Proposition \ref{p:tPh59} applied to 
$\wt \Zc$ implies that $\wt \Zc^\cross$ is also  maximal accretive. However, 
$\wt \Zc^\cross$ is an extension of  $(\mulp_{\imp}^{1/2, -s})^\cross$. By Lemma \ref{l:Mf12-s},
$(\mulp_{\imp}^{1/2, -s})^\cross =\ove{\mulp_{\ove \imp}^{s,-1/2}}$. Thus, $\mulp_{\ove \imp}^{s, -1/2}$ is accretive, and so, statement (i) holds.

(iv) $\Rightarrow$ (ii). Assume that (iv) holds. The part (i) $\Leftrightarrow$ (iv) of Theorem \ref{t:DtN-GIBC} yields that $\wh \Zc $ is accretive, and so is  $\mulp_{\imp}^{s, -1/2}$ (as a restriction of $\wh \Zc $).
\end{proof}

\begin{proof}[Proof of Theorem \ref{t:Macc}]
Let $\imp \in \MM (H^{1/2}_\DepaD , H^{-s}_\DepaD )$ for a certain $s \ge 1/2$ and assume 
 $\int_\paD \re (\imp) |g|^2 \ge 0 $ for all $g \in H^\infty_\DepaD$.
Lemma \ref{l:Mf12-s} and the part (i) $\Lra$ (ii) of Theorem \ref{t:Inter} imply that 
$\mulp_{\imp}^{s, -1/2}$ and $\mulp_{\ove \imp}^{s, -1/2}$ are densely defined and accretive.

The implication (i) $\Rightarrow$ (ii) follows from Lemma \ref{l:Mf12-sAcc}.
Let us prove (ii) $\Rightarrow$ (i). Suppose $\ove{\mulp_\imp^{s,-1/2}} = \mulp_\imp^{1/2,-s}$.
Since $\mulp_{\imp}^{s, -1/2}$ is accretive, so is its closure
$\mulp_\imp^{1/2,-s}$. This completes the proof of the equivalence (i) $\Lra$ (ii).

Now, we put $\Yc = \mulp_\imp^{s,-1/2}$ and show how Proposition \ref{p:EssMD} yields (i) $\Lra$ (iii) and  (i) $\Lra$ (iv).

Assume (i). Then (ii) also holds, and $\Yc$ is accretive, densely defined, and has the closure $\Zc = \ove{\Yc} = \mulp_\imp^{1/2,-s}$. Hence, $\Zc = \mulp_\imp^{1/2,-s}$ is accretive, and by Lemma \ref{l:Mf12-sAcc}, is also maximal accretive and densely defined. Theorem \ref{t:DtN-GIBC}
yields that $\Ac_\Zc$ is m-dissipative. Proposition \ref{p:EssMD} implies that $\Ac_\Yc$ is essentially m-dissipative. Thus, (i) implies (iii) and  (iv).

Assume (iii), i.e., assume that $\ove{\A_\Yc}$ is m-dissipative. Recall that $\Yc = \mulp_\imp^{s,-1/2}$ is accretive and, by Lemma \ref{l:Mf12-s}, closable. Proposition \ref{p:EssMD} implies $\ove{\A_\Yc} = \A_{\ove \Yc} $. By Theorem \ref{t:DtN-GIBC}, $\ove \Yc$ is a densely defined maximal accretive operator and there exist a densely defined accretive operator $Q$ such that $\ove \Yc = Q^\cross$. However, $\ove \Yc = \ove{\mulp_\imp^{s,-1/2}} = (\mulp_{\ove \imp}^{1/2, -s})^\cross $ and $(\ove \Yc)^\cross =  \mulp_{\ove \imp}^{1/2, -s}$  due to Lemma \ref{l:Mf12-s}.  
Hence, $\ove Q = (\ove \Yc)^\cross = \mulp_{\ove \imp}^{1/2, -s}$. Since $Q$ is accretive, $\ove Q = \mulp_{\ove \imp}^{1/2, -s}$ is also  accretive.  Lemma \ref{l:Mf12-sAcc} (i) applied to $\mulp_{\ove \imp}^{1/2, -s}$ implies that $\Zc = \mulp_\imp^{1/2,-s}$ is accretive, i.e., statement (i) is proved.

Assume (iv), i.e., assume that $\A_\Zc$ with $\Zc = \mulp_\imp^{1/2,-s}$ is m-dissipative. By Theorem \ref{t:DtN-GIBC}, $\Zc$ is maximal accretive. This implies statement (i).
\end{proof}

\section{Distributions as multipliers on $\paD$ and m-dissipativity}
\label{s:Gen}

\subsection{Sobolev multipliers on $C^{k-1,1}$-boundaries $\paD$}
\label{s:MultGen}

We assume in this section that, for a bounded domain $\D \subset \RR^d$,
\begin{gather} \label{a:Ck-11}
\text{the boundary $\paD$ is of $C^{k-1,1}$-regularity with $k \in \NN$ such that $k> (d-1)/2 $.}
\end{gather}
Note that, if $d=2$, this assumption is valid for every Lipschitz domain $G$.

The $C^{k-1,1}$-regularity of $\paD$  ensures that that fractional Sobolev spaces $H^s (\paD)$ are well-defined for all $s \in (0,k] $ (see \cite{G11} for the basic definitions).
%, see \cite[Section 1.3.3]{G11}. 
In turn,
the spaces $H^{-s} (\paD)$ of negative regularity can be defined as duals of $H^s (\paD)$ w.r.t.  
$H^0 (\paD) :=L^2 (\paD) $ for  all $s \in (0,k] $, and the sesquilinear pairing $\<\cdot,\cdot\>_\paD$ can be extended to the mutually dual pairs of these spaces together with the definitions of accretive, nonnegative, and $\cross$-adjoint operators analogously to Section \ref{s:m-disGIBC}.

In particular, 
$T:\dom T \subseteq H^s (\paD) \to H^{-s} (\paD)$ with $s \in [-k,k]$ is called accretive (nonnegative) if 
$ \re \<T g,g\>_\paD \ge 0 $ (resp., $\<T g,g\>_\paD \ge 0$) for all $g \in \dom T$.

\begin{prop} \label{p:HkBAl}
Under assumption \eqref{a:Ck-11}, the space 
$H^k (\paD)$ is a Banach algebra in the sense that, for all $f,g \in H^k (\paD)$,
the product $fg$ belongs to $ H^k (\paD)$  and 
\begin{gather} \label{e:BAlg}
\| fg \|_{H^k (\paD)} \lesssim  \| f\|_{H^k (\paD)} \|g \|_{H^k (\paD)} . 
\end{gather}
\end{prop}

\begin{proof}
%In the case $d=2$, the statement follows directly from the
% definition of $H^1 (\paD)$ via the localization.
%Let $d\ge 3$.
If $\Nc$ is a bounded domain in $\RR^{d-1}$ with a smooth enough boundary,
 the combination of Strichartz's results \cite[Theorems 5.2 and 2.1]{S67} implies that 
$ H^k (\Nc) $ is  a Banach algebra.  This and the $C^{k-1,1}$-regularity of $\paD$ allows one to apply the localization in order to get \eqref{e:BAlg}.
  \end{proof}

Using Proposition \ref{p:HkBAl} and employing  $H^k (\paD) $ instead of the  test function space $C^\infty (\RR^n)$, we adapt the definition of the space of Sobolev (i.e., distributional) multipliers  $M[H^{s_1} (\RR^n) \sto H^{s_2} (\RR^n)]$  to the case of a boundary $\paD$ satisfying \eqref{a:Ck-11}.
%Multipliers of such type are called sometimes Sobolev multipliers.

Some differences in the process of the adaptation of the theory of Sobolev multipliers  appear. One obvious  reason is the compactness of $\paD$. Other (mainly technical) issues 
appear due to the use of fractional Sobolev-type spaces on nonsmooth surfaces $\paD$, see the discussions in \cite{T02,GMMM11}.

In the rest of the paper, the shortened notation $H^s $ will be used often  for the spaces $H^s (\paD)$ if this does not lead to ambiguities.

\begin{defn} \label{d:MMs1s2}
Assume \eqref{a:Ck-11}. Let $s_1,s_2 \in [-k,k]$. Let $\vphi \in H^{-k} (\paD)$ be such that 
\begin{gather} \label{e:MGen}
|\<  \vphi , \ove g_1 g_2 \>_\paD |  \lesssim \| g_1 \|_{H^{s_1} (\pa \D)} \| g_2 \|_{H^{s_2} (\pa \D)} \qquad \text{ for all } \quad g_1, g_2  \in H^k (\pa \D).
\end{gather}
\item[(i)] Then we say that $\vphi$ belongs to the space
$\MM [H^{s_1} (\pa \D) \sto H^{-s_2} (\pa \D) ]$ 
of Sobolev multipliers from $H^{s_1} (\pa \D)$ to $H^{-s_2} (\pa \D)$.
\item[(ii)] Let $
\mul_\vphi^{H^{s_1} \sto H^{-s_2}} \in \Lc (H^{s_1}  ,H^{-s_2} )
$
be 
the unique bounded operator from $H^{s_1} (\pa \D)$ to $H^{-s_2} (\pa \D)$ associated 
with the sesquilinear form $\<  \vphi , \overline{g_1} g_2 \>_\paD$, where $g_1, g_2  \in H^k $. If
$
\mul_\vphi^{H^{s_1} \sto H^{-s_2}} \in  \Sf_\infty (H^{s_1} (\paD) ,H^{-s_2} (\paD)) ,
$
we  say that 
$\vphi$ belongs to the corresponding space $ \MM^{\Sf_\infty} [H^{s_1} (\pa \D) \sto H^{-s_2} (\pa \D) ]$ of compact Sobolev multipliers.
\end{defn}

Definition \ref{d:MMs1s2}	uses the dense continuous embeddings $H^k (\pa \D) \imb H^s (\paD)$ for $s \in [-k,k]$. The density in these embeddings follows from the method of local coordinates \cite{T78,G11,GMMM11}.

Using the constant function $\one $ as $g_1 $ or $g_2$ in \eqref{e:MGen}, we get the following  observations.
The operator norm of the space  $\Lc (H^{s_1}, H^{-s_2})$ makes 
$ \MM [H^{s_1}  \sto H^{-s_2}  ]$ a normed space (and so $ \MM^{\Sf_\infty} [H^{s_1}  \sto H^{-s_2}  ]$ is also a normed space with the same operator norm).
Indeed, if $\vphi \in \MM [H^{s_1}  \sto H^{-s_2}  ]$
is not a zero vector of $H^{-k} $, then there exists $g \in H^k $
such that 
$0 \neq \<  \vphi , g \>_\paD
= \< \mul_\vphi^{H^{s_1} \shortto H^{-s_2} }   \one , g \>_\paD .
$

Besides, Definition \ref{d:MMs1s2} and $\one \in H^k $ imply the continuous embedding
\begin{gather} \label{e:Mds1s2s2s1}
\MM [H^{s_1} (\pa \D) \sto H^{-s_2} (\pa \D) ] = \MM [H^{s_2} (\pa \D) \sto H^{-s_1} (\pa \D) ] \imb H^{-\min\{s_1,s_2\}} (\paD). 
\end{gather}

Let $\vphi  \in H^{-s} (\paD)$ with $s > 0$.
The complex conjugate generalized function $\ove \vphi $ is defined by the equalities $ \< \ove \vphi , g\>_\paD =\<   \ove g , \vphi \>_\paD $, where $g$ runs through $H^s$. So,
$\ove \vphi \in H^{-s} $ and $\ove{\ove \vphi}  = \vphi$.
The real and imaginary parts of $\vphi  \in H^{-s} $ are 
$
\re \vphi  := \frac{1}{2} (\vphi+\ove \vphi) \in H^s$ and $\im \vphi := \frac{1}{2i} (\vphi- \ove \vphi) \in H^s$,  respectively. 

For $\vphi \in \MM [H^{s_1} \sto H^{-s_2}]$, Definition \ref{d:MMs1s2} yields $\ove \vphi \in \MM [H^{s_1} \sto H^{-s_2}] =
\MM [H^{s_2} \sto H^{-s_1}]$
and   $(\mul_\vphi^{H^{s_1} \sto H^{-s_2}})^\cross = \mul_{\ove \vphi}^{H^{s_2} \sto H^{-s_1}}$. Hence, $ \MM^{\Sf_\infty} [H^{s_1} \sto H^{-s_2}] =
\MM^{\Sf_\infty} [H^{s_2} \sto H^{-s_1}]$.
Combining \eqref{e:Mds1s2s2s1} with \eqref{e:BAlg}, we see that 
\begin{equation*} %\label{e:MHkH-k}
\MM [H^k (\pa \D) \sto H^{-k} (\pa \D) ] = H^{-k} (\paD) \qquad \text{ (up to equivalence of norms)}.
\end{equation*}
Thus, $\MM [ H^k  \sto H^{-k}  ]$ is a Banach space.

The compact embeddings 
$H^t \imb \imb H^s $ hold for all $-k \le s < t \le k$ and can be proved using an appropriate decomposition of unity (see \cite{T92,T02}). Hence,
\begin{gather} \label{e:MinM12S}
\MM [ H^{s_1} (\pa \D) \sto H^{-s_2} (\pa \D) ] \imb \MM [ H^{t_1} (\pa \D) \sto H^{-t_2} (\pa \D) ]  \quad \text{ if $ s_1 \le t_1 $ and $ s_2 \le t_2$;}\\
\bigcup_{\substack{s_1 \le t_1, \ s_2 \le t_2\\
		s_1+s_2 < t_1 + t_2}}
	 \MM [ H^{s_1} (\pa \D) \sto H^{-s_2} (\pa \D) ] \subseteq \MM^{\Sf_\infty} [H^{t_1} (\pa \D) \sto H^{-t_2} (\pa \D) ] . \label{e:MinMSinf}
\end{gather}

\begin{lem} \label{l:MMBsp}
For arbitrary $s_1,s_2 \in [-k,k]$,   
$\MM [H^{s_1} (\pa \D) \sto H^{-s_2} (\pa \D) ]$ is a Banach space.
Its closed subspace $\MM^{\Sf_\infty} [H^{s_1} (\pa \D) \sto H^{-s_2} (\pa \D) ]$ is a separable Banach space.
\end{lem}
\begin{proof}
Let $\{\vphi_n\}_{n\in \NN}$ be a Cauchy sequence  in $ \MM [H^{s_1}  \sto H^{-s_2}  ]$. Denote by $T$ the limit of operators $\mul_{\vphi_n}^{H^{s_1}  \sto H^{-s_2}}$ in  $\Lc (H^{s_1}, H^{-s_2}) $. We need to prove that $T = \mul_\vphi^{H^{s_1}  \sto H^{-s_2}} $
for a certain $\vphi \in \MM [H^{s_1}  \sto H^{-s_2}  ] $.

We have seen that the Banach space $\MM [ H^k  \sto H^{-k}  ] = H^{-k} $ is the largest of all the normed spaces of multipliers defined in Definition \ref{d:MMs1s2}.
By \eqref{e:MinM12S} and \eqref{e:Mds1s2s2s1}, $\{\vphi_n\}$ has a limit $\vphi$ in $H^{-k} $. One can pass to the limit in Definition \ref{d:MMs1s2} applied to $\vphi_n$ and 
get that $\vphi $ belongs to $\MM [H^{s_1}  \sto H^{-s_2}  ] $ with 
$ \| \vphi\|_{\MM [H^{s_1}  \sto H^{-s_2}  ]}
\le \sup_{n\in \NN} \|\vphi_n\|_{\MM [H^{s_1}  \sto H^{-s_2}  ]} <\infty  .
$
Moreover, 
$ \<Tg_1,g_2\>_\paD = \<\mul_\vphi^{H^{s_1}  \sto H^{-s_2}} g_1,g_2\> $ for all $ g_1,g_2 \in H^k$.
Since  $H^k $ is dense in $H^{s_1}$ and $H^{s_2}$, we get $T=\mul_\vphi^{H^{s_1}  \sto H^{-s_2}}$.  This proves the completeness of  $\MM [H^{s_1} \sto H^{-s_2}  ]$.

By the definition, $\MM^{\Sf_\infty} [H^{s_1}  \sto H^{-s_2} ]$ can be considered as a subspace of \linebreak $\Sf_\infty (H^{s_1},H^{-s_2}) $.
Since $\Sf_\infty (H^{s_1}, H^{-s_2} )$ is closed in $\Lc (H^{s_1}, H^{-s_2}) $, we see that \linebreak
$\MM^{\Sf_\infty} [H^{s_1}  \sto H^{-s_2} ]$ is a closed subspace of $\MM [H^{s_1} \sto H^{-s_2}  ]$.  Since all the Hilbert spaces $H^s$ are separable, $\Sf_\infty (H^{s_1},H^{-s_2}) $ is also separable, and so is
 $\MM^{\Sf_\infty} [H^{s_1}  \sto H^{-s_2} ]$.
\end{proof}

 For $s \in [0,k]$ and $\vth \in (0,1)$, the following interpolation formula is valid 
\begin{gather} \label{e:IntHs-s}
	H^{s(1-2\vth)} (\paD) = [H^s (\paD),H^{-s} (\paD)]_\vth .
	%= (H^s (\paD),H^{-s} (\paD))_{\vth,2} . % \qquad \vth \in [0,1].
\end{gather}
In  Appendix \ref{s:FSp}, we explain how \eqref{e:IntHs-s} can be obtained from the known results (this formula can be found in \cite{T92} for compact  Riemann $C^{\infty}$-manifolds,
but for a $C^{k-1,1}$-boundary $\paD$ the proof of \cite{T92} requires modifications).

For $\vphi \in \MM [H^{s_0} \sto  H^{-s_1}] =\MM [H^{s_1} \sto  H^{-s_0}]$ with $s_0,s_1 \in [-k,k]$,
we interpolate similarly to \eqref{e:MpInt} between the bounded operators  $ \mul_\vphi^{H^{s_0} \sto H^{-s_1}}$ and  $ \mul_\vphi^{H^{s_1} \sto H^{-s_0}}$ using  \eqref{e:IntHs-s}.  This gives  for all $\vth \in [0,1]$ the embedding
\begin{gather}\label{e:MInt}
	\MM [H^{s_0}  \sto  H^{-s_1}] = \MM [H^{s_1}  \sto H^{-s_0}  ]\imb 
	\MM [H^{s_0(1-\vth) +s_1 \vth}  \sto H^{-s_1 (1-\vth) - s_0 \vth}  ] .
\end{gather}

Let $(d-1)/2 < t \le k$. Then  \cite[Theorems 5.2 and 2.1]{S67}  combined with the arguments of the proof of Proposition \ref{p:HkBAl} imply that
\begin{gather} \label{e:BAlgs1}
\| fg \|_{H^{t}} \lesssim  \| f\|_{H^{t}} \|g \|_{H^{t} }, \qquad  f,g \in H^{t} (\paD),
\end{gather}
and that the space $H^t $ is a Banach algebra.
%The equality \eqref{e:MHkH-k} implies 
%\begin{gather} \label{e:N-kN-k}
%\| \mul_\vphi^{H^{k} \shortto H^{-k}} f \|_{H^{-k} (\paD)} \lesssim  \| \vphi \|_{H^{-k} (\paD)} \|f \|_{H^k (\paD)}, \qquad f \in  H^k (\paD).
%\end{gather}
For $f,g_1,g_2 \in  H^t$, \eqref{e:BAlgs1} implies  
\[
|\<  f g_1 , g_2 \>_\paD |\lesssim  \| f g_1 \|_{H^{t} } \| g_2 \|_{H^{-t}}\lesssim \| f \|_{H^{t} } \|g_1 \|_{H^{t} } \|g_2 \|_{H^{-t}}  ,
\]
which due to \eqref{e:Mds1s2s2s1} and Definition \ref{d:MMs1s2} leads to 
$
f \in \MM [H^{t}  \sto H^{t}  ] 
= \MM [H^{-t}  \sto H^{-t} ]
$
and to the two  embeddings 
\begin{gather} \label{e:HkinMH-k}
 H^t (\paD) \imb  \MM [H^{\pm t} (\pa \D) \sto H^{ \pm t} (\pa \D) ]  \quad
% H^t (\paD) \imb \MM [H^{-t} (\pa \D) \sto H^{-t} (\pa \D) ] \\  
\text{ for each } t \in ((d-1)/2,k].
\end{gather}
Formula \eqref{e:MInt} applied to these two embeddings gives 
\begin{gather} \label{e:HkinMH-s}
H^t (\paD) \imb \MM [H^s (\pa \D) \sto H^s (\pa \D) ] , \quad \text{ for  $s \in [-t,t]$ and $ t \in ((d-1)/2,k]$.}
\end{gather}

Now, the results of Maz’ya \& Shaposhnikova \cite{MS04,MS09} and of Neiman-zade \& Shkalikov \cite{NS06} on $M[H^{s_1} (\RR^n) \sto H^{-s_2} (\RR^n)]$ with $|s_2|>n/2$ can be carried over to $\paD$ in the following somewhat modified form.

\begin{thm} \label{t:H-s}
Assume \eqref{a:Ck-11}. Let $ |s_1| \le s_2 \le k$ and $(d-1)/2 < s_2$. Then
\[
H^{-s_1} (\paD) = \MM [H^{s_1} (\pa \D) \sto H^{-s_2} (\pa \D) ] = \MM [H^{s_2} (\pa \D) \sto H^{-s_1} (\pa \D) ] .
\] 
\end{thm}

\begin{proof}
Due to \eqref{e:Mds1s2s2s1}, it is enough to prove 
%\begin{gather} \label{e:H-s1imbM}
$H^{-s_1}  \imb \MM [H^{s_1} \sto H^{-s_2} ]$.
%\end{gather}
By Definition \ref{d:MMs1s2}, formula \eqref{e:HkinMH-s} with $t=s_2$ and $s=s_1$
can be written for $ f \in H^{s_2} $ as 
\begin{gather*} %\label{e:HtinMH-s||} 
%|\<   g_2, f g_1 \>_\paD | = 
|\<  f g , h \>_\paD | \lesssim  \| f \|_{H^{s_2}} \| g \|_{H^{s_1}}\| h \|_{H^{-s_1}}  , \quad \text{  $g \in H^k $, $ h  \in H^k $.}
\end{gather*}
If $f \in H^k$, this estimate can be extended by continuity to all $h \in H^{-s_1} $, i.e.,
\begin{gather*} %\label{e:HtinMH-s||2} 
	|\<   h , f g \>_\paD | 
	 \lesssim \| h \|_{H^{-s_1}} \| f \|_{H^{s_2} } \| g \|_{H^{s_1} }  , \qquad \text{ $h \in H^{-s_1} $, $f \in H^k $, $ g  \in H^k $.}
\end{gather*}
Definition  \ref{d:MMs1s2} applied to the multiplication on $h$  implies $H^{-s_1}  \imb \MM [H^{s_1} \sto H^{-s_2} ]$.
 \end{proof}
 
 %\begin{rem}
 %Roughly speaking, under the regularity assumption \eqref{a:Ck-11} on $\paD$, distributions $\vphi \in H^{-s} (\paD)$ can be multiplied with functions $g \in H^s (\paD)$ producing  
 %$\vphi g \in H^{-t} (\paD)$, where $t$ can be taken arbitrary in the interval $ \max{s, (d-1)/2} < t \le k$. 
% This statement is crucial for the proof of the subsequent Theorem \ref{t:1/2maxdisGen}.
%\end{rem}

Consider now the case where $\max \{|s_1|,|s_2| \} \le \frac{d-1}{2}$. %, but $s_1+s_2 > \frac{d-1}{2}$.

\begin{thm} \label{t:H-sStar}
Suppose \eqref{a:Ck-11}. Let $ |s_1| \le s_2 \le \frac{d-1}{2}$ %$(d-1)/2 $  
 and  $-k \le s_* < s_1+s_2 - \frac{d-1}{2}$.
  %$ - (d-1)/2$. 
Then 
$
H^{-s_*}  (\paD) \imb \MM^{\Sf_\infty} [H^{s_1}  (\paD)  \sto H^{-s_2} (\paD) ] =  \MM^{\Sf_\infty} [H^{s_2} (\paD)  \sto H^{-s_1} (\paD)  ] 
$.
\end{thm}

\begin{proof}
One can take $t_1 :=s_1+s_2 - \frac{d-1}{2} - \ep$ and $t_2 :=
\frac{d-1}{2} + \ep $ with sufficiently small $\ep>0$ such that 
$s_* < t_1$ and 
$  - \frac{d-1}{2} - \ep  \le  t_1 < s_1 \le s_2  < t_2 = \frac{d-1}{2} + \ep \le k$.

Note that $t_1+t_2 = s_1 +s_2$ and that, by Theorem \ref{t:H-s}, we have 
\[
H^{-t_1}  = \MM [H^{t_1}  \sto H^{-t_2}  ] = \MM [H^{t_2}  \sto H^{-t_1} ] .
\] 
% and $s_1 \le s_2 \le \frac{d-1}{2}$, we see that $t_1 <s_1 \le s_2 <t_2$. 
Formula  \eqref{e:MInt}  for an appropriate $\vth \in (0,1)$ 
yields
\[
H^{-t_1}  \imb \MM [H^{s_1}  \sto H^{-s_2}  ] = \MM [H^{s_2}  \sto H^{-s_1}  ].
\]
Since $s_* < t_1$ and $H^{-s_*}  \imb \imb H^{-t_1} $, we get $H^{-s_*} (\paD) \imb \MM^{\Sf_\infty} [H^{s_1}  \sto H^{-s_2} ]$.
 \end{proof}
 
In the critical case $|s_1| \le s_2 = \frac{d-1}{2}$, embedding \eqref{e:Mds1s2s2s1} and Theorem \ref{t:H-sStar} 
imply
\begin{gather} \label{e:(d-1)/2}
\bigcup_{t< s_1 } H^{-t}  \subseteq \MM^{\Sf_\infty} [H^{s_1}  \sto H^{-\frac{d-1}{2}} ] \imb \MM [H^{s_1}  \sto H^{-\frac{d-1}{2}} ] \imb H^{-s_1} .
\end{gather}
%(where one can replace  $[H^{s_1}  \sto H^{-\frac{d-1}{2}} ] $ with $[ H^{\frac{d-1}{2}} \sto H^{-s_1}] $). 
For $d=2$ and $s_1  = 1/2$, this critical case is crucial for Sections \ref{s:2d} and \ref{s:RandAcOp}.

\subsection{Generalized functions as impedance coefficient}
\label{s:GenFImp}

We assume in this subsection that  \eqref{a:Ck-11} is satisfied and that $s,s_1,s_2 \in [-k,k]$.
As it was mentioned in the previous subsection, the shortened notation $H^{s} $ is often used for $H^{ s} (\paD) $. %(if this does not lead to ambiguities).

Let $\vphi \in \MM [H^{s_1}   \sto H^{-s_2}  ]$.
Similarly to Definition \ref{d:wtM}, we introduce the restricted generalized multiplication operator  $\mul_\vphi^{s_1, -s_2}$ by 
\begin{gather}
  \mul_\vphi^{s_1, -s_2} : \dom \mul_\vphi^{s_1,-s_2} \subseteq H^{1/2}  \to H^{-1/2} , \label{e:Mphi1}
\\
\quad \dom   \mul_\vphi^{s_1,-s_2} := \{ u \in H^{1/2}  \cap H^{s_1}  : \ \mul_\vphi^{H^{s_1}  \sto H^{-s_2} } u \in  H^{-1/2} \cap H^{-s_2}   \}, \label{e:Mphi2}  \\
  \mul_\vphi^{s_1, -s_2} u := \mul_\vphi^{H^{s_1} \shortto H^{-s_2}} u  \quad \text{ for all } u \in \dom   \mul_\vphi^{s_1,-s_2}, \label{e:Mphi3}
\end{gather}
where $\mul_\vphi^{H^{s_1} \shortto H^{-s_2}} \in \Lc(H^{s_1},H^{-s_2})$
 is  the multiplication operator of  Definition  \ref{d:MMs1s2}.
 
 Let $\zeta \in H^{-1/2} $.
Theorem \ref{t:H-s} implies that
$
\zeta \in  \MM [H^k  \sto H^{-1/2}  ]  = \MM [H^{1/2}  \sto H^{-k}  ].
$
Hence, formulae \eqref{e:Mphi1}-\eqref{e:Mphi3} define 
the operators  
 $    \mul_\zeta^{k, -1/2} $ and $\mul_\zeta^{1/2, -k}$.

We say that a generalized function $\psi  \in H^{-k} $  is nonnegative and write $\psi \succeq 0$ 
if $\<\psi,g\>_\paD \ge 0$ for all nonnegative $g \in H^k (\paD)$. 
 
Assumption \eqref{a:Ck-11} yields the continuous embedding $H^k \imb C (\paD)$, where 
$C (\paD) $ is the space of $\CC$-valued continuous functions. Hence, every complex Borel measure $\mu$  on $\paD$ with a finite norm $\| \mu \|_{(C (\paD))'}$ generates a unique generalized function  $\vphi_\mu \in H^{-k}$ by the rule $\<\vphi_\mu,g\> = \int_\paD \ove g \ \dd \mu$. We say that $\vphi \in H^{-k}$ is a measure if $\vphi$ is of the form $\vphi_{\mu}$ for a certain $ \mu \in (C (\paD))'$. In this case, $\vphi = \vphi_\mu$ is said to be a nonnegative measure if $\mu$ is a finite  nonnegative Borel measure on $\paD$. 

\begin{thm} \label{t:InterD}
Suppose \eqref{a:Ck-11}. Let $\zeta \in H^{-1/2} (\paD)$.
Then  the following statements are equivalent:
\begin{itemize}
\item[(i)] $\re \zeta$ is a nonnegative measure; 
\item[(ii)] $\re \zeta \succeq 0$ (i.e., $\re \zeta$ is nonnegative as a generalized function);
\item[(iii)] the operator $\mul_{\zeta}^{k, -1/2}$ is accretive;
\item[(iv)] there exists a restriction $\wt \Zc$ of $\mul_\zeta^{1/2, -k}$ such that 
		the acoustic operator $\A_{\wt \Zc} $ is m-dissipative.
%\item[(v)] there exists an extension  $\wh \Zc$ of $\mul_\zeta^{k, -1/2}$ such that 
%		$\A_{\wh \Zc} $ is m-dissipative.
\end{itemize}
 \end{thm}

This theorem is proved in  Section \ref{ss:LemDistr}. 

\begin{thm} \label{t:Hacc}
Suppose \eqref{a:Ck-11}. Assume that $\zeta \in H^{-1/2} (\paD)$ and $\re \zeta \succeq 0$.
Then the following statements are equivalent: 
		\begin{itemize}
		\item[(i)] the operator $\mul_\zeta^{1/2,-k}$ is accretive;
		\item[(ii)] $\mul_\zeta^{k,-1/2}$ is closable and $\ove{\mul_\zeta^{k,-1/2}} = \mul_\zeta^{1/2,-k}$;
				\item[(iii)] the acoustic operator $\A_\Zc$ with $\Zc =  \mul_\zeta^{1/2,-k}$ is m-dissipative.
		\item[(iv)] the acoustic operator $\A_\Yc$ with $\Yc =  \mul_\zeta^{k,-1/2}$ is essentially m-dissipative.
		\end{itemize}
\end{thm}

This theorem is proved in  Section \ref{ss:LemDistr}.

For $T \in \Lc (H^s, H^{-s}) $,  the real and imaginary parts
\[ \textstyle
\re T := \frac{1}{2} (T+T^\cross) \quad  \text{ and } \quad \im T := \frac{1}{2\ii} (T-T^\cross) , \ 
\]
  are also in  $ \Lc (H^s, H^{-s}) $.
Let $\vphi \in \MM [H^s  \sto H^{-s}  ] $ with $s \in [0,k]$.
Then $\ove \vphi $ and $\re \vphi$ also belong to $\MM [H^s \sto H^{-s}  ] $. The operator  $\mul_\vphi^{H^s \sto H^{-s} } $
has the $\cross$-adjoint $\mul_{\ove \vphi}^{H^s \sto H^{-s} } \in \Lc (H^s, H^{-s}) $. This implies 
\begin{gather} \label{e:reMul=re}
	\re \mul_\vphi^{H^{s} \sto H^{-s} } = \mul_{\re \vphi}^{H^{s} \sto H^{-s} } \qquad \text{ for all $\vphi \in \MM [H^{s} (\paD) \sto H^{-s} (\paD) ]$}.
\end{gather}

\begin{rem} \label{r:ReM>0}
Let us assume \eqref{a:Ck-11} and consider the case $\zeta \in \MM [H^{1/2}  \sto H^{-1/2}  ]$. Then 
\[
\text{$\Zc = \mulp_\zeta^{1/2,-k}$ coincides with the bounded operator $\ove{ \mul_\zeta^{k,-1/2}} = \mul_\zeta^{H^{1/2} \sto H^{-1/2}} $.}
\]
Hence, the operator $\Zc$ is accretive if and only if 
\begin{gather} \label{e:ReZ>0}
\text{the operator $\re \mul_\zeta^{H^{1/2} \sto H^{-1/2} }$ is nonnegative}.
\end{gather}
So, condition \eqref{e:ReZ>0} is equivalent to each of the conditions (i)-(iii) of Theorem \ref{t:InterD}. In particular, \eqref{e:ReZ>0} is equivalent to $\re \zeta \succeq 0$.
If the condition $\re \zeta \succeq 0$ is satisfied, Theorem \ref{t:Hacc} implies that the acoustic operator $\A_\Zc$ is m-dissipative.
It follows from \eqref{e:reMul=re} that $\A_\Zc = \A_\Zc^*$ if and only if $\re \zeta = 0$ (in the sense of $H^{-k}$).
\end{rem}

\subsection{Multiplicative positivity and  proofs of  Theorems \ref{t:InterD} and \ref{t:Hacc}.}\label{ss:LemDistr}

In this section we suppose \eqref{a:Ck-11}.
The next lemma gives the equivalence of positivity and multiplicative positivity for $\vphi \in H^k$ (cf. \cite{GV4} for the case of distributions on $\RR^n$).

	\begin{lem}	\label{l:M>0}
		Assume \eqref{a:Ck-11}. Let $\vphi \in \MM [H^s (\paD) \sto H^{-s} (\paD) ]$ with $s \in [0,k]$. Then the operator $\mul_\vphi^{H^s \sto H^{-s} }$ is  nonnegative if and only if $\vphi \succeq 0$.
	\end{lem}
	\begin{proof}
		\emph{Step 1.} Let $\vphi \succeq 0$. Then, for all $f \in H^k$,
		$
		\<\mul_\vphi^{H^s \sto H^{-s} } f , f\>_\paD =  \< \vphi , | f|^2\>_\paD  \ge 0 .
		$
		The inequality  $\<\mul_\vphi^{H^s \sto H^{-s} } f , f\>_\paD   \ge 0$ can be extended to all $f \in H^s$ by continuity. Thus, $\mul_\vphi^{H^s \sto H^{-s} }$ is a nonnegative operator.
		
		\emph{Step 2.} Assume that $\mul_\vphi^{H^s \sto H^{-s} }$ is  nonnegative.
		This implies the following version of the multiplicative positivity property
		 for $\vphi$ (cf. \cite{GV4}):  
		\begin{gather} \label{e:MulPos}
			\< \vphi , |f|^2 \>_\paD \ge 0 \quad \text{ for all $f \in H^k$}.
		\end{gather}	
		Let $g \in H^k$ and $g \ge 0$ a.e. on $\paD$. Then $f_n = (g+\frac{1}{n})^{1/2}$, $n \in \NN$,
		is a sequence of $H^k$-functions such that $|f_n|^2 \to g$ in
		the norm of $H^k$ as $n \to \infty$.
		Applying   \eqref{e:MulPos} to $f_n$ and passing to the limit,
		one gets $\< \vphi , g \>_\paD \ge 0$. That is, 
		$\vphi \succeq 0$.
	\end{proof}

\begin{lem}	\label{l:meas}
		Assume \eqref{a:Ck-11}. Then:
		\item[(i)] The continuous embedding $H^k (\paD) \imb C (\paD)$ is dense.
		\item[(ii)] Every nonnegative $\vphi \in H^{-k} (\paD)$ is a nonnegative measure.
	%	$\<\vphi,f\>_\paD = \int_\paD \ove g \ \dd \mu$.
	\end{lem}
	\begin{proof}
 The embedding $H^k  \imb C (\paD)$ follows from \eqref{a:Ck-11}	and
the definition of $H^k$ via localization. Using a suitable $C^{k-1,1}$-smooth partition of unity on $\paD$ and mollifiers (in local coordinates), one proves the density of $H^k$ in $C (\paD)$ similarly to the case of  compact $C^\infty$-manifolds \cite{T92}. This gives statement (i).

The statement analogous to (ii) is well-known for nonnegative distributions on $\RR^n$, e.g., see the proof in \cite{GV4}. This proof can be  adapted without essential changes to 
$\paD$ using statement (i).
	\end{proof}

Let $\zeta  \in H^{-1/2} $. Then $\ove \zeta  \in H^{-1/2} $ and,
by Theorem \ref{t:H-s}, 
$\zeta$ and $\ove \zeta$ belong to $\MM [H^{1/2}  \sto H^{-k}  ] = \MM [H^k  \sto H^{-1/2} ]$.
Hence, the operators 
$  \mul_\zeta^{k, -1/2} $, $  \mul_\zeta^{1/2, -k}$, $  \mul_{\ove \zeta}^{k, -1/2} $, $  \mul_{\ove \zeta}^{1/2, -k}$ can be defined according to \eqref{e:Mphi1}-\eqref{e:Mphi3}.
Now, we formulate the analogues of the lemmas of Section \ref{ss:PrT:1/2maxdis} indicating the steps where the passage from pointwise to Sobolev multipliers requires certain modifications.

\begin{lem} \label{l:Mpsi*}
(i)  $  \mul_\zeta^{k, -1/2} $ and its extension  $  \mul_\zeta^{1/2, -k}$ are densely defined.
%  operators, and $  \mul_\zeta^{1/2, -k}$ is an extension of $  \mul_\zeta^{k, -1/2} $.
\item[(ii)] $(  \mul_\zeta^{k, -1/2})^\cross =   \mul_{\ove \zeta}^{1/2, -k}$ and 
$\ove{\mul_\zeta^{k, -1/2}} = (\mul_{\ove \zeta}^{1/2, -k})^\cross$.
\end{lem}

\begin{proof}
The proof  can be obtained along the lines of the proof of Lemma \ref{l:Mf12-s}
with obvious changes due to the replacement of the dense linear subspace $H^\infty_\DepaD$ and the integrals by the space $H^k$ and the pairing $\<\cdot,\cdot\>_\paD$, respectively.
\end{proof}

\begin{lem} \label{l:k-1/2Acc}
Assume additionally that  $\mul_\zeta^{1/2,-k}$ is accretive. Then:
\item[(i)] The operators $ \mul_\zeta^{1/2,-k} $ and $\mul_{\ove \zeta}^{1/2,-k}$ are  maximally accretive  and densely defined.
\item[(ii)] $ \ove{ \mul_\zeta^{k,-1/2}} = \mul_\zeta^{1/2, -k} $
\end{lem}

\begin{proof}
Using Lemma \ref{l:Mpsi*}, one proves Lemma \ref{l:k-1/2Acc} similarly to Lemma \ref{l:Mf12-sAcc}.
\end{proof}

\begin{proof}[Proof of Theorem \ref{t:InterD}]
Since $ \zeta \in H^{-1/2} = \MM [H^{1/2} \sto H^{-k}] \imb  \MM [H^k \sto H^{-k}]$,
the bounded operator $\mul_\zeta^{H^{k} \sto H^{-k}} $
is well-defined and coincide with $\mul_\zeta^{k, -1/2}$ as a map.
Hence, the equivalences (i) $\Lra$ (ii) $\Lra$ (iii) of Theorem \ref{t:InterD} follow from Lemmas \ref{l:M>0} and  \ref{l:meas}.
Now, (iii) $\Lra$ (iv)   can be proved similarly to the proof of Theorem \ref{t:Inter}.
\end{proof}

\begin{proof}[Proof of Theorem \ref{t:Hacc}]
Using Theorem \ref{t:InterD}, Lemma \ref{l:Mpsi*}, and Lemma \ref{l:k-1/2Acc},
one proves Theorem \ref{t:Hacc} similarly to Theorem \ref{t:Macc}.
\end{proof}

\section{Friedrichs-type  extensions of nonnegative measures}
\label{s:Pos}

Let $\Zc: \dom \Zc \subseteq H^{1/2} \to H^{-1/2}$ be nonnegative and densely defined.
Then $T_\Zc := (\De_\paD+1)^{-1/4}_{ H^{-1/2} \shortto H^0 } \ \ \Zc \ \ (\De_\paD+1)^{-1/4}_{H^0 \shortto H^{1/2}}  $
is a nonnegative densely defined operator in $L^2 (\paD)$.
Let us denote by $[T_\Zc]_\Fr$ the Friedrichs extension of of $T_\Zc$ (see, e.g., \cite{Kato}).
We  define the Friedrichs-type  extension $[\Zc]_\Fr : \dom [\Zc]_\Fr \subseteq H^{1/2} \to H^{-1/2}$ of  $\Zc$ by 
\[
[\Zc]_\Fr :=  (\De_\paD+1)^{1/4}_{ H^{0} \shortto H^{-1/2} }  [T_\Zc]_\Fr (\De_\paD+1)^{1/4}_{H^{1/2} \shortto H^{0}} .
\]
Consider the acoustic operator $\A_{[\Zc]_\Fr }$ associated with the boundary condition 
\[
[\Zc]_\Fr \ \ga_0 (p) =  \gan (\vbf).
\] 
We  say that the operator $\A_{[\Zc]_\Fr }$ is \emph{an F-extension} of the acoustic operator $\A_\Zc$.

The F-extensions, as well as Krein-type extensions $[\Zc]_\Kr$ and the associated K-extensions, were introduced in \cite{EK22} in the context of Maxwell systems.
Note that it is difficult to apply classical Friedrichs extensions directly to 1st order  acoustic or Maxwell operators because they do not possess necessary sectoriality properties.

\begin{thm} \label{t:F}
Suppose \eqref{a:Ck-11}. Let $\zeta \in H^{-1/2} (\paD)$ be a nonnegative measure.
Put  $\Zc = \mul_{\zeta}^{k, -1/2}$ and $\Yc=\mul_{\zeta}^{1/2, -k}$.
Then:
\begin{itemize}
\item[(i)] The Friedrichs-type  extension $ [\Zc]_\Fr=[\mul_{\zeta}^{k, -1/2}]_\Fr$ is a nonnegative $\cross$-selfadjoint operator that is intermediate between 
$\mul_{\zeta}^{k, -1/2}$ and $\mul_{\zeta}^{1/2, -k}$ in the sense that the associated graphs satisfy $\Gr \Zc \subseteq \Gr [\Zc]_\Fr \subseteq  \Gr \Yc$.
\item[(ii)] The $F$-extension   $\A_{[\Zc]_\Fr }$ is an m-dissipative extension of the acoustic operator $\A_\Zc$.
\end{itemize}
 \end{thm}
 
 \begin{proof}
 The results of Section \ref{ss:LemDistr} ensure that $\Zc$ is nonnegative and densely defined and that 
 $\Yc = \Zc^\cross$. 
 Statement (i) follows from the analogous results for the standard  Friedrichs extension $[T_\Zc]_\Fr$ of  the operator $T_\Zc$ in  
 $L^2 (\paD)$ (see \cite{BHdS20,DM26} for the extension theory of symmetric operators in a Hilbert space).
 Besides, the  extension theory  implies that $[T_\Zc]_\Fr$ is a nonnegative  selfadjoint operator,  and so, is an m-accretive operator in $L^2 (\paD)$. 
The part (i) $\Lra$ (ii) of Theorem \ref{t:DtN-GIBC} implies statement (ii).
 \end{proof}

 \begin{rem} \label{r:LipPos}
In the case where $\paD$ has only the Lipschitz regularity, the Friedrichs-type  extensions can be used similarly. 
Combining the proofs of Theorems \ref{t:F} and \ref{t:Inter}, we obtain the following result.
Assume that $s \ge 1/2$, $\imp \in \MM (H^{1/2}_\DepaD, H^{-s}_\DepaD )$, and $\imp \ge 0$ a.e. on $\paD$.
Then $\Zc = \mulp_\imp^{s,-1/2}$ is a nonnegative densely defined impedance operator that admits 
 the nonnegative Friedrichs-type  extension $[\Zc]_\Fr = ([\Zc]_\Fr)^\cross $, which is intermediate between 
$\mulp_\imp^{s,-1/2}$ and $\mulp_\imp^{1/2,-s}$. This associates with the impedance coefficient $\imp (\cdot)$ the boundary condition $[\Zc]_\Fr \ \ga_0 (p) =  \gan (\vbf)$, which is  m-dissipative  in the sense that the acoustic operator$\A_{[\Zc]_\Fr }$ is m-dissipative.
\end{rem}

\section{Discreteness of acoustic spectra}
\label{s:Disc}

\subsection{M-dissipative acoustic operators with compact resolvents}
\label{s:CompRes}

An eigenvalue $\la$ of an operator $T:\dom T \subseteq \Xf \to\Xf$  is called \emph{isolated} if $\la$ is an isolated point of the spectrum $\si (T)$ of $T$. The \emph{discrete spectrum} $\si_\disc (T)$ of $T$ is the set of isolated eigenvalues of $T$ with finite algebraic multiplicities.
An operator $T$ is said to have a  \emph{purely discrete spectrum} if $\si(T) = \si_\disc (T)$. The closed set $\si_\ess (T) := \si (T) \setminus \si_\disc (T)$ is called the \emph{essential spectrum} of $T$  \cite[Vol.4]{RS}.

Assume that the resolvent set $\rho (T)$ is not empty  and $(T-\la_0 )^{-1} \in \Sf_\infty (\Xf) $ for a certain $\la_0 \in \rho(T)$. 
Then $(T-\la)^{-1} \in \Sf_\infty (\Xf)$ for every $\la \in \rho (T)$, and one says that $T$ is an \emph{operator with compact resolvent}. The following statements are well-known \cite{Kato}:
\begin{gather} \label{e:SiDisc1}
	\text{	if $T$ has a compact resolvent, then $\si (T) = \si_\disc (T)$;}\\
		\text{if $T=T^*$ and $\si(T) = \si_\disc (T) $, then $T$ has  compact resolvent.}  \label{e:SiDisc2}
\end{gather} 

We use the  (orthogonal) decomposition \cite{L13,K26}
\begin{gather*} %\label{e:La}
	\LL^2_{\ab} (\D) = \HH_0 (\Div 0,\D) \oplus \ab^{-1} \gradm H^1 (\D) , 
\end{gather*}
where $ \ab^{-1} \gradm H^1 (\D) := \{ \ab^{-1} \nabla p : p \in H^1 (\D)\}$ is a closed subspace of $\LL^2_{\ab} (\D)$.
Hence, the space $ \LL^2_{\ab,\beta} (\D) = \LL^2_{\ab} (\D) \oplus L^2_\beta (\D) $ 
admits the decomposition
\begin{gather} \label{e:LabDec}
	\LL^2_{\ab,\beta} (\D) = (\HH_0 (\Div 0,\D) \oplus\{0\}) \  \oplus \ \GG_{\ab,\beta} ,  %\  \text{ where }\GG_{\ab,\beta} := \ab^{-1} \gradm H^1 (\D) \oplus L^2_{\beta} (\D).
\end{gather}
where 
$ \GG_{\ab,\beta} := \ab^{-1} \gradm H^1 (\D) \oplus L^2_{\beta} (\D)
$.

Let $\Zc :\dom \Zc \subseteq H^{1/2} (\paD) \to H^{-1/2} (\paD)$ be an arbitrary impedance operator and let $\A_\Zc$ be the acoustic operator  associated with  $\Zc  \ga_0 (p) =  \gan (\vbf)  $. It is show in \cite{K26} that the decomposition \eqref{e:LabDec} reduces  $\A_\Zc$  to  
\begin{gather} \label{e:Red}
\text{the orthogonal sum} \qquad \A_\Zc = 0 \oplus (\A_\Zc |_{\GG_{\ab,\beta}}) ,
\end{gather}
where $\A_\Zc |_{\GG_{\ab,\beta}} : (\dom \A_\Zc \cap  \GG_{\ab,\beta}) \subset \GG_{\ab,\beta} \to \GG_{\ab,\beta}$ is the part of $\A_\Zc$ in its reducing subspace $\GG_{\ab,\beta} $, and $0 $ stands for the zero operator in the reducing subspace $\HH_0 (\Div 0,\D) \oplus \{0\}$.
Actually, similar  orthogonal decompositions w.r.t. \eqref{e:LabDec} hold also for $\A_\Min$, $\A_\Max$, and for every acoustic operator defined by a boundary condition in the sense of Definition \ref{d:BC}, see \cite[Sections 2.4 and 6]{K26}.

The decomposition \eqref{e:Red} shows, in particular, that $\A_\Zc$ has 
an infinite-dimensional kernel 
$
\ker\A_\Zc  \supseteq (\HH_0 (\Div 0,\D) \oplus\{0\}) .
$
Thus, $0 \in \si_\ess (\A_\Zc)$, and so, any of operators $\A_\Zc$ cannot have a compact resolvent.

In order to  apply resolvent compactness to the study of discrete spectra of acoustic operators, \cite{JS21,K26,K25}  consider the part $\A_\Zc |_{\GG_{\ab,\beta}}$ 
of $\A_\Zc$  in the Hilbert space $\GG_{\ab,\beta}$.

It follows from \cite[Theorem 2.15 and Section 2.1]{K26}  that, 
\begin{multline} \label{e:CRAZ}
	\text{for an m-dissipative  $A_\Zc $, the operator $\A_\Zc |_{\GG_{\ab,\beta}}$ has compact resolvent} \\ \text{if and only if $\Zc \in \Sf_\infty (H^{1/2} (\pa \D), H^{-1/2} (\pa \D))$.}
\end{multline}
An equivalent reformulation of this result is \cite[Theorem 1.1]{K25} (for the passage from  \cite[Theorem 1.1]{K25} to  \eqref{e:CRAZ}, see \cite[Sections 2.3, 3.2, and 3.3]{K25}).
Besides, the operator $A_\Zc $ is m-dissipative if and only if $\A_\Zc |_{\GG_{\ab,\beta}}$ is so \cite{K26}.

\subsection{Impedance coefficients generating discrete spectra}
\label{s:InpCoefDisc}

\begin{thm}
\label{t:AZG}
Assume \eqref{a:Ck-11}. Let $\zeta \in \MM [H^{1/2} (\paD) \sto H^{-s} (\paD) ]$ with a certain $s \in [-1/2,k]$. Let $\Zc = \mul_\zeta^{1/2,-s}$. 
%Let $\A_\Zc$ be the acoustic operator associated the boundary condition  $\zeta \ga_0 (p) =  \gan (\vbf)$.
Then:
\item[(i)]  $\A_\Zc |_{\GG_{\ab,\beta}}$ is an m-dissipative operator with compact resolvent if and only if 
\begin{gather} \label{e:MSInfCr}
\text{$\re \zeta $ is a nonnegative measure  \quad and } \quad \zeta \in  \MM^{\Sf_\infty} [H^{1/2} (\pa \D) \sto H^{-1/2} (\pa \D) ] 
 \end{gather}
(note that, due to Lemma \ref{l:meas}, condition \eqref{e:MSInfCr} can be equivalently reformulated as $\re \zeta \succeq 0$  and $\zeta \in  \MM^{\Sf_\infty} [H^{1/2}  \sto H^{-1/2}  ]$).
\item[(ii)] Assume that \eqref{e:MSInfCr} is satisfied. Then $\si (\A_\Zc |_{\GG_{\ab,\beta}}) = \si_\disc (\A_\Zc |_{\GG_{\ab,\beta}})$, and so,
the spectrum of the acoustic operator $\A_\Zc$ consists of isolated eigenvalues, among those only the zero eigenvalue $\la_0 =0$ has an infinite algebraic multiplicity.
\item[(iii)] $\A_\Zc |_{\GG_{\ab,\beta}}$ is a selfadjoint operator with purely discrete spectrum if and only if 
\begin{gather*} %\label{e:MSInfCrRe0}
	\re \zeta =  0 \quad \text{ and } \quad
\zeta \in  \MM^{\Sf_\infty} [H^{1/2} (\paD) \sto H^{-1/2} (\pa \D) ].
\end{gather*}
\end{thm}

\begin{proof}
Statement (i) follows from \eqref{e:CRAZ} and Remark \ref{r:ReM>0}.
Statement (ii) follows from statement (i) and \eqref{e:SiDisc1}.
Statement (iii) is the combination of statement (i), the equivalence $\A_\Zc = \A_\Zc^* \Lra \re \zeta = 0$ of Remark \ref{r:ReM>0}, and statements \eqref{e:SiDisc1}- \eqref{e:SiDisc2}.
\end{proof}

%The following result is helpful for the verification of  $\zeta \in \MM^{\Sf_\infty} %[H^{1/2} \sto H^{-1/2}  ]$  because it provides a sufficient condition in terms of more %accessible spaces $ \MM [H^{s_1}  \sto H^{-s_2} ]$.

Consider the case where $\re \zeta \succeq 0$  and
$\zeta \in  \MM [H^{1/2}  \sto  H^{-s} ]$ with a certain $s < 1/2$. Then the compact embedding $H^{-s}  \imb \imb H^{-1/2} $ implies
\[
\MM [H^{1/2}  \sto H^{-s} ] \imb \MM^{\Sf_\infty} [H^{1/2} \sto H^{1/2} ],
\]
and so, by Theorem \ref{t:AZG},
	$\A_\Zc |_{\GG_{\ab,\beta}}$ is an m-dissipative operator with a compact resolvent.
	
A more general sufficient condition of the same type is the following. 

\begin{thm} \label{t:M<1212}
	Suppose \eqref{a:Ck-11}, $t_1 \ge 0$, $t_2 \ge 0$, and $s_* = \frac{t_1+t_2}{2} <1/2$. Assume that  $\zeta \in \MM [H^{t_1} (\pa \D) \sto H^{-t_2} (\pa \D) ]$. 
 Then:
	\item[(i)] $\zeta \in \MM [H^{s_*} (\pa \D) \sto H^{-s_*} (\pa \D) ] \subseteq \MM^{\Sf_\infty} [H^{1/2} (\pa \D) \sto H^{-1/2} (\pa \D) ]$; in particular,  the operators $\Zc =  \mul_\zeta^{H^{1/2} \sto H^{-1/2}}$ and 
	$\A_\Zc $ are well-defined.
	\item[(ii)] If  $\re \zeta \succeq 0$,  then $\A_\Zc |_{\GG_{\ab,\beta}}$ is an m-dissipative operator with compact resolvent.	
\end{thm}
\begin{proof}
Statement (i) follows from \eqref{e:MInt} and the arguments of the proof of Theorem \ref{t:H-sStar}. 
Since $0\le s_*<1/2$, statement (ii)  follows from Theorem \ref{t:AZG} and the part $\zeta \in  \MM^{\Sf_\infty} [H^{1/2} (\pa \D) \sto H^{-1/2} (\pa \D) ]$ of statement (i). 
\end{proof}

\begin{rem} \label{r:LipDisc}
Consider the case where $\paD$ has only the Lipschitz regularity and restrict ourselves to pointwise multipliers. In this case, a multiplicative positivity result similar to Lemma \ref{l:M>0} seems to be not available yet. A simplified analogues of Theorems \ref{t:AZG} and \ref{t:M<1212} with  similar, but simpler, proofs can be formulated as follows. 
Assume that $s \ge 1/2$, $\imp \in \MM (H^{1/2}_\DepaD, H^{-s}_\DepaD )$, and $\re \imp \ge 0$ a.e. (on $\paD$).
Put $\Zc = \mulp_\imp^{1/2,-s}$. Then $\A_\Zc |_{\GG_{\ab,\beta}}$ is an m-dissipative operator with compact resolvent if and only if 
$\imp \in  \MM^{\Sf_\infty} (H^{1/2} , H^{-1/2})$;
besides, $\A_\Zc |_{\GG_{\ab,\beta}}$ is a selfadjoint operator with purely discrete spectrum if and only if 
$\re \imp = 0$ a.e. and $\imp \in  \MM^{\Sf_\infty} (H^{1/2} , H^{-1/2})$.
The interpolation \eqref{e:MpInt}  can be used similarly to  Theorem \ref{t:M<1212} and leads to the following result. If $\re \imp \ge 0$, $t_1 , t_2 \in [0,1]$, and $ t_1+t_2 <1$, then
$\imp  \in \MM [H^{t_1} , H^{-t_2} ]$ implies that 
$\A_\Zc |_{\GG_{\ab,\beta}}$ is an m-dissipative operator with  compact resolvent.
\end{rem}

\subsection{The case of 2-dimensional acoustic operators}
\label{s:2d}

%We assume in this subsection that $d=2$. The 2-dimensional case allows us to give a more explicit version of preceding results. 

In the case of an arbitrary 2-dimensional Lipschitz domain $\D \subset \RR^2$, the regularity assumption  
\eqref{a:Ck-11} is satisfied automatically. So, all results of  Sections \ref{s:PM+M-d}-\ref{s:InpCoefDisc} are applicable. The classes of pointwise multipliers $\MM^{(\Sf_\infty)} (H^{t}_\DepaD, H^{-s}_\DepaD)$ for $s,t \in [-1,1]$ are subsets of the  classes of Sobolev multipliers $  \MM^{(\Sf_\infty)} [H^{t}  \sto H^{-s} ]$.

%The shortened notation $H^s $ is used for $H^s (\paD) $.

%Recall that by $\A_\Zc$ we denote an acoustic operator associated with an impedance operator $\Zc$ via the boundary condition $\Zc \ga_0 (p) =  \gan (\vbf)$.

For $d=2$, the condition $\zeta \in \MM [H^{t_1} (\pa \D) \sto H^{-t_2} (\pa \D) ]$ in Theorem \ref{t:M<1212} can be replaced by the more explicit condition $\zeta \in H^{-1/2+\ep} (\paD)$.

\begin{thm} \label{t:H-sd=2}
	Let $d=2$. Let $0<s<1/2$ and $\zeta \in H^{-s} (\paD)$.
	Then 
	\[ \text{$\zeta \in   \MM^{\Sf_\infty} [H^{1/2} (\pa \D) \sto H^{-1/2} (\pa \D) ]$, \qquad $\Zc = \mul_\zeta^{H^{1/2} \sto H^{-1/2}}$ is well-defined,}
	\]  
and the following statements hold for the acoustic operator $\A_\Zc$:
\item[(i)] The operator $\A_\Zc$ is m-dissipative if and only if $\re \zeta$ is a nonnegative measure
(or, equivalently, if and only if $\re \zeta \succeq 0$).
Besides, $\re \zeta = 0$ is equivalent to $\A_\Zc = \A_\Zc^*$.
\item[(ii)] If $\re \zeta \succeq 0$,  then $\A_\Zc |_{\GG_{\ab,\beta}}$ is an  m-dissipative operator with compact resolvent.
\end{thm}
\begin{proof}
Let $d=2$, $0<s <1/2$,  and $\zeta \in H^{-s} $. 
By Theorem  \ref{t:H-s}, there exists  $\ep \in (0,1/2]$  such that $s+1/2+\ep <1$ and $\zeta \in \MM [H^s  \sto H^{-1/2-\ep} ]$. Theorem \ref{t:M<1212} applied to $t_1=s$ and $t_2 = 1/2+\ep$ implies $\zeta \in   \MM^{\Sf_\infty} [H^{1/2}  \sto H^{-1/2} ]$. Now, Theorems \ref{t:AZG} and \ref{t:M<1212} imply statements (i) and (ii).
\end{proof}

\begin{ex}[images of 1-d measures]\label{ex:C}
%Following \cite{M15}, 
Let $\mu $ be a 1-dimensional probability measure supported on $\Ic=[0,1]$ and let $\mu$ be of Sobolev dimension $\dim_{\Sr}  (\mu) >0$, e.g., the Cantor probability measures on $\Ic$ corresponding to  symmetric Cantor sets with dissection ratio $ r \in (0,1/2)$ can be taken as $\mu$ (see  \cite{M15} for $\dim_{\Sr}  (\cdot)$ and for Cantor measures).
%$ \mu^r \in H^{-t} (\RR)$ with any $t>  (1-\dim_\Hr \Cc_r)/2$, where $\dim_{\Sr}  \mu_r$ 
%the Sobolev dimension of the measure $\mu^r$ (these inequalities can be seen by the %compasion of  Sections 1, 5.2, and 8.1 in \cite{M15}). 
Then, $ \mu  \in H^{-1/2+\ep} [0,1]$ with a certain $\ep>0$. Let us take any nonempty  connected closed subset $\Ga_1 \subseteq \paD \subset \RR^2$. Using local coordinates, one can construct a  $C^{0,1}$-parametrization $\Phi :[0,1] \to \Ga_1$ that is locally bi-Lipschitz on $[0,1]$ and bijective on $[0,1)$.
The corresponding push-forward $\Phi_\cross \mu \in H^{-1/2+\ep} (\paD)$ is a nonnegative measure on $\paD$. Put $\Zc = c \mul_{\Phi_\cross \mu}$ with a constant $c \in \ove \CC_\rr$.  By Theorem \ref{t:H-sd=2},  $\A_\Zc$  is  an m-dissipative acoustic operator  with compact resolvent and purely discrete spectrum. Besides, $\A_\Zc = \A_\Zc^*$ if and only if $\re c = 0$.
\end{ex}

\section{Random impedance coefficients}
\label{s:Gaussian}

Let $(\Om,\Fc,\PP)$ be a complete probability space.
In what follows, random variables are $(\Om,\Fc)$-measurable functions from $\Om$ to $\CC$. Functions $h:\Om \to \Yf$, $h:\om  \mapsto h_\om$, with values in a separable Banach space $\Yf$ such that $\{ \om \in \Om : h_\om \in B\} \in \Fc$ for any Borel set $B$ in $\Yf$ are called random ($\Yf-$)vectors. Let $\Yf'$ be the dual space of $\Yf$.
Then $h:\Om \to \Yf$ is a random vector if and only if $\ell (h)$ is a random variable for every $\ell \in \Yf'$ \cite{DZ92}. That is, for a separable Banach space $\Yf$,
\begin{gather} \label{e:w=sR}
\text{weak randomness of $h:\Om \to \Yf$ is equivalent to strong randomness.}
\end{gather}
Let $\Xf$ be a separable Hilbert space and let $h$ be a random $\Xf$-vector. By $\EE h \in \Xf$ we denote the expectation of $h$ if it exists in the Gelfand–Pettis sense, see \cite{K85}. In this case, 
the variance of $h$ is $\Var (h):= \EE (\| h - \EE h\|_\Xf^2) $.

\subsection{Fractional Gaussian fields and rough Weyl's asymptotics}
\label{s:FGF}

Let $\D$ be a  Lipschitz domain. 
Recall that $\{Y_n\}_{n \in \NN}$ is an orthonormal basis in $L^2 (\pa \D)$ consisting of real-valued eigenfunctions of the nonnegative Laplace-Beltrami operator $\De_\paD$, where $Y_n$ are numbered in such a way that $\DepaD Y_n =  \mu_n Y_n $ and $\{\mu_n\}_{n \in \NN}$ is a non-decreasing sequence of eigenvalues of $\De_\paD$. We denote by $b_0 \in \NN$ the dimensionality of the space $\ker \DepaD$ of locally constant function on $\paD$. For $1 \le n \le b_0$, we choose $Y_n$ such that they are nonnegative.
% (this is always possible  because the harmonic functions are constant on the connected components of $\paD$).

Let $s \in \RR$. Using the analogy with the $C^\infty$-manifolds and open bounded domains (see \cite{CS25} and the references therein),
we introduce a fractional Gaussian field (FGF) with index $s$ on a Lipschitz boundary $\paD$ as the formal random series 
\begin{gather} \label{e:FGF}
\Xi_s = \sum_{n=b_0+1}^{+\infty} \xi_n \mu_n^{-s/2} Y_n ,
\end{gather}
where $\xi_n$ are i.i.d. $\RR$-valued Gaussian random variables such that $\EE (\xi_n) = 0$ and $\EE (\xi_n^2) =1$ for all $n$.
Note that $\{\mu_n\}_{n=b_0+1}^{+\infty} = \si (\DepaD) \cap (0,+\infty)$.

In order to show that \eqref{e:FGF} strongly converges in an appropriate  $H^s_\DepaD$-sense,
we need the growth rate of the Laplace-Beltrami eigenvalues $\mu_n$. 
The following very rough asymptotics is sufficient for the complete characterization of the  convergence of  \eqref{e:FGF} in  Theorem \ref{t:FGF}.

\begin{thm}[rough Weyl-type asymptotics] \label{t:W}
The eigenvalues of $\DepaD$ satisfy
\begin{equation} \label{e:asymp}
	\mu_n  \asymp  n^{2/(d-1)} \quad \text{ as $n \to +\infty$},
\end{equation}
where $g_n \asymp h_n$ means that there exists constants $c_1$ and $c_2$ such that $0<c_1 \le g_n/h_n \le c_2$ for sufficiently large $n$.
\end{thm}

We give the proof of Theorem \ref{t:W}
in Appendix \ref{a:W}.  

\begin{rem}
For compact $C^{2, \ep}$-surfaces, \eqref{e:asymp} follows from much more accurate Weyl asymptotics obtained in  \cite{Z99,I00} by methods of microlocal analysis. For an arbitrary compact Lipschitz boundary $\paD$, Theorem \ref{t:W} was conjectured in \cite{K26}. 
\end{rem}

In what follows we use for brevity the notation $|h|_t = \| h \|_{H^t_\DepaD}$.

\begin{lem} \label{l:RandSer}
Let $s,t \in \RR$. Let $\al_n$, $n \in \NN$, be a sequence of $\CC$-valued random variables.
Then the following statements are equivalent:
\item[(i)] $\sum_{n=b_0+1}^{+\infty} \al_n \mu_n^{-s/2} Y_n $ a.s. weakly  converges in $H^t_\DepaD$;
\item[(ii)] $ \sum_{n=b_0+1}^{+\infty} \al_n \mu_n^{-s/2} Y_n $ a.s. strongly converges  in  $H^t_\DepaD$ to a certain random \linebreak $H^t_\DepaD$-vector $\Af_s$;
\item[(iii)] $\sum_{n=b_0+1}^{+\infty} n^{2(t-s)/(d-1)} |\al_n|^2 < \infty $ with probability 1.
\end{lem}
\begin{proof}
Note that $Y_n$, $n \in \NN$, form an orthogonal basis in every space $H^t_\DepaD$. 
Hence, statement (i) is equivalent to the a.s. uniform boundedness in $H^t_\DepaD$ of the partial sums $S_N = \sum_{n=b_0+1}^N \al_n \mu_n^{-s/2} Y_n$. 	
Calculating $| S_N|^2_t$ via the equalities
\begin{gather*}
| Y_n|_t = (\|\De_\paD^{t/2} Y_n \|_{L^2}^2 + \| Y_n \|^2_{L^2})^{1/2} = (\mu_n^t +1)^{1/2}  \qquad \text{ for $t\ge 0$}, \\
\text{ and } | Y_n|_t = \max_{|g|_{-t} =1} |\<Y_n , g\>_\paD| =  | Y_n |_{-t}^{-1} = (\mu_n^{-t} +1)^{-1/2} \text{ for $t<0$ and $n>b_0$},
\end{gather*} 
one sees from \eqref{e:asymp} that 
%$| Y_n|^2_t \asymp \mu_n^t \asymp n^{2t/(d-1)}$ and 
$ |\mu_n^{-s/2} Y_n|^2_t \asymp \mu_n^{t-s} \asymp n^{2(t-s)/(d-1)} $ as $n \to + \infty$, and that  statement (iii) is equivalent to $\PP \{\sup_{N\in \NN} | S_N|^2_t <\infty\}= 1$.
This and \eqref{e:w=sR} proves that (iii) is equivalent to  a.s. weak convergence $S_n \to \Af_s$ to
a certain  random $H^t_\DepaD$-vector  $\Af_s$.
The direct calculation of $|\Af_s - S_N|_t$ shows also that (iii) is equivalent  to 
a.s. strong convergence of $ \sum_{n=b_0+1}^{+\infty} \al_n \mu_n^{-s/2} Y_n $.
\end{proof}

The number $\hf_\paD : = s-(d-1)/2$ is an analogue of the Hurst parameter for $\paD$.

\begin{thm}[fractional Gaussian fields] \label{t:FGF}
	(i) For every $t <s-(d-1)/2$, the series \eqref{e:FGF} a.s. strongly  converges in $H^t_\DepaD$ to a random $H^t_\DepaD$-vector, which we denote by $\Xi_s$ and call the fractional Gaussian field  with index $s$ on $\paD$.
\item[(ii)] If $t \ge s-(d-1)/2$, \eqref{e:FGF} diverges weakly in $H^t_\DepaD$  with probability 1 .
\end{thm}

\begin{proof}
Statement (i) follows from the part (ii) $\Lra$ (iii) of Lemma \ref{l:RandSer}  and \cite[Theorem 3.6]{K85}. 

Let us prove (ii). It follows from  \cite[Theorems 3.6 and 3.7]{K85} that  \eqref{e:FGF} diverges weakly  in $H^t_\DepaD$ with probability 1 
if 
\[
 \text{$\textstyle \sum_{n>b_0} \EE (\min \{ 1 , |X_n|_t^2) = \infty$, where $X_n (\om) = \xi_n \mu_n^{-s/2} Y_n$. }
 \]
From Theorem \ref{t:W} and the proof of Lemma \ref{l:RandSer} we see that $ |\mu_n^{-s/2} Y_n|^2_t \asymp  n^{2(t-s)/(d-1)} $ as $n \to \infty$. If $t-s \ge  0$, this implies immediately that $\EE (\min \{ 1 , |X_n|_t^2\})$ is uniformly positive and 
 $ \sum_{n>b_0} \EE (\min \{ 1 , |X_n|_t^2) = \infty$. 
 
 Let  $-1\le p= 2(t-s)/(d-1)  < 0$. For certain constants $N \in \NN$ and $c_1 \in (0,1)$,
\[\textstyle 
\sum_{n>N} \EE (\min \{ 1 , |X_n|_t^2) \ge  \sum_{n>N} \EE (\min \{ 1 , c_1 n^p |\xi_n|^2\}) 
\ge  \sum_{n>N} c_1 n^p \EE |\al_n|^2,
\]
where $\al_n (\om) = \xi_n (\om) $ if $\xi_n (\om) \in [-1,1]$, and $\al_n (\om) = 0$ otherwise.
Then there exists a constant $c_2 >0$ such that $c_2 =  \EE |\al_n|^2$ for all $n$.
Thus, $ \sum_{n>b_0} \EE (\min \{ 1 , |X_n|_t^2) = \infty$. This completes the proof of the theorem.
\end{proof}

\subsection{Random m-dissipative acoustic operators}
\label{s:RandAcOp}

Let $\Xf$ and $\Yf$ be separable Hilbert spaces. The simplest definition of a random operator  $T : \om \mapsto T_\om$ acting from $\Xf$ to $\Yf$ is introduced as follows. 
A function $T:\Om \to \Lc (\Xf,\Yf)$,  $T : \om \mapsto T_\om$, is said to be a  \emph{random bounded operator (from $\Xf$ to $\Yf$)} if $\om \mapsto (T_\om f|g)_{\Yf}$ is a random variable for all $f\in \Xf$ and $g \in \Yf$
(for the case  $\Xf=\Yf$, see, e.g., the monograph  of  Skorohod \cite{S84}). 
In particular, $T: \Om \to \Lc (H^s,H^{-s})$ is a random bounded operator  if and only if $\<Tf,g\>_\paD$ is a random variable for all 
$f,g \in H^s$.
%A random bounded operator $T$ is said to be a \emph{random  contraction} if a.s. $\| T \| \le 1$.
 In these definition and in what follows, $T:\Om \to \Lc (\Xf,\Yf)$ means that $T_\om$ is defined for $\om \in \Om_1$, where $\Om_1$ is a certain subset of $\Om$ of probability 1.

%In this case, 
%\begin{equation} \label{e:|T|}
%\text{$\| T\|$ is random variable with the values in $\overline{\RR}_+ := [0,+\infty)$ (e.g., \cite{BR72,DZ92}).} 
%\end{equation}

%A random bounded operator $T$ is said to be a \emph{random  contraction} if $\| T \| \le 1$ a.s., 
%and is said to be  a \emph{random unitary operator} if $T$ is a.s. a unitary operator.

\begin{lem} \label{l:MRandB}
	Suppose \eqref{a:Ck-11}. Let $\vphi:\om \mapsto \vphi_\om$ be a random $H^{-1/2} (\paD)$-vector such that  a.s.
$\vphi \in \MM [H^{1/2} (\paD)  \sto H^{-1/2} (\paD)]$. Then $\mul_\vphi^{H^{1/2} \sto H^{-1/2} }$ is a random bounded operator from $H^{1/2} (\paD)$ to $H^{-1/2} (\paD)$.
%	\item[(ii)] If $\vphi_\om \in \MM^{\Sf_\infty} [H^{1/2}  \sto H^{-1/2} ]$ a.s., then $\vphi: \Om \to \MM^{\Sf_\infty} [H^{1/2}  \sto H^{-1/2} ]$ is a random vector in the separable Banach space $\MM^{\Sf_\infty} [H^{1/2}  \sto H^{-1/2} ]$.
\end{lem}
\begin{proof}
By \eqref{e:w=sR},  
	$\<\mul_\vphi^{H^{1/2} \sto H^{-1/2} }g_1,g_2\>_\paD = \<\vphi, \ove g_1 g_2\>$ is a random variable for all $g_1,g_2 \in H^k$. Approximating arbitrary $g_1,g_2 \in H^{1/2}$ by $H^k$-functions and using the assumption that 
	$\vphi \in \MM [H^{1/2}  \sto H^{-1/2} ]$ a.s.,
	one gets that $\<\mul_\vphi^{H^{1/2} \sto H^{-1/2} }g_1,g_2\>_\paD $ is a random variable for all $g_1,g_2 \in H^{1/2} $.
\end{proof}

The definition of random bounded operators can be applied to resolvents of randomized unbounded operators as follows (see, e.g., \cite{KM82,K26}). An operator-valued function $A: \om \mapsto A_\om$, $\om \in \Om$,  with the property that $A_\om : \dom A_\om \subseteq \Xf \to \Xf$ is a.s. an m-dissipative operator (selfadjoint operator) is said  to be a \emph{random operator in the resolvent sense}  if $(A - \la)^{-1}$ is a random bounded operator in $\Xf$ for every $\la \in \CC_+$ (resp., for every $\la \in \CC \setminus \RR$). In this case, we say in short that $A$ is random m-dissipative (resp., random selfadjoint). 
It is easy to see \cite{K26} that %For a randomized operator $A: \om \mapsto A_\om$, $\om \in \Om$, 
\begin{gather}\label{e:Rsa}
	\text{if $A$ is random m-dissipative and $A=A^*$ a.s., then $A$ is random selfadjoint.}
\end{gather}

\begin{thm} \label{t:RandM-dis} 
	Let $\Zc:\om \mapsto \Zc_\om$ be a random bounded operator
	from $H^{1/2} (\paD)$ to $H^{-1/2} (\paD)$ such that $\Zc$ is a.s. accretive. Then the acoustic operator $\A_\Zc$ is  random m-dissipative (in the resolvent sense).
\end{thm}

This theorem is proved in Appendix \ref{s:ProofA} as a corollary of the results of \cite{K26} on general random m-dissipative  boundary conditions.

\begin{thm} \label{t:RandSM}
	Suppose \eqref{a:Ck-11}. 	Let $\A_\Zc$ be associated with a randomized impedance boundary condition $\Zc \ga_0 (p) =  \gan (\vbf)$, where 
	 $\Zc = \mul_\zeta^{H^{1/2} \sto H^{-1/2} }$ and $\zeta$ is a random $H^{-k} (\paD)$-vector such that a.s.
	   $\re \zeta \succeq 0$ and  $\zeta \in \MM [H^{1/2} (\paD) \sto H^{-1/2} (\paD)]$.
	Then:
\item[(i)]  $\Zc$ is a random bounded operator and $\Zc$ is a.s. accretive.
\item[(ii)] The operators $\A_\Zc$ and $\A_\Zc |_{\GG_{\ab,\beta}}$ are  random m-dissipative.
\item[(iii)] If $\re \zeta = 0$ a.s., then 
$\A_\Zc$ and $\A_\Zc |_{\GG_{\ab,\beta}}$ are random selfadjoint operators.
\end{thm}
%This theorem is proved in Section \ref{s:ProofRandM-dis}.
\begin{proof}
	Lemma \ref{l:MRandB} implies that $\Zc:\Om \to \Lc (H^{1/2},H^{-1/2})$ is a random bounded operator. Since $\re \zeta \succeq 0$  a.s., Remark \ref{r:ReM>0}
	yields that $\Zc$ is a.s. accretive. This proves statement (i). 	
Statements (ii) 	and (iii) follow from Theorem \ref{t:RandM-dis} and \eqref{e:Rsa}.
\end{proof}

\begin{cor}[randomization by FGFs] \label{c:FGF}
	Let $d=2$, $s>0$, and $c \in \RR$. Let $\A_\Zc$ be associated with $\Zc  \ga_0 (p) =  \gan (\vbf)$, where
	\begin{gather*} 
	\text{$\Zc = \ii c \mul_{\Xi_s}^{H^{1/2} \sto H^{-1/2}} + \Yc$ and the FGF 
	\quad	$\Xi_s$ is defined by  Theorem \ref{t:FGF},}
	% \eqref{e:asymp},} %and  }\\
	\end{gather*}
%	\text{whereas $\Xi_s$ is a FGF defined by  \eqref{e:asymp} and Lemma \ref{l:FGF}, 
whereas	$\Yc: \Om \to \Lc (H^{1/2} (\paD),H^{-1/2} (\paD)) $ is a random bounded operator 
such that $\Yc$ is a.s. accretive.
Then:
\item[(i)] $\A_\Zc$ and $\A_\Zc |_{\GG_{\ab,\beta}}$ are  random m-dissipative operators.
\item[(ii)] $\PP \{\A_\Zc |_{\GG_{\ab,\beta}} \text{has compact resolvent}\} = \PP \{\Yc \in \Sf_\infty (H^{1/2}, H^{-1/2} )\}$
and \\ $\PP \{\si (\A_\Zc |_{\GG_{\ab,\beta}})=\si_\disc (\A_\Zc |_{\GG_{\ab,\beta}})\}  \ge \PP \{\Yc \in \Sf_\infty (H^{1/2}, H^{-1/2} )\}$.
% if $\Yc = \mul_\vphi^{H^{1/2} \sto H^{-1/2}}$
%for a deterministic nonnegative generalized function  $\vphi \in %\MM^{\Sf_\infty} [H^{1/2} \sto H^{-1/2} ] $.
\item[(iii)] If $\Yc^\cross = - \Yc$ a.s. (e.g., if $\Yc = 0$),
then $\A_\Zc$ and $\A_\Zc |_{\GG_{\ab,\beta}}$ are random selfadjoint.
\end{cor}

\begin{proof}
Let $s>0$ and $c \in \RR$. Theorem \ref{t:FGF} 	implies that $\Xi_s$
is a random $H^t (\paD)$-vector for $t \in (-1/2, -1/2+s)$.  Theorem \ref{t:H-sd=2} yields  
$\Xi_s \in \MM^{\Sf_\infty} [H^{1/2} \sto H^{-1/2} ]$, and so the operator $T := \ii c \mul_{\Xi_s}^{H^{1/2} \sto H^{-1/2}}$ is a.s. compact.
Formula \eqref{e:FGF} gives  $\im \Xi_s =0$ a.s., and so,
$\re (\ii c \Xi_s) =0$ with probability 1. 
Using Remark \ref{r:ReM>0},  we conclude that $T $ is a.s. accretive, moreover, a.s. $T = - T^\cross$.

(i) By Theorem \ref{t:RandSM} (i), $T$ is a random bounded operator from from $H^{1/2} $ to $H^{-1/2} $, and so is $\Zc = T +\Yc$.
Theorem \ref{t:RandM-dis}  implies that $\A_\Zc$ is  random m-dissipative operator. The  decomposition \eqref{e:Red} 
yields that so is also $\A_\Zc |_{\GG_{\ab,\beta}}$.

(ii) Clearly the events $\{\Yc \in \Sf_\infty (H^{1/2}, H^{-1/2}) \}$ and
$\{\Zc \in \Sf_\infty (H^{1/2}, H^{-1/2}) \}$ coincide up to a set of probability $0$.
Theorem \ref{t:AZG} implies statement (ii).

(iii) Since $T = - T^\cross$, statement (iii) follows Theorem \ref{t:DtN-GIBC} and \eqref{e:Rsa}.
\end{proof}

Taking in Corollary \ref{c:FGF} a random bounded operator $\Yc= \mul_{\vphi}^{H^{1/2} \sto H^{-1/2}}$ with a random $H^{-1/2}$-vector $\vphi$ satisfying $\re \vphi \succeq 0$,  $\vphi \in \MM^{\Sf_\infty} [H^{1/2}  \sto H^{-1/2} ]$, 
and additionally $\re \vphi \neq 0$ a.s., we get a model $\A_\Zc$ for a random lossy  resonator in 
the sense that $\A_\Zc$ is a random m-dissipative acoustic operator such that $\si (\A_\Zc)\setminus\{0\}$ consists of isolated eigenvalues of finite algebraic multiplicity  and additionally $\A_\Zc \neq \A_\Zc^*$ a.s. (i.e., a.s. there exists leakage of energy into the 
surrounding). An explicit example of this type is produced if  $\vphi$ is any of deterministic measures $\Phi_\cross \mu$ of 
Example \ref{ex:C}.

Another  example is provided by Theorem \ref{t:ExFGF}.

\begin{rem} \label{ex:ExFGF}
Let us explain how Theorem \ref{t:ExFGF} follows from the results of this section. 
Since $c \in \RR$ and since $Y_n$ are nonnegative for $1 \le n \le b_0$, we see that  $\re \zeta \succeq 0$ and $\re \zeta \neq 0$ with probability $1$.
Lemma \ref{l:RandSer} and Theorem \ref{t:FGF} imply that $\zeta$ is a random $H^{-1/2+\ep}$-vector for a  small enough $\ep>0$. 
By Theorem  \ref{t:H-sd=2},
$H^{-1/2+\ep} \imb \MM^{\Sf_\infty} [H^{1/2} \sto H^{-1/2}]$, and so, $\zeta $ is a random $\MM^{\Sf_\infty} [H^{1/2} \sto  H^{-1/2}]$-vector (note that $\MM^{\Sf_\infty} [H^{1/2} \sto  H^{-1/2}]$ is a separable Banach space by Lemma \ref{l:MMBsp}). Thus, 
Theorem \ref{t:H-sd=2}
and Corollary \ref{c:FGF} imply statements (i)-(iii) of Theorem \ref{t:ExFGF}.
\end{rem}

\begin{ex} \label{e:Ess}
Another class of examples for Corollary \ref{c:FGF}, which involves a  randomization by 
$\ii c \Xi_s$ of a nonlocal deterministic operator $\Yc$, 
emerges in the case where $\Yc \in \Lc (H^{1/2},H^{-1/2})$ is an extension by continuity of 
$c_2 (\De_\paD+c_1)^{t/2}$ 
with $ t\le 1$, $c_1>0$, and $c_2>0$. In the case $t<1$, $\A_\Zc |_{\GG_{\ab,\beta}}$ is a  random m-dissipative operator with a.s. compact resolvent and purely discrete spectrum.
Besides, $\A_\Zc \neq \A_\Zc^*$ with probability 1. In the case $t=1$, the operator $\Yc$ does not belong to 
$\Sf_\infty (H^{1/2},H^{-1/2})$, and so,  \eqref{e:CRAZ} yields that the resolvent of 
$\A_\Zc |_{\GG_{\ab,\beta}}$ is not compact a.s. (we conjecture that $\si_\ess (\A_\Zc |_{\GG_{\ab,\beta}}) \neq \varnothing$ a.s. in this case). If $\Yc = \pm \ii  c_3 (\De_\paD+ c_1 )^{1/2}_{H^{1/2} \shortto H^{-1/2}}$ with $c_3>0$ and $c_1 \ge 0$, then $\A_\Zc |_{\GG_{\ab,\beta}}$ is a random selfadjoint operator such that a.s. $\si_\ess (\A_\Zc |_{\GG_{\ab,\beta}}) \neq \varnothing$.
\end{ex}

\section{Conclusion, discussion, and additional remarks}
\label{s:dis}

\subsection{Conclusion and discussion}

We have shown in Theorem \ref{t:H-sd=2} that, for the 2-dimensional case $d=2$ and arbitrarily small $\ep>0$,
generalized functions $\zeta \in H^{-1/2+\ep} (\paD)$ such that $\re \zeta$ is a positive measure
generate m-dissipative impedance boundary conditions $\zeta  \ga_0 (p) =  \gan (\vbf)$, and that $\si (\Ac_\Zc) \setminus \{0\} = \si_\disc (\Ac_\Zc)$ for the associated acoustic operators $\Ac_\Zc$.
This allowed us to use fractional Gaussian fields $\Xi_s$
with arbitrarily small $s>0$ for the randomization of $\zeta$. However, this does not cover the limiting case of the white noise $\Xi_0$, which, in our opinion, may require additional tools (e.g., a renormalization \cite{GUZ20}).

The cases $d \ge 3$ and $\zeta \in H^{-1/2} (\paD)$ are treated to some extent by Theorems \ref{t:Inter}, \ref{t:Macc}, \ref{t:InterD},
and \ref{t:Hacc}. Theorems \ref{t:Inter} and \ref{t:InterD} require only nonnegativity of the real part $\re \zeta$ of the impedance coefficient and establish the existence of  certain restrictions $\wt \Zc$ of the multiplication operator $\mul_\zeta^{1/2,-k}$
that generate an m-dissipative boundary condition  $\wt \Zc  \ga_0 (p) =  \gan (\vbf)$. However, there is no control on the uniqueness of such an implicit restricted multiplication operator $\wt \Zc$.   Theorems  \ref{t:Macc} and \ref{t:Hacc} work with explicit multiplication operators, but shift the difficulty to the verification of the condition $\ove{\mul_\zeta^{k,-1/2}} = \mul_\zeta^{1/2,-k}$,
which is easy only in the case $\zeta \in  \MM [H^{1/2} (\paD)  \sto H^{-1/2} (\paD) ]$ (see Remarks  \ref{r:ReM>0} and \ref{r:Mp12-12}).

The verification of the condition $\zeta \in  \MM [H^{1/2} (\paD)  \sto H^{-1/2} (\paD) ]$ is also a nontrivial task,
but there is a hope that  it can be addressed by the adaptation to Lipschitz boundaries of the $H^{1/2} (\RR^n)$-capacity approach to $\MM [H^{1/2} (\RR^n)  \sto H^{-1/2} (\RR^n) ]$, which was developed by Maz'ya \& Verbitsky \cite{MV95,MV04}.

In Section \ref{s:Pos}, we have suggested another way to construct explicit m-dissipative boundary conditions for positive measures $\zeta \in H^{-1/2} (\paD)$ using Friedrichs-type extensions. This method can be extended to distributions $\zeta$ with certain sectoriality properties and to Krein-type extensions of \cite{EK22}. 
%(e.g, with the property that the operator 
 %$ (\De_\paD+1)^{1/4}_{ H^{0} \shortto H^{-1/2} }  \mul_\zeta^{k,-1/2} (\De_\paD+1)^{1/4}_{H^{1/2} \shortto H^{0}}$  in $L^2 (\paD)$) and to Krein-type extensions of \cite{EK22}.

For $\paD$ of Lipschitz regularity,  one more (rougher) way to obtain  explicit conditions of m-dissipativity is considered briefly in the next subsection.

\subsection{$L^q$-impedance coefficients and their Friedrichs-type extensions}
\label{s:Lq}

Consider the general case where $d \ge 2$ and $\D \subset \RR^d$ is a Lipschitz domain. The following sufficient conditions for the verification of $f \in \MM (H^{s_1} , H^{-s_2} ) $ can be obtained from 
the three-function Hölder inequality and the Sobolev-type embeddings, see the proof in Appendix \ref{s:FSp}.

\begin{lem} \label{l:LqMs}
	Let $s_j \in [0,1]$ and $\vka_j = \frac{2s_j}{d-1}$ for $j=1,2$. Then
	the  embedding
%	\begin{gather*} \label{e:LqImbM}
$	\quad	
L^q (\paD) \imb \MM (H^{s_1} (\paD), H^{-s_2} (\paD) ) 
\quad $
%	\end{gather*}
	takes place in each of the following cases: 
	\begin{itemize}
		\item[(i)] $\vka_1 <1$, $\vka_2<1$, and $q \ge \frac{2}{\vka_1+\vka_2} = \frac{d-1}{s_1+s_2}$ (where $\frac{d-1}{0} = +\infty$);
		\item[(ii)] $\vka_1 \le 1$, $\vka_2 \le 1$, $(\vka_1 -1) (\vka_2 -1) = 0$, and 
		$q > \frac{d-1}{s_1+s_2}$;
		\item[(iii)] $\min\{\vka_1,\vka_2\} < 1$, $\max \{\vka_1, \vka_2\} > 1$,  and 
		$
		q =  \frac{2}{1+ \min\{\vka_1,\vka_2\}}$;
		% = \frac{2(d-1)}{d-1+2 \min\{s_1,s_2\}} 		$;
		%(or, similarly, in the case  $\vka_1 > 1$, $\vka_2 < 1$,   and 
		%$ 		q =  \frac{2}{1+ \vka_2} = \frac{2(d-1)}{d-1+2s_2} 	$);
		\item[(iv)] $\vka_1 +\vka_2 >2$, $(\vka_1 -1) (\vka_2 -1) = 0$, and 
		$
		q > 1 ;
		$
		\item[(v)] $\vka_1 >1 $, $\vka_2>1$,  and 
		$
		q = 1 
		$ (since, by definition, $\MM (H^{s_1} , H^{-s_2} ) \subseteq L^1 (\paD)$, we have in this case  $ L^1 (\paD) = \MM (H^{s_1} , H^{-s_2}  ) $).
	\end{itemize}
%	Moreover,  
%	\begin{gather*} \label{e:Lq=M}
%		\text{if $\vka_1 >1 $ and $\vka_2>1$, then $ L^1 (\paD) = \MM (H^{s_1} (\pa \D), H^{-s_2} (\pa \D) ) $}.
%	\end{gather*}
\end{lem}

This lemma and Remark \ref{r:Mp12-12} on the case $s_1=s_2 =1/2$ immediately imply the following sufficient condition in therms of $L^q$-spaces.

\begin{cor} \label{c:M-disM12}
Let $d \ge 3$ and $\imp \in L^{d-1} (\pa \D, \overline{\CC}_\rr)$. Let $\Zc = \mulp_\imp^{H^{1/2} \sto H^{-1/2}}$. Then the acoustic operator $\A_\Zc$   associated with $\Zc  \ga_0 (p) =  \gan (\vbf)$ is m-dissipative.
\end{cor}

In the case $d=2$, the parameters $s_1=s_2 =1/2$ are in the critical case (ii) of Lemma \ref{l:LqMs}. Taking $ s_1 =  s_2 = 1/2 - \ep$ with arbitrarily small $\ep>0$, one gets that for  $\imp \in L^{q} (\pa \D, \overline{\CC}_\rr)$ with $q>d-1$, the operator $\A_\Zc$ is m-dissipative.
Remark \ref{r:LipDisc} yields the following more detailed result (which is essentially obtained in \cite{K25} with somewhat different, but equivalent,  formulation): for any $d \ge 2$ and $\imp \in L^q (\pa \D, \overline{\CC}_\rr)$ with $q > d-1$, the acoustic operator  $\A_\Zc$ with $\Zc = \mulp_\imp^{H^{1/2} \sto H^{-1/2}}$ is m-dissipative and $\A_\Zc |_{\GG_{\ab,\beta}}$ has  compact resolvent and purely discrete spectrum. 

For nonnegative impedance coefficients $\imp$, the $L^{d-1} $-condition can be substantially relaxed by means of F-extensions of Section \ref{s:Pos}.

\begin{cor} \label{c:LipPos}
Let $\imp \in L^{q} (\paD, \overline{\RR}_+)$ with $q>4/3$ in the case   $d=3$,  or $q\ge 2(d-1)/3$ in the case 
$d \ge 4$. Let $\Zc = \mulp_\imp^{1,-1/2}$. 
Then the operator  $\Zc$ admits 
 a Friedrichs-type  extension $[\Zc]_\Fr$, and the acoustic operator $\A_{[\Zc]_\Fr }$ is m-dissipative. 
\end{cor}  

\begin{proof}
By Lemma \ref{l:LqMs}, the assumptions of the corollary imply that $L^q (\paD) \imb \MM (H^{1} , H^{-1/2}  )$.  Since  
$\imp \in L^{q} (\paD)$ and $\imp \ge 0$ a.e.,  Remark \ref{r:LipPos} completes the proof.
\end{proof}

\begin{appendices}

\section{Sobolev spaces and Laplacians  on non-smooth surfaces}
\label{a:Analysis}

\subsection{Proof of Theorem \ref{t:W} on  rough Weyl's asymptotics}
\label{a:W}

Let $\D \subset \RR^d$ be a Lipschitz  domain, i.e.,  a bounded connected open set $\D \neq \varnothing$ with a $C^{0,1}$-boundary  $\Ga=\paD$.
The $C^{0,1}$-regularity of $\Ga$  means \cite{G11,GMMM11} that, for every $x \in \Ga$, there exists an open neighborhood $\Oc$ of $x $ and new coordinates $y= (y_1, \dots,y_d)$ obtained from the original
one via a rigid motion such that:
\begin{itemize}
	\item[(A)]  in the new coordinates,  
	$\Oc= \{y : -\ell_j < y_j < \ell_j, 1\le j \le d\}$ is a nonempty hyperrectangle;
	\item[(B)] there exists a $C^{0,1}$-function $\psi:\Oc' \to [-\ell_d/2,\ell_d/2] $ that is defined on $\Oc' = \{y' \in \RR^{d-1} : -\ell_j < y_j < \ell_j, 1\le j \le d-1\}$ and has  the properties 
	\[ \D \cap \Oc = \{(y', y_d) \in \Oc: y_d < \psi (y') \}
	\text{ and } \Ga  \cap \Oc = \{(y', y_d) \in \Oc: y_d = \psi (y') \}.
	\]
\end{itemize}
In these settings, $\Psi (y') = (y',\psi (y'))$, $y \in \Oc$, defines a bi-Lipschitz mapping $\Psi $ from $\Oc'$ onto $\Psi  \Oc' = \Ga  \cap \Oc $.

Since $\Ga$ is compact, one can find a finite open cover  $  \{\Oc_m\}_{m=1}^M$ of $\Ga$
by hyperrectangles $\Oc_m =\{y^m : -\ell_j^m < y^m_j < \ell_j^m, 1\le j \le d\}$ satisfying conditions (A) and (B).
Here  $y^m= (y^m_1, \dots,y^m_d)$ denotes the system of coordinates for $\Oc^m$ as in (A)-(B). We use the similar indexing for 
$\psi_m$, $\Oc'_m$, $\Psi_m$, etc., and put $\Ga_m = \Ga \cap \Oc_m$. Then 
$\wh \Ga = \{\Ga_m\}_{m=1}^M$ is a cover of $\Ga$ by relatively open subsets.

For $S \subset \RR^n$ and $1/2<\al \le 1$, we put $\al S:= \{ \al x : x \in S\}$ 
and consider the subsets $\al \Oc'_m $ of $\Oc'_m$, which are bijectively mapped by $\Psi_m $ onto $\Ga^\al_m := (\al \Oc_m) \cap \Ga$.  

%Whenever one of such neighbourhoods $\Oc^m$ is fixed, we drop the index $m$ in $\Oc_m$, $y^m$, $\Psi_m$, etc.
Let   $\wt \Ga$ be a (relatively) open subset of  $\Ga$ such that its boundary $\pa \wt \Ga$ (w.r.t. $\Ga$) is a $C^{0,1}$-submanifold. In what follows $\wt \Ga$ is typically one of the sets $\Ga_m^\al$, but the case $\wt \Ga = \Ga$ is not excluded.
The space $H^1 (\wt \Ga)$ can be defined in the standard way by the localization.
We denote by $ \LL^2_\tr (\wt \Ga)$ the Hilbert space of tangential vector-fields on $\wt \Ga$ that are $L^2$-summable (w.r.t. the surface measure $\dd \Si$ of $\Ga$). The tangential gradient operator $\Na_{\wt \Ga}$ on $\wt \Ga$ can be defined such that 
$\Na_{\wt \Ga} \in \Lc (H^1 (\wt \Ga),\LL^2_\tr (\wt \Ga))$. 
In the sequel, it is convenient for us to fix $\|u\|_{H^1 (\wt \Ga)} = (\|u \|_{\LL^2 (\wt \Ga) }^2 + \|\Na_{\wt \Ga} u \|_{\LL^2 (\wt \Ga) }^2)^{1/2} $ as the norm in $H^1 (\wt \Ga)$.
Note that $H^1 (\wt \Ga)$ is a Hilbert space with this norm, and that for $\wt \Ga = \Ga$ this norm is equivalent to the norms of Section \ref{s:Ac1}.

The Hilbert space 
$
H^1_0 (\wt \Ga) 
$
is the closure in $H^1 (\wt \Ga)$ of the set of all $f \in H^1 (\wt \Ga)$ supported on a compact subsets of $\wt \Ga$.  A function $ f \in H^1_0 (\wt \Ga)$ extended to the whole $\Ga$ by  $0$ belongs to $H^1 (\Ga)$ and has the same norm. %In other words,  
%an isometric embedding  $H^1_0 (\wt \Ga) \imb H^1 (\Ga)$ takes place.

The quadratic form 
$
Q^{\wt \Ga} [u] = \|\Na_{\wt \Ga} u\|^2_{\LL^2_\tr (\wt \Ga)} 
$
defined on $\dom Q^{\wt \Ga} = H^1 (\wt \Ga)$ is nonnegative and closed in $L^2 (\wt \Ga)$.
The  quadratic form $Q_0^{\wt \Ga}$ defined as the restriction of $Q^{\wt \Ga}$ to 
$\dom Q_0^{\wt \Ga} = H^1_0 (\wt \Ga)$ is also nonnegative and closed in $L^2 (\wt \Ga)$.

The role of the forms $Q^{\wt \Ga}$ and $Q^{\wt \Ga}_0$ is that they define in $L^2 (\wt \Ga)$ two
nonnegative selfadjoint  Laplace-Beltrami operators, $\De_{\wt \Ga}^\Nr$ and $\De_{\wt \Ga}^\Dr$, that correspond to Neumann and Dirichlet boundary conditions, resp.,
(for the basics concerning closed quadratic forms, see, e.g., \cite{Kato,GMMM11}). 
In particular, 
the  Laplace-Beltrami operator $\De_\Ga  : \dom \De_\Ga \subset L^2 (\Ga) \to L^2 (\Ga)$ on the whole  $\Ga$
is the nonnegative selfadjoint  operator in $L^2 (\Ga) $ 
associated with the form $ Q^\Ga$ \cite{GMMM11}.

Let a selfadjoint operator $T$ be nonnegative and such that $\si (T) = \si_\disc (T)$. Under the ordered multiset  of eigenvalues of $T$, we understand a non-decreasing sequence $\{\la_n\}_{n \in \NN}$ consisting of all eigenvalues of $T$, where each eigenvalue is repeated according to its multiplicity. 
By  $\Nc (\la;T) := \sum_{\la_n \le \la } 1$ we denote the standard eigenvalue counting function for $T$.
Let $\{\mu_n\}_{n\in \NN}$  be the (ordered) multiset of eigenvalues of  $\De_\Ga$.
By the min-max principle, 
\[
\mu_n = \min_{\substack{E \subset H^1 (\Ga)\\\dim E = n}} \max_{\substack{u \in E\\ u \ne 0}}  \frac{Q^\Ga [u] }{\|u\|^2_{L^2 (\Ga)}} ,
\qquad n \in \NN.
\]
Here $\dim E$ is the dimensionality of $E$, and $E$ runs through all linear subspaces of the prescribed dimensionality.
Let us prove the asymptotics \eqref{e:asymp}.

\emph{Step 1. The estimate from above.} Let us put $\wt \Ga = \Ga_m^{\al}$ for certain fixed parameters $\al \in (1/2,1)$ and $m\in \{1,\dots,M\}$.
Then the numbers 
\[
\la_n^{\Dr} = \la_n^{\Dr,m,\al} = \min_{\substack{E \subset H^1_0 (\wt \Ga) \\ \dim E = n}} \max_{\substack{u \in E\\ u \ne 0}}  \frac{Q^{\wt \Ga} [u] }{\|u\|^2_{L^2 (\wt \Ga)}} 
\]
have the property $\mu_n \le \la_n^\Dr$ for all $n \in \NN$. This follows from the isometric embedding 
$H^1_0 (\wt \Ga) \imb H^1 (\Ga)$.
Writing the  form $Q^{\wt \Ga}$  in local coordinates over $\al \Oc'_m$ (see \cite{GMMM11}) one checks that the coefficient emerging from the metric tensor-field  is  uniformly positive and essentially bounded.
Hence,  for a certain constant  $C_+>0$,  one has the inequalities 
$\la_n^\Dr \le C^1_+ \vka_n^{\Dr} , \quad n \in \NN, 
$  
where $\{\vka_n^{\Dr}\}_{n \in \NN}$ is the multiset of eigenvalues of  the (nonnegative) Dirichlet Laplacian  operator in the flat hyperrectangle $\al \Oc'_m$.
Using the classical Weyl asymptotics (or by direct calculation of the Laplacian's eigenfunctions in the hyperrectangle),
we obtain $\vka_n^{\Dr} \lesssim n^{2/(d-1)}$ as $n\to+\infty$, and so
\[
\mu_n \le \la_n^\Dr \le C^1_+ \vka_n^{\Dr} \lesssim n^{2/(d-1)} \qquad \text{for sufficiently large $n$}. 
\]
In other words, a single restricted set $\Ga_m^{\al}$ is enough in order to obtain the rough Weyl-type  estimate from above.

\emph{Step 2. For the lower bound,  we estimate 
on  the whole cover $\wh \Ga =\{\Ga_m\}_{m=1}^M$ of $\Ga$ and use all the quadratic forms $Q^{\Ga_m}$ that correspond to Neumann Laplacians.} Since the sets $\Ga_m$ overlap,
this procedure requires additional preparations.

Let us define the Hilbert space $L^2_{\wh \Ga} = \bigoplus_{m=1}^M L^2 (\Ga_m) $ of m-tuples $\wh u = (u_1, \dots, u_M)$ consisting of functions $u_m \in L^2 (\Ga_j)$, $m=1$, \dots, $M$, chosen independently of each other. Similarly, we put $ H^1_{\wh \Ga} = \bigoplus_{m=1}^M H^1 (\Ga_m) $.
In $L^2_{\wh \Ga}$, the closed form 
\[
\wh u = (u_1, \dots, u_M) \mapsto \sum_{m=1}^M Q^{\Ga_m} [u_m], 
\qquad \wh u \in H^1_{\wh \Ga} ,
\] 
 defines the nonnegative selfadjoint operator 
$
\De^\Nr_{\wh \Ga} = \De_{\Ga_1}^\Nr \oplus \dots \oplus  \De_{\Ga_M}^\Nr $.
%\De^\Nr_{\wh \Ga} : \dom \De_{\wh \Ga} \subset  L^2_{\wh \Ga} \to L^2_{\wh \Ga} .

Let $\{\la_n^\Nr\}_{n \in \NN}$ be the multi-set of eigenvalues of $\De_{\wh \Ga}^\Nr$. It is clear that $\{\la_n^\Nr\}_{n \in \NN}$ is a (non-decreasingly reordered) union of the multisets of eigenvalues for the operators  $\De_{\Ga_m}^\Nr$.
For arbitrary $u \in H^1 (\Ga)$, the restriction $u\uph_{\Ga_m}$ of $u$ to $\Ga_m$ belongs to $H^1 (\Ga_m)$ for every $m$. %Hence the operator 
%\[
%  P: u \mapsto (u\uph _{\Ga_1} , \dots, u\uph _{\Ga_M}), \qquad u \in \dom P = H^1 (\Ga),
%\]
%maps $H^1 (\Ga)$ to a subspace of $ H^1_{\wh \Ga} $. 
This allows us to combine the inequality 
\[
\frac{\sum_{m=1}^M Q^{\Ga_m} (u) }{\sum_{m=1}^M \|u\|_{L^2 (\Ga_m)}^2} \le M \frac{ Q^\Ga (u)}{\|u\|^2_{L^2 (\Ga)}} , \qquad u \in H^1 (\Ga) \setminus \{0\},
\]
with the min-max principle in order to obtain 
$
\la_n^\Nr \ \le \ M \ \mu_n , \qquad n \in \NN.
$

Let us estimate the asymptotics  of eigenvalues $\la_n^\Nr$ of $\De_{\wh \Ga}^\Nr$. 
Passing as before to the local coordinates for each of quadratic forms $Q^{\Ga_m}$ and estimating their uniformly positive coefficients from below by positive constants, one obtains from the Weyl law for the Neumann Laplacian $\De_{\Oc'_m}^\Nr$ on  a hyperrectangle $\Oc'_m$ that 
\[
\Nc (\la;\De_{\Ga_m}^\Nr) \lesssim \Nc (\la; \De_{\Oc'}^\Nr) \lesssim \la^{(d-1)/2} \quad \text{ for sufficiently large } \la >0 .
\]
Since $\Nc (\la; \De_{\wh \Ga}^\Nr) = \sum_{m=1}^M \Nc (\la;\De_{\Ga_m}^\Nr)$,
we get the similar bound $\Nc (\la;\De_{\wh \Ga}^\Nr) \lesssim \la^{(d-1)/2} $ for $\De_{\wh \Ga}^\Nr$.
This implies $n^{2/(d-1)} \lesssim \la_n^\Nr  $ for sufficiently large $n$ and, in turn, implies 
$n^{2/(d-1)} \lesssim \frac{1}{M}\la_n^\Nr  \le \mu_n $.
This completes the proof of Theorem \ref{t:W}.

\subsection{Functional spaces  on $\paD$ and the proof of Lemma \ref{l:LqMs}}
\label{s:FSp}
\label{s:LqA}

In this subsection, we skip $\paD$ in the notation $L^p (\paD)$ and $H^s (\paD)$ whenever this does not lead to ambiguities.
For an interpolation pair $\{\Yf_0,\Yf_1\}$ of Banach spaces, we denote by $[\Yf_0,\Yf_1]_\vth$ and 
$(\Yf_0,\Yf_1)_{\vth,q}$ the interpolation spaces produced by 
the complex interpolation and by 
the K-method of Peetre, resp.  (see, e.g., \cite{T78}).

In Section \ref{s:MultGen}, we used  under  assumption \eqref{a:Ck-11} the  formula
\begin{gather} \label{e:IntHs-sA}
	 H^{s(1-2\vth)}  = [H^s  ,H^{-s} ]_\vth = (H^s  ,H^{-s}  )_{\vth,2}, \quad  \vth \in [0,1], \quad  s \in [-k,k],
\end{gather}
which is known for compact  Riemann $C^{\infty}$-manifolds \cite{T92} 
and, for  $s \in [0,1]$, in the case of compact  Lipschitz boundaries \cite{GMMM11}.
In our case of compact $C^{k-1,1}$-boundary, the proof of \eqref{e:IntHs-sA} in \cite{T92} requires some changes.
Namely, the well-known  interpolation formula for a rigged Hilbert space $[H^s, H^{-s}]_{1/2} = H^0 = L^2 $
 has to be used in order to extend the interpolation to the spaces of generalized functions $H^t$ with $t<0$.
  Using this formula and the reiteration theorem for the complex interpolation, it is possible to obtain \eqref{e:IntHs-sA} from 
the corresponding results on the Triebel-Lizorkin spaces $F^t_{p,q} (\paD)$
and the equality $H^s = F^s_{2,2} (\paD)$  along the lines indicated in \cite{T92,T02}.

Consider now the general case where $\D \subset \RR^d$ is a Lipschitz domain.
The following continuous embeddings of the fractional Sobolev spaces on $\paD$ take place :
\begin{gather*} %\label{e:HsLp}
\textstyle	H^{s} (\pa \D) \imb L^p (\pa \D)  \ \text{ if $ 1 \le  p<\infty$ and $ 0<s \le 1$ satisfy  
		$\frac{1}{2} - \frac{s}{d-1} \le \frac{1}{p}$};\\
\text{$H^{s} (\pa \D) \imb C (\pa \D) $ \qquad  if \quad $d=2$ and $1/2 <s \le 1$ (see \cite{T92,T02,G11}).}
\end{gather*}
Here and below, $C(\paD)$ is the space of continuous functions.

\begin{proof}[Proof of Lemma \ref{l:LqMs}]
In the case $s_1=s_2=0$, Lemma \ref{l:LqMs} is obvious. Assume that $s_1+s_2>0$ and $s_1,s_2 \in [0,1]$.
Let $\pi_j :=\frac{1}{2} - \frac{s_j}{d-1} = (1-\vka_j)/2$ for $j=1,2$. We put $\pi_0 =1 - \pi_1 - \pi_2 = \frac{s_1+s_2}{d-1}$ and note that 
$\pi_0>0$.

\emph{Step 1.} Assume that $\pi_j>0$ for $j=1,2$. Let $p_j = 1/\pi_j$ for $j=0,1,2$.
Then $p_1, p_2 \in (2,+\infty)$ and $p_0 \in (1,+\infty)$. The Sobolev-type  embedding gives
$H^{s_j}  \imb L^{p_j} $ for $j=1,2$. 
Using the Hölder inequality (for three functions), we get 
\begin{align*}
\int_{\pa \D} |f g_1 g_2|  
\le  \| f \|_{L^{p_0} } \| g_1 \|_{L^{p_1} } \| g_2 \|_{L^{p_2} } 
\lesssim  \|f\|_{L^{p_0}} \| g_1 \|_{H^{s_1} } \| g_2 \|_{H^{s_2} }
\end{align*}
for every $f \in L^{p_0} $,
$g_1 \in H^{s_1}$, and $g_2 \in H^{s_2}$, i.e.,
$
L^{p_0}  = L^{\frac{d-1}{s_1+s_2}}  \imb \MM (H^{s_1} , H^{-s_2}  ) $.

\emph{Step 2.} Assume that $\pi_j \ge 0$ for $j=1,2$, and $\pi_1 \pi_2 =0$. We take small enough $
\ep >0$ such that $r_j = 1/(\pi_j+\ep) \in (1,+\infty)$ for $j=1,2,$ and $1-1/r_1- 1/r_2>0$. Then 
$r_0 := (1-1/r_1- 1/r_2)^{-1}$ satisfies $1 <r_0 <+\infty$.
From the  embedding $H^{s_j}  \imb L^{r_j} $ for $j=1,2$, we get like in the Step 1 that 
%\begin{align*}
$\int_{\pa \D} |f g_1 g_2| 
%\le  \| f \|_{L^{r_0} } \| g_1 \|_{L^{r_1}} \| r_2 \|_{L^{q_2} } 
\lesssim  \|f\|_{L^{r_0}} \| g_1 \|_{H^{s_1} } \| g_2 \|_{H^{s_2} }$.
%\end{align*}
%for every $f \in L^{r_0} $,
%$g_1 \in H^{s_1}$, and $g_2 \in H^{s_2}$.
Letting $\vep \to +0$, one obtains  
$
L^q \imb \MM (H^{s_1} , H^{-s_2}  )$  for every $ q > \frac{d-1}{s_1+s_2}$.
Note that, in the case under consideration, $\frac{d-1}{s_1+s_2} \ge 1$.

\emph{Step 3. Consider the case $\pi_2<0$}.  
%Since $\MM (H^{s_1} , H^{-s_2}) = \MM (H^{s_2} , H^{-s_1})$, we may assume without loss of generality that $\pi_2 <0$.
In this case, the embeddings $H^{s_2}  \imb C (\paD)\imb L^\infty $ hold true. Hence, 
\begin{align*} %\label{e:pi2<0}
\int_{\pa \D} |f g_1 g_2|  
\lesssim \| g_2 \|_{L^\infty } \int_{\pa \D} |f g_1 |   \lesssim  \| g_2 \|_{H^{s_2} }\int_{\pa \D} |f g_1 | .
\end{align*}
whenever $f g_1\in L^1 (\paD)$.

If  $\pi_1 <0$ (this case is equivalent to $\vka_1 >1 $, $\vka_2>1$),
we get from $H^{s_1}  \imb L^\infty $ and \eqref{e:MimbL1} 
that $
L^1  \imb \MM (H^{s_1} (\pa \D), H^{-s_2} (\pa \D) )  \imb L^1$, i.e., 
$ \MM (H^{s_1} (\pa \D), H^{-s_2} (\pa \D) )  = L^1.$

Assume now that $\pi_1>0$.  Let $p_1 := 1/\pi_1$ and $p'_1 := (1 - \pi_1)^{-1}$. Since $0<\pi_1\le 1/2$, we see that 
$p'_1 = \frac{2(d-1)}{d-1+2s_1} \in (1,2) $. 
Then %\eqref{e:pi2<0} and 
$H^{s_1} (\paD) \imb L^{p_1} (\paD)$ implies
\begin{align*} 
\int_{\pa \D} |f g_1 g_2|  
\lesssim  \|f\|_{L^{p'_1}} \|g_1 \|_{L^{p_1} }   \| g_2 \|_{H^{s_2} }  
\lesssim \|f\|_{L^{p'_1} } \|g_1 \|_{H^{s_1} }   \| g_2 \|_{H^{s_2} }   
\end{align*}
for all $f \in L^{p'_1} $,  $g_1 \in H^{s_1}$, and $g_2 \in H^{s_2}$.
Thus, 
$L^{p'_1}  = L^{\frac{2(d-1)}{d-1+2s_1}}   \imb \MM (H^{s_1} , H^{-s_2}  ) $.

In the case where $\pi_1=0$ and $\pi_2<0$, one can use the arguments similar to Step 2 in order to obtain 
$
 L^q   \imb \MM (H^{s_1} , H^{-s_2}  )$  for all $ q > \frac{2(d-1)}{d-1+2s_1} =1$.
This completes the proof of Lemma \ref{l:LqMs}.
\end{proof} 
	
\section{Operator theory approach to boundary conditions}
\label{a:AbsBC}

In this section we give the proofs of several operator theory results related to m-dissipative acoustic operators that were left without proofs in the preceding sections (this concerns Theorems \ref{t:DtN-GIBC} and Theorem \ref{t:RandM-dis}). 

\subsection{Accretivity and $\cross$-adjoints in generalized rigged Hilbert spaces}
\label{s:GenAcc}

Let $(\Hs_{-,+},\Hs,\Hs_{+,-})$ be the \emph{generalized rigged Hilbert space} in the sense of
 \cite[Vol. II, Appendix to IX.4, Ex.~3]{RS} (see also detailed descriptions in \cite{S21,EK22}).
 In particular, the Hilbert spaces $\Hs_{-,+}$ and $\Hs_{+,-}$ are dual to each other with respect to the sesquilinear $\Hs$-pairing $\<\cdot,\cdot\>_\Hs$
  that is obtained from the inner product $(\cdot|\cdot)_\Hs$ of the pivot Hilbert space $\Hs$ as an extension by continuity to $\Hs_{\mp,\pm} \times \Hs_{\pm,\mp}$. 
Note that a generalized rigged Hilbert space does not assume embeddings between spaces $\Hs_{-,+}$, $\Hs_{+,-}$, and $\Hs$
(an important example without any embeddings between $\Hs_{-,+}$, $\Hs_{+,-}$, and $\Hs$ appears naturally if one consider the trace spaces associated with Maxwell operators).

The notion of adjoint operator and the notions of accretive, dissipative, and selfadjoint operators have a straightforward extensions to the case an operator $T : \dom T \subseteq \Hs_{-,+} \to \Hs_{+,-}$ \cite{GMMM11,EK22}. For example, 
an operator $Z: \dom Z \subseteq \Hs_{-,+} \to \Hs_{+,-} $ is called \emph{accretive} if $\re \<Z y , y \>_{\Hs} \ge 0$ for all $y \in \dom Z$ \cite{EK22}.
%An accretive operator $Z: \dom Z \subseteq \Hs_{-,+} \to \Hs_{+,-} $ is called \emph{maximal accretive} if it has no a proper accretive extension $\wh Z: \dom \wh Z \subseteq \Hs_{-,+} \to \Hs_{+,-} $.

Let $T:\dom T \subseteq \Hs_{-,+} \to \Hs_{+,-}$ be densely defined (in $\Hs_{-,+}$).
Then the operator $T^\cross:\dom T^\cross \subseteq \Hs_{-,+} \to \Hs_{+,-}$ is defined as the adjoint to $T$ 
w.r.t. the $\Hs$-pairing $\<\cdot,\cdot\>_\Hs$. In short, we  say that  $T^\cross$ is the \emph{$\cross$-adjoint operator to $T$}.
This definition  means that the graph $\Gr T^\cross = \{\{f,T^\cross f\} : f \in \dom T^\cross\}$ 
consist of all $\{g_{-,+},g_{+,-}\} \in \Hs_{-,+} \oplus \Hs_{+,-}$ such that $\<h_{+,-} | g _{-,+}\>_\Hs  =  \<h_{-,+} | g_{+,-}\>_\Hs$ for all  $\{ h_{-,+} , h_{+,-} \} \in \Gr T$.
Consequently, the operator $T^\cross$ is closed and $(T^\cross)^\cross = \ove T$.

The notion of  $\cross$-adjoint operator can be extended also to operators acting between spaces $\Hs$ and $\Hs_{\mp,\pm}$. For example, for an operator $V \in \Lc (\Hs, \Hs_{-,+})$, there exists a unique $\cross$-adjoint operator $V^\cross \in 
\Lc (\Hs_{+,-}, \Hs)$ such that 
$\< V f,g\>_\Hs = (  f | V^\cross g)_\Hs$   for all $ f \in \Hs$ and $g \in \Hs_{+,-}$.
For a linear homeomorphism  $V \in \Hom (\Hs,\Hs_{-,+})$, one has $V^\cross \in \Hom (\Hs_{+,-}, \Hs)$ and $(V^\cross)^{-1} = (V^{-1})^\cross$.

In what follows, we fix a certain pair of mutually $\cross$-adjoint linear homeomorphisms $V \in \Hom (\Hs, \Hs_{-,+})$ and $V^\cross \in \Hom (\Hs_{+,-}, \Hs)$.
Such a pair allows one to reduce statements for an operator $Z: \dom Z \subseteq \Hs_{-,+} \to \Hs_{+,-} $ to  known statements for the operator 
$\wt Z:=V^\cross Z V$ (which acts in $\Hs$ in the sense that $\wt Z : \dom \wt Z \subseteq \Hs \to \Hs$); e.g., 
\begin{gather} 
	\text{$Z$ is accretive or densely defined if and only if 
		$V^\cross Z V$ is so;} \label{e:V*ZV} \\
\text{if $Z$ is densely defined, then  $V^\cross Z^\cross V = (V^\cross Z V)^*$, see \cite{EK22}.}
\label{e:V*Z*V}
\end{gather}

Maximal accretive densely defined operators $Z: \dom Z \subseteq \Hs_{-,+} \to \Hs_{+,-} $ serve as analogues of m-accretive operators $\wt Z : \dom \wt Z \subseteq \Hs \to \Hs$.

\begin{thm}[cf. \cite{P59,EK22,K25}] \label{t:CrM-acrGen}
	For an operator $Z:\dom Z \subseteq \Hs_{-,+} \to \Hs_{+,-}$, the following statements are equivalent:
	\begin{itemize}
		\item[(i)] $Z$ is densely defined and maximal accretive.
		\item[(ii)] $Z$ is closed and maximal accretive.
		\item[(iii)] $Z$ is accretive and there exists a densely defined accretive operator 
		$Q : \dom Q \subseteq \Hs_{-,+} \to \Hs_{+,-}$ such that $Z = Q^\cross$.
		\item[(iv)] $V^\cross Z V$ is m-accretive.
	%	\item[(v)] $Z$ is a closed densely defined operator and the operator $Z^\cross$ is maximal accretive.
	%	\item[(vi)] $(-1)V^\cross Z V$ is a generator of a (strongly continuous) semigroup.
	\end{itemize}
\end{thm}
\begin{proof}
In the case where $(\Hs_{\mp,\pm},\| \cdot \|_{\Hs_{\mp,\pm}})= (\Hs, \|\cdot\|_\Hs)$ and $I = V^\cross = V$,	this theorem follows from a combination of results of \cite{P59,Kato} on m-accretive operators (compare \cite[Section V.3.10]{Kato} with 
the following statements in \cite{P59}:  
Theorems 1.1.1 and 1.1.3, Lemma 1.1.1, the remarks before Lemma 1.1.1, and the corollary after Theorem 1.1.2).
The general case $Z:\dom Z \subseteq \Hs_{-,+} \to \Hs_{+,-}$ with an arbitrary  generalized rigged Hilbert space can be reduced to the case of 
the trivial duality using \eqref{e:V*ZV}-\eqref{e:V*Z*V}.
\end{proof}

\subsection{Boundary tuples and abstract boundary condition}
\label{s:BT}

Let  $A: \dom A \subseteq \Xf \to \Xf$ be a closed symmetric densely defined  operator  in a Hilbert space $\Xf$.
The following abstract integration by parts helps to define abstract  boundary conditions associated with restrictions of $A^*$.

\begin{defn}[\cite{EK22}] \label{d:MBT}
	Assume that $(\Hs_{-,+},\Hs,\Hs_{+,-})$ is a generalized rigged Hilbert space. 
	Then $(\Hs_{-,+},\Hs, \Hs_{+,-}, \Ga_0,\Ga_1)$ is called an \emph{m-boundary tuple}  for $\Ao^*$
	if  the map 
	$\Ga : f \mapsto \{ \Ga_0 f , \Ga_1 f \} $ is a surjective linear operator from $\dom \Ao^*$ onto $\Hs_{-,+} \oplus \Hs_{+,-}$ and 
	\begin{equation*} %\label{e:AIP}
		\text{$(\Ao^*f |g)_{\Xf} - (f |\Ao^*g)_{\Xf} =  \< \Ga_1 f , \Ga_0 g \>_\Hs  - \<\Ga_0 f , \Ga_1 g \>_\Hs $ for all $f,g \in \dom \Ao^*$.}
	\end{equation*}
\end{defn}

%Note that  $\< \Ga_1 f , \Ga_0 g \>_\Hs$  in \eqref{e:AIP} is understood in the sense of the sesquilinear pairing 
%$\< h_{+,-}  , h_{-,+} \>_{\Hs}$ of $\Hs_{+,-}$ and $\Hs_{-,+}$ that is generated by the inner product $(\cdot|\cdot)_\Hs$ of the pivot space $\Hs$. 
%We use the same notation 
%$\< \cdot  , \cdot \>_{\Hs}$ for the sesquilinear pairing of   $\Hs_{-,+}$ with $\Hs_{+,-}$ used in the term $\<\Ga_0 f , \Ga_1 g \>_\Hs$
%of  \eqref{e:AIP}.

In the case of the trivial duality $(\Hs, \| \cdot\|)=(\Hs_{\mp,\pm}, \|\cdot\|_{\Hs_{\mp,\pm}})$, an m-boundary tuple becomes effectively a boundary triple $(\Hs,\Ga_0,\Ga_1)$ in the sense of Kochubei, see \cite{K75} and also see  \cite{P12,BHdS20,DHM22,EK22,DM26} for modern expositions of the theory of boundary tuples.

\begin{ex} \label{ex:acute}
	In the settings of Section \ref{s:Ac1}, let us  put $\wh \ga_0  (\{\vbf,p\}) := \ga_0 (p) $ and $\wh \gan (\{\vbf,p\}) := \gan (\vbf)$. Then
	$\Tc :=(H^{1/2} (\pa \D), L^2 (\pa \D) , H^{-1/2} (\pa \D),  \wh \ga_0, \ (-\ii) \wh \gan )$  
%and $ \Tc_* :=(H^{-1/2} (\pa \D), L^2 (\pa \D) , H^{1/2} (\pa \D), 
% \wh \gan, (- \ii)\wh \ga_0  )$ 
	is an  m-boundary tuple for the acoustic operator $\A_\Max$ \cite{K26}. 
\end{ex}

For an abstract closed densely defined symmetric operator $A$ in a Hilbert space  $\Xf$, let us assume that $(\Hs_{-,+},\Hs, \Hs_{+,-}, \Ga_0,\Ga_1)$ is an m-boundary tuple for $A^*$.
Then an \emph{abstract impedance boundary condition} \cite{EK22} takes the form 
%\begin{gather*} %\label{e:IBCAbs}
$	Z \Ga_0 y = \ii \Ga_1 y  $,
%\end{gather*}
where $Z:\dom Z \subseteq \Hs_{-,+} \to \Hs_{+,-}$.
Let the operator $A_Z $ be the restriction  of $A^*$ to 
the set $\dom A_Z$ consisting of all $y \in \dom A^*$ such that $\Ga_0 y \in \dom Z$ and $Z \Ga_0 y = \ii \Ga_1 y$.

\begin{prop}[\cite{EK22}, see also \cite{K26,K25}] \label{p:CrMDis}
(i) The operator $A_Z$ is m-dissipative in $\Xf$ if and only if $Z$ is closed and maximal accretive.
\item[(ii)] The operator $A_Z$ is selfadjoint if and only if $Z^\cross = - Z$.
\end{prop}

This proposition is implicitly present in  \cite[Sections 6.2 and Section 7]{EK22} (compare in \cite{EK22} Corollaries 6.4 and 7.2 with Remarks 6.1-6.3). In the explicit form, it can be found in \cite[Proposition 2.1]{K25} (for the particular case of  boundary triples, its analogue is well-known \cite{BHdS20,DM26}).

\subsection{Proofs of Theorems \ref{t:DtN-GIBC} and \ref{t:RandM-dis}}
\label{s:ProofA}

\begin{proof}[Proof of Theorem \ref{t:DtN-GIBC}]
We use  the homeomorphisms 
\begin{gather*}
V=(\De_\paD+I)^{-1/4}_{H^0 \shortto H^{1/2}} \in \Hom (L^2 (\paD),H^{-1/2} (\paD)), \\
V^\cross = (\De_\paD+I)^{-1/4}_{H^{-1/2} \shortto H^0} \in \Hom (H^{-1/2} (\paD),L^2 (\paD))
\end{gather*}
and the m-boundary tuple $\Tc$ of Example \ref{ex:acute}.
Then Proposition \ref{p:CrMDis} and Theorem \ref{t:CrM-acrGen} imply Theorem \ref{t:DtN-GIBC}.
% The equivalence $\A_\Zc = \A_\Zc^* \Lra \Zc = - \Zc^\cross$ and the equivalence (i)	$\Lra$ (iv) of Theorem \ref{t:DtN-GIBC} are part of \cite[Theorem 2.3]{K26}.
%The equivalences (ii)	$\Lra$ (iii) $\Lra$ (iv) of Theorem \ref{t:DtN-GIBC} follow from Theorem \ref{t:CrM-acrGen}.
\end{proof}

For the proof of Theorem \ref{t:RandM-dis}, we need the following technical tool.

\begin{prop}[Hanš lemma, e.g., \cite{H61}] \label{p:HTh}
	Let $T:\Om \to \Lc (\Xf)$ be a random bounded operator such that a.s. $T \in \Hom (\Xf)$. Then $T^{-1}$ is a random bounded  operator. 
\end{prop}

\begin{proof}[Proof of Theorem \ref{t:RandM-dis}.]
	Let $V$ and $V^\cross$ be as in the proof of Theorem \ref{t:DtN-GIBC}.
For the acoustic system, a general parametrization of m-dissipative boundary conditions by contractions $K \in \Lc (L^2(\paD))$ 
 is given in \cite{K26} in the form 
 \begin{gather} \label{e:mdisBC}
	(K+I) V^{-1} \ga_0 (p)  + (K-I) V^\cross \gan (\vbf) = 0 .
\end{gather}
The acoustic operator $\wh \A_K$ associated with \eqref{e:mdisBC} is m-dissipative if and only if $K$ is a contraction. Moreover, the corresponding contraction $K$ is unique.
 
By \cite[Theorem 2.8]{K26}, $\wh \A_K$ is a random m-dissipative operator (in the resolvent sense) if $K$ is a random contraction (i.e., if $K$ is a random bounded operator such that 
a.s. $\|K\|\le 1$).
In order to prove Theorem \ref{t:RandM-dis}, it is enough to prove that, for a random bounded operator  $\Zc:\Om \to \Lc (H^{1/2},H^{-1/2})$ such that $\Zc$ is a.s. accretive, the boundary condition 
$\Zc  \ga_0 (p) =  \gan (\vbf)  $ can be reformulated in the form \eqref{e:mdisBC} with a random contraction $K$.

The passage from $\Zc  \ga_0 (p) =  \gan (\vbf)  $ to \eqref{e:mdisBC} 
can be done in the way described (in the deterministic case) in \cite{K25,K26}. Proposition \ref{p:HTh} ensures that this process preserves appropriate randomness of the involved operators. Namely, we first pass to the operator $\wt Z:=V^\cross Z V$, which is a random bounded operator in $L^2 (\paD)$ and is a.s. accretive. Hence, with probability 1, $\wt Z$ is m-accretive  and $\si (\wt Z) \subseteq \oCC_\rr$. In particular,  a.s.  $\wt \Zc +I \in \Hom (L^2 (\paD))$.  By Proposition \ref{p:HTh},
$(\wt \Zc +I)^{-1}$ is a random bounded operator.
Since a product of random bounded operators is random bounded (see, e.g., \cite{S84}), the Cayley transform $\Cc_{\wt \Zc} = (\wt \Zc- I)( \wt \Zc+I)^{-1}$ of $\wt \Zc$ is also a random bounded operator. 
A Cayley transform of an m-accretive operator is a contraction.
One checks by direct calculation  that with $K=\Cc_{\wt \Zc}$ the boundary condition \eqref{e:mdisBC} becomes equivalent to $\Zc  \ga_0 (p) =  \gan (\vbf)  $. 
% of Theorem \ref{t:RandM-dis}. 
\end{proof}
\end{appendices}

%\section*{Statements and Declarations}
%\begin{itemize}
%\item The author declare that he has no competing interests.
%\item All data generated or analyzed during this study are included in this article.
%\end{itemize}

\end{document}